\documentclass[11pt,a4paper]{article}

\usepackage{dsfont}
\usepackage[utf8]{inputenc}
\usepackage[T1]{fontenc}
\usepackage[english]{babel}
\usepackage{amsmath}
\usepackage{amssymb}
\usepackage{amsthm}
\usepackage{sansmath}
\usepackage{calligra}
\usepackage{mathtools}
\usepackage[titletoc,toc,title]{appendix}

\usepackage{ae}
\usepackage{icomma}
\usepackage{units}
\usepackage{color}
\usepackage{graphicx}
\usepackage{caption}
\usepackage{subcaption}
\usepackage{bbm}
\usepackage[square, numbers, sort]{natbib}
\usepackage{multirow}
\usepackage{array}
\usepackage{geometry}
\usepackage{fancyhdr}
\usepackage{fncychap}
\usepackage[hyphens]{url}
\usepackage[pdfpagelabels=false]{hyperref}
\usepackage{lettrine}
\usepackage{theoremref}
\usepackage[]{todonotes}

\newcommand{\Line}{\mathcal{L}} 

\newcommand{\vol}{\mathrm{Vol}}
\newcommand{\reg}{\mathrm{reg}}

\newcommand{\Pm}{\mathbb{P}}

\newcommand{\E}{ \mathbb{E}}

\newcommand{\Q}{\mathbb{Q}}
\newcommand{\R}{\mathbb{R}}

\newcommand{\V}{\mathcal{V}}

\newcommand{\id}{\mathrm{d}}

\newcommand{\e}{\mathrm{e}}

\newcommand{\supp}{\mathrm{supp}\,}

\newcommand{\dist}[2]{\|#1-#2\|}
\newcommand{\diam}{\mathrm{diam}}

\newcommand{\ba}{a}
\newcommand{\bp}{p}
\newcommand{\auk}{a^{(k)}}
\newcommand{\adk}{a_{(k)}}
\newcommand{\puk}{p^{(k)}}
\newcommand{\pdk}{p_{(k)}}

\newcommand{\haara}{\nu_d}
\newcommand{\cent}{\mathrm{cent}}

\newcommand{\indicator}{{\rm I}}
\newcommand{\uhs}{\partial B(o,1)_+}

\newcommand{\BE}{{\mathbb{E}}}

\newcommand{\BP}{{\mathbb{P}}}

\newcommand{\BR}{{\mathbb{R}}}

\newcommand{\BZ}{{\mathbb{Z}}}

\newcommand{\FE}{{\mathfrak{E}}}

\newcommand{\CA}{{\mathcal{A}}}

\newcommand{\CC}{{\mathcal{C}}}

\newcommand{\CH}{{\mathcal{H}}}
\newcommand{\CI}{{\mathcal{I}}}

\newcommand{\CL}{{\mathcal{L}}}
\newcommand{\CM}{{\mathcal{M}}}

\newcommand{\CP}{{\mathcal{P}}}

\newcommand{\CR}{{\mathcal{R}}}
\newcommand{\CS}{{\mathcal{S}}}

\newcommand{\CV}{{\mathcal{V}}}

\newcommand{\Fc}{{\mathfrak{c}}}

\newtheorem{theorem}{Theorem}[section]

\newtheorem{proposition}[theorem]{Proposition}
\newtheorem{lemma}[theorem]{Lemma}

\newtheorem{corollary}[theorem]{Corollary}

\theoremstyle{remark}
\newtheorem{remark}[theorem]{Remark}

\begin{document}

\pagenumbering{arabic}
\title{\bf{The fractal cylinder process: existence and connectivity phase transition} \rm}
\author{Erik Broman\footnote{Department of Mathematics, Chalmers University of Technology and Gothenburg University, Sweden. E-mail: broman@chalmers.se. Research supported by the Swedish research council.} \and Olof Elias\footnote{Department of Mathematics, Chalmers University of Technology and Gothenburg University, Sweden. E-mail: olofel@chalmers.se.} \and Filipe Mussini\footnote{Department of Mathematics, Uppsala University, Sweden. E-mail: filipe.mussini@math.uu.se} \and Johan Tykesson\footnote{Department of Mathematics, Chalmers University of Technology and Gothenburg University, Sweden. E-mail: johant@chalmers.se. Research supported by the Swedish research council.}}
\date{\today}

\maketitle
\thispagestyle{empty}

\begin{abstract}
We consider a semi-scale invariant version of the Poisson cylinder model 
which in a natural way induces a random fractal set. We show that this
random fractal exhibits an existence phase transition for any dimension 
$d\geq 2,$ and a connectivity phase transition whenever $d\geq 4.$ 
We determine the exact value of the critical point of the existence phase 
transition, and we show that the fractal set is almost surely empty 
at this critical point. 

A key ingredient when analysing the connectivity phase transition is 
to consider a restriction of the full process onto a 
subspace. We show that this restriction results in a fractal ellipsoid
model which we describe in detail, as it is key to obtaining 
our main results. 

In addition we also determine the almost sure Hausdorff dimension of 
the fractal set.
\end{abstract}
\pagestyle{fancy}
\setlength{\headheight}{14pt} 
\fancyhf{}
\rhead{Broman, Elias, Mussini, Tykesson}
\lhead{The fractal cylinder process}
\cfoot{\thepage}

%
%
%
%
%
%
%
%
%
%
%

\section{Introduction and statement of results}
This paper introduces a random fractal model which we call
the fractal cylinder model. Informally speaking, the fractal cylinder model 
is a scale invariant version of the so-called Poisson cylinder model.
Percolation aspects of the Poisson cylinder model was first studied in \cite{TW_2012} where it was shown that
the vacant set undergoes a non-trivial percolative phase transition in 
dimensions $d\geq 4,$ while the analogous result for when $d=3$ was 
established in \cite{HST_2015}. Later, it was established in 
\cite{BT_2016} that the occupied set does not undergo a similar
phase transition in any dimension. In contrast, it was shown that 
for any $d\geq 2$ any two cylinders will be connected by using 
at most $d-2$ other cylinders of the process. The Poisson cylinder model was previously studied in connection with stochastic geometry, see for example \cite{SPIESS_2012}.

The fractal version that we study here is a natural generalization of 
classical models such as the Mandelbrot fractal percolation 
model (see \cite{CCD_1988}) and the class of
Poissonian random fractal models (see \cite{BC_2010})
generated by bounded subsets of $\BR^d$. 
A key difference between the previously mentioned models and the 
fractal cylinder model is that the random objects generating the fractal 
are here unbounded. This introduces infinite range dependencies
which has been absent in all previous random fractal models.
In the non-fractal case, many such extensions of classical models have 
been recently studied. Such extensions include the Poisson cylinder model 
mentioned above, but also the so-called random interlacement model 
(introduced in \cite{S_2010}) and the Brownian interlacement model 
(introduced in \cite{S_2013}).

There are many natural questions to ask about random fractals. In this 
paper we focus on the study of phase transitions and Hausdorff dimensions. 
In order to state our main results we first need to give an informal 
explanation of our model (see further Section \ref{sec:modelsanddef} 
where we give a formal definition with details). 

Let $\mathrm{A}(d,1)$ denote the space of lines in $\BR^d$, and let $\nu_d$ 
be the unique measure (up to scaling) on $\mathrm{A}(d,1)$ which is 
invariant under the isometries of $\R^d$. Then, consider the space 
$\mathrm{A}(d,1) \times (0,1]$ and let 
\[
\omega = \sum_{i\geq1} \delta_{(L_i,r_i)}
\] 
be a locally finite Poisson point process with intensity measure  
\[
\lambda \nu_d \times \indicator ( 0 < r \leq  1 ) r^{-d} \id r, \ \lambda>0,
\]
and let $\Pm_\lambda$ denote the corresponding law. Often, we will suppress $\lambda$ from the notation and write $\Pm$ instead of $\Pm_{\lambda}$. Here 
$\delta_{(L,r)}$ denotes point measure at 
$(L,r)\in \mathrm{A}(d,1) \times (0,1]$. We let 
\[
\V=\V(\omega) = \R^d \setminus \bigcup_{ (L,r) \in \omega } L+ B(o,r)
\]
denote the vacant set of the fractal cylinder process.

As the parameter $\lambda>0$ varies, the random fractal model exhibits 
several phase transitions. The first that we shall consider is between 
the empty phase (i.e. where $\CV=\emptyset$ a.s.) and the non-empty 
phase where $\BP(\CV\neq \emptyset)>0$. The critical value corresponding
to this phase transition is denoted by $\lambda_e$ and is defined by 
\begin{equation} \label{eqn:lambdaedef}
\lambda_e := \inf \left\{ \lambda >0 : \Pm_\lambda \left(  \V = \emptyset \right)=1 \right\}.
\end{equation}
We refer to this phase transition as the existence phase transition. 
We observe that one can also consider similar phase transitions of the 
model restricted to subspaces. Indeed, if we let $H_k:=\BR^k \times \{0\}^{d-k}$ and
we define
\[
\lambda_e(d,k):= \inf \left\{ \lambda >0 : \Pm_\lambda 
\left(  \V \cap H_k = \emptyset \right)=1 \right\},
\]
then our main result concerning $\lambda_e(d,k)$
is as follows.
\begin{theorem} \label{thm:lambdae}
For any $d\geq 2$ and $k \in\{1,\ldots,d\}$ we have that 
$\lambda_e(d,k)=k$ and 
\[
\Pm_{\lambda_e(d,k)} (  \V\cap H_k = \emptyset) = 1.
\]
\end{theorem}
\noindent

\begin{remark} Theorem \ref{thm:lambdae} determines the value of 
$\lambda_e(d,k)$. Furthermore, it states that at the critical 
value of this phase transition, the model is in the empty phase.

The corresponding result has been established for a general class
of models in \cite{B_2018}. However, in that paper
the objects generating the random fractal had a diameter of at most one,
while here the cylinders are obviously unbounded. Historically, the 
type of result covered by Theorem \ref{thm:lambdae} was considered first by 
Shepp (see \cite{S_1972}) in 1972.
\end{remark}

The second phase transition that we will study is the so-called connectivity 
phase transition. The critical value of this phase transition is defined
by letting
\begin{equation}\label{eqn:connectivityparameter}
\lambda_c := \inf \left\{  \lambda>0 : \Pm_\lambda \left( \V \text{ is totally disconnected} \right)>0 \right\}.
\end{equation}
Some of the earliest results related to the study of this phase 
transition obtained in \cite{CCD_1988} and 
\cite{M_1992} where the so-called Mandelbrot 
fractal percolation model was studied. Among other results, it 
was proven that the critical parameter value was non-trivial.
Later, this phase transition was studied for general Poissonian 
random fractal models
with bounded generating sets in \cite{BC_2010}. There, 
the main result was that at the critical threshold, the models were 
in the connected phase. However, the fact that we are working with 
unbounded generating objects makes the study of this phase transition 
much more complicated.  Our main result concerning the connectivity 
phase transition is the following theorem.
\begin{theorem}\label{thm:connectivitytranstition}
For $d \geq 4,$ we have that $\lambda_c \in(0,\infty).$ When $d=3,$
then for any $\lambda>0,$ $\CV \cap H_2$ almost surely does not contain any 
connected components while for $d=2,$ we have that $\lambda_c=0$.
\end{theorem}
\begin{remark}
Obviously, Theorem \ref{thm:connectivitytranstition} is incomplete
in that we do not establish the full result when $d=3.$

There is a difference between a set being topologically connected
and being path connected. In this paper, whenever we refer to a set 
being connected, we mean in the topological sense. However, we mention the 
paper \cite{M_1992} where it is proven that for the Mandelbrot 
fractal percolation model, the set is path connected whenever it is 
topologically connected. It is an open question whether this also holds
for our model.
\end{remark}

The intersection of the non-fractal and the fractal 
cylinder process with $H_k$ both induces a collection of random 
sets in $\BR^k.$ These induced processes are then used in order 
to obtain a proof of Theorem \ref{thm:connectivitytranstition}. We anticipate
that these induced processes will be useful for further studies of 
the fractal cylinder process, and therefore we chose to state our 
main results concerning them here. The result concerns both the 
fractal cylinder process (described above) as well as the 
regular (i.e. non-fractal) cylinder process. This process 
is a random collection of cylinders generated by first picking 
lines $L\in {\mathrm A}(d,1)$ using $\lambda \nu_d$ as the intensity measure, 
and then placing a cylinder $\Fc(L,r)$ of radius $r$ around each such
line. Here we let $\FE^k_o$ denote the set of ellipsoids centred at
the origin $o,$ and we let $\ell_k$ denote $k$-dimensional 
normalized Hausdorff measure (see further Section \ref{sec:dimensions}).



\begin{theorem} \label{thm:cylellips}
Consider the Poisson cylinder model in $\R^d$ with cylinder radius $r$. 
The restriction of this cylinder process to the subspace $H_k$ 
where $k\in\{1,\ldots,d-1\}$ is a Poisson process of
ellipsoids with intensity measure 
\begin{equation}\label{eqn:measdecomp1}
\lambda \ell_k \times \xi_{k,r}.
\end{equation}
Here, $\xi_{k,r}$ is a measure on $\FE_o^k$ given by \eqref{eqn:xikrdef}.
Furthermore, the restriction of the fractal cylinder process to $H_k$ 
is a Poisson process with intensity measure  
\begin{equation}\label{eqn:measdecomp2}
\lambda \ell_k \times \xi_{k},
\end{equation}
where
\[
\xi_k(\cdot)=\int_0^1 \xi_{k,r}(\cdot) r^{-d}\id r.
\]
\end{theorem}

\begin{remark} 
Informally, Theorem \ref{thm:cylellips} states that the 
induced model \eqref{eqn:measdecomp1} can be described in the following 
way. The centres of the shapes
are picked according to a Poisson process with 
Lebesgue measure on $\BR^k$ as the 
intensity measure, while the actual shapes are then
given by the measure $\xi_{k,r}$ which is supported on ellipsoids. 
Therefore, Theorem \ref{thm:cylellips} determines the intuitive fact that 
the induced model is indeed an ellipsoid model, and it also gives 
an explicit description of this model. This description (i.e. 
\ref{eqn:xikrdef}) is postponed until Section \ref{sec:cylintersect}, as 
defining it here would require too much space. Apart from its intrinsic
value, Theorem \ref{thm:cylellips} will also be used when proving Theorem 
\ref{thm:connectivitytranstition}.

It should be noted that although we write $\xi_{k,r}$, it 
also depends on the dimension $d.$ The reason for not writing e.g. 
$\xi_{d,k,r}$ 
is that we think of $d$ as being fixed, and so adding 
it to the notation throughout the paper would be unnecessarily
cumbersome. See also the comment just below \eqref{eqn:xikrdef}.
\end{remark}


The last of our main result concerns the Hausdorff dimension 
(see Section \ref{sec:dimensions} for a short overview)
of the random fractal set $\CV.$ 
\begin{theorem} \label{thm:Hausdorff}
For any $\lambda<k$ we have that 
\[
\BP(\dim_{\mathcal{H}}( \V \cap H_k) = k-\lambda)=1.
\]
\end{theorem}
\noindent

\begin{remark}
As we will see, there is some overlap between 
Theorem \ref{thm:lambdae} and Theorem \ref{thm:Hausdorff}.
This is discussed in further details in Remark~\ref{rem:vkuppdim}.
\end{remark}

The rest of the paper is structured as follows. In Section 
\ref{sec:modelsanddef} we establish notation and define the models
that we study in this paper. In Section \ref{sec:existence},
we study the existence phase transition and prove Theorem \ref{thm:lambdae}
while in Section \ref{sec:HaussofV} we establish the almost sure
Hausdorff dimension of the set $\CV$ by proving Theorem \ref{thm:Hausdorff}.
It turns out that while the standard representation (see Section 
\ref{sec:lincyl}) of the invariant line process is very useful for 
the study of the cylinder fractal model in dimension $d,$ it is not 
suitable for the study of the restriction of said process onto $H_k.$
This is why we in Section \ref{sec:intensitymeasure} find an alternative 
representation which will be useful to us. This representation is 
then used in Section \ref{sec:cylintersect} to establish 
Theorem \ref{thm:cylellips}. 
The only main result that is left to prove deals with the connectivity 
phase transition. For this, we will use Theorem \ref{thm:cylellips}
to couple the induced ellipsoid model with a more standard model, 
namely the so-called fractal ball model (see for instance \cite{MR_1996}, \cite{BC_2010} and \cite{BJT_2017}). However, in order to do 
this we will need to carefully analyse certain statistics of the 
induced ellipsoid model, and this is done in Section 
\ref{sec:anaofellips}. Finally, the coupling with the fractal 
ball model is performed in Section \ref{sec:connectivitydomination},
and this will be used to prove Theorem \ref{thm:connectivitytranstition}
when $d\geq 4.$

\section{Models and definitions}\label{sec:modelsanddef}
In this section, we define the models we study in this paper, 
and in addition, we introduce much of the notation we shall use later.

\subsection{Hausdorff measure and Fractal dimensions} \label{sec:dimensions}

In this subsection we will briefly discuss the concept of Hausdorff 
measures and Hausdorff dimension, see \cite{F_2014} for further details. 

For $F \subset \BR^d,$ we define the $s$-dimensional Hausdorff 
measure of $F$ to be 
\[
\CH^s(F):=\lim_{\delta \to 0} \inf\left\{\sum_{i=1}^\infty \diam(U_i)^s:
\{U_i\}_{i \geq 1} \textrm{ is a } \delta\textrm{-cover of } F\right\},
\]
where $\{U_i\}_{i \geq 1}$ is a $\delta$-cover of $F$ if 
$\diam(U_i)\leq \delta$ for every $i\geq 1$ and 
$F\subset \bigcup_{i=1}^\infty U_i.$ We then define the normalized 
Hausdorff measure $\ell_k(\cdot):=c\CH^k(\cdot)$ where the constant $c$ 
is chosen so that $\ell_k([0,1]^k \times \{0\}^{d-k})=1$. When constructed
is this way, the measure $\ell_k$ is a measure on $\BR^d.$ However, 
it will be convenient to not reference this fact. Thus we will write 
$\ell_k([0,1]^k)$ or $\ell_{d-1}(S)$ if $S$ is a $d-1$-dimensional
surface in $\BR^d$ etc.

Next, the Hausdorff dimension of the set $F$ is defined to be 
\[
\dim_\CH(F):=\inf\{s>0:\CH^s(F)=0\}=\sup\{s>0:\CH^s(F)=\infty\}.
\]

\subsection{General notation} \label{sec:gennotation}
Throughout, if $A\subset {\mathbb R}^d$, we let $A^r$ be the closed 
$r$-neighbourhood of $A$. Furthermore, the $L^2$-norm on $\R^d$ is denoted 
by $\|\cdot \|$ so that 
$B(x,r) = \left\{  y \in \R^d : \| y - x \| \leq r  \right\}$ 
is the closed ball centred at $x$ and with radius $r.$
In a few places it will be important to emphasize the dimension $d,$ and
in those places we shall write $B^d(x,r)$ in place of $B(x,r).$
Recall that for $k\in \{1,\ldots, d-1\}$ we define
$H_k$ to be $\R^k\times \{0\}^{d-k}$. With a slight abuse of notation, we 
will routinely identify $H_k$ with $\R^k$. We let $e_1,\ldots,e_d$ denote the standard basis of $\R^d$.

The unit sphere in $\BR^d$ is the boundary of $B(o,1)$, and 
we denote this by $\partial B(o,1).$
The following expression will surface often, and so we let
\begin{equation} \label{eqn:psidef}
\psi_{d-1}:=\ell_{d-1}(\partial B^d(o,1))=\frac{2\pi^{d/2}}{\Gamma(d/2)}.
\end{equation}
We note for future reference that 
\begin{equation} \label{eqn:ballvolume}
\ell_d(B^d(o,1))=\frac{\pi^{d/2}}{\Gamma(1+d/2)}=\frac{\psi_{d-1}}{d}
\end{equation}
and that 
\begin{equation}\label{eqn:psirecform}
\psi_{l+1}=\frac{2\pi}{l}\psi_{l-1}.
\end{equation}

\medskip

We use the following convention for constants. With $c$ and $c'$ we denote 
strictly positive constants which might only depend on the dimensions $d$ 
and $k\le d$. Its value might change from place to place. If a constant 
depend on other quantities than $d$ or $k$, this will be indicated. 
For example, $c(\lambda)$ stands for a constant depending on $k,d$ 
and $\lambda$. Numbered constants $C_1,\ldots ,C_3$ and $c_1,\ldots c_{10}$ are defined where they first appear and keep their values throughout the paper.

\subsection{Lines and cylinders}\label{sec:lincyl}
Let ${\mathrm A}(d,1)$ be the set of all bi-infinite lines in $\R^d$, 
and let ${\mathrm G}(d,1)$ be the set of of all bi-infinite lines in 
$\R^d$ containing the origin. That is, ${\mathrm A}(d,1)$ is the set of 
all $1$-dimensional affine subspaces of $\R^d$, while ${\mathrm G}(d,1)$ 
is the set of all $1$-dimensional linear subspaces of $\R^d$. 
For any measurable set $A \subset \R^d$ we let $\Line_A$ denote the set of 
lines that intersects $A,$ that is 
$\Line_A=\{L\in {\mathrm A}(d,1):L\cap A \neq \emptyset\}.$ 
For $A,B\subset \R^d$, let $\Line_{A,B} = \Line_A \cap \Line_B$ be 
the set of lines intersecting both $A$ and $B$.
On ${\mathrm A}(d,1)$ there is a unique (up to constants) Haar measue, which 
we denote by $\haara$. Furthermore, we shall assume that $\haara$ is normalized 
so that $\haara(\Line_{B(o,1)})=1.$ 
We can identify a line $L\in {\mathrm G}(d,1)$ with the unique point 
on the upper hemisphere (i.e. where the first coordinate is positive) 
$\partial B^+(o,1)$ which $L$ intersects. With this in mind we see that the
uniform measure on ${\mathrm G}(d,1)$ is simply a multiple of the surface measure on $\partial B(o,1)$, i.e.
\begin{equation}\label{eqn:gduniform}
\frac{2\id \ell_{d-1}( L)}{\psi_{d-1}}.
\end{equation}
The following is a standard representation of the measure $\nu_d$ (see
\cite{SW_2008} Theorem 13.2.12 p.588).  
For any measurable $A \subset \R^d$ we have that 
\begin{equation} \label{eqn:nudformula}
\haara(\Line_A)=\frac{4 \pi}{\psi_{d}\psi_{d-1}}
\int_{{\mathrm G}(d,1)}\int_{L^\perp}
\indicator (L+y\in \Line_A)\id \ell_{d-1}( y)\id \ell_{d-1}( L),
\end{equation}
where here and in the future, $\indicator (\cdot)$ denotes an indicator 
function. 

Note that if $A=B^d(o,1)$, then for any fixed $L,$ 
\[
\int_{L^\perp}
\indicator(L+y\in \Line_{B^d(o,1)})\id\ell_{d-1}(y)
=\ell_{d-1}(B^{d-1}(o,1))=\frac{\psi_{d-2}}{d-1}=\frac{\psi_d}{2 \pi},
\]
where we used \eqref{eqn:psirecform}. Combining this with~\eqref{eqn:gduniform} and~\eqref{eqn:nudformula}, we see that
$\nu_d(\CL_{B(o,1)})=1$ as desired.

For $L\in {\mathrm A}(d,1)$ and $r>0$, we let $\mathfrak{c}(L,r)$ denote the open 
cylinder of base-radius $r$ centred at $L$:
$$
\mathfrak{c}(L,r)=\{x\in \R^d\,:\,\|x-L\|< r\}.
$$

We now state two basic result that we shall make frequent use of 
in the rest of the paper. The proof of the following lemma is an 
elementary exercise using \eqref{eqn:nudformula} and can be found in 
\cite{BM_2017}.

\begin{lemma}\label{lem:basicmeasure}
Let $r  > 0$ and $A \subset \BR^d$ be a measurable set.
We have that  
\mbox{}
 \begin{enumerate}
  \item[a)] for any $c> 0,$ $\nu_d\left(\CL_{cA}\right) 
  = c^{d-1}\nu_d\left(\CL_A\right)$,
  \item[b)] $\nu_d\left(\CL_{B(x,r)}\right)
   = r^{d-1}$,
  \item[c)] $\nu_d\left(\CL_{B(x,r),B(y,r)}\right)
  = r^{d-1}\nu_d\left(\CL_{B(x',1),B(y',1)}\right)$, 
 \end{enumerate}
where $(x',y')$ is any pair of points such that 
$\dist{x}{y} = r \dist{x'}{y'}$.
\end{lemma}


The next Lemma gives us estimates for the measure of the
set of lines intersecting two distant balls. The proof can be found in 
\cite{TW_2012}.
\begin{lemma}\label{lem:measurebounds_org}
Let $x_1,x_2\in \BR^d$. Then 
there exists constants $c_1$ and $c_2$ depending only on $d$ such that
\[
\frac{c_1}{\dist{x_1}{x_2}^{d-1}} 
\leq \nu_d\left(\CL_{B(x_1, 1),B(x_2, 1)}\right) \leq 
\frac{c_2}{\dist{x_1}{x_2}^{d-1}} ,
\]
for every pair $x_1, x_2$ such that $\dist{x_1}{x_2} \geq 4.$
\end{lemma}

Next, we describe the parametrization of lines which we will use in this paper. 
Given $(\ba,\bp)\in (\R^{d-1} \times \{1\})\times (\R^{d-1}\times\{0\})$, we 
let $L(\ba,\bp):=\{\ba t +\bp\,:\,t\in \R\}$. Observe that $L(\cdot,\cdot)$ 
is a bijection between $(\R^{d-1} \times \{1\})\times (\R^{d-1}\times\{0\})$ 
and $\tilde{{\mathrm A}}(d,1)$, where $\tilde{{\mathrm A}}(d,1)$ is the set 
of lines in ${\mathrm A}(d,1)$ not parallel to $H_{d-1}$. Since $\nu_d({\mathrm A}(d,1)\setminus \tilde{{\mathrm A}}(d,1))=0$, we can disregard lines in ${\mathrm A}(d,1)\setminus \tilde{{\mathrm A}}(d,1)$ and therefore the discrepancy between ${\mathrm A}(d,1)$ and 
$\tilde{{\mathrm A}}(d,1)$ will hereafter be ignored. Note that the line $L(\ba,\bp)$ intersects $H_{d-1}$ at the point $\bp$, 
and the vector $\ba$ describes the direction of the line. 

For reasons that will be clear later, we will often consider different parts of the vectors $\ba$ and $\bp$ separately. First, we write $\ba=(a_1,\ldots,a_{d-1},1)$ and $\bp=(p_1,\ldots,p_{d-1},0)$. For $k\in \{1,\ldots,d-1\}$, we then let $\adk:=(a_1,\ldots\,a_k)$, $\auk:=(a_{k+1},\ldots,a_{d-1})$, $\pdk:=(p_1,\ldots\,p_k)$ and $\puk:=(p_{k+1},\ldots,p_{d-1})$. With a slight abuse of notation, we then have $\ba=(\adk,\auk,1)$ and $\bp=(\pdk,\puk,0)$. Note that if $k=d-1$, then both $\auk$ and $\puk$ are empty. Depending on the situation, we will write $L$, $L(\ba,\bp)$ or $L(\adk,\auk,\pdk,\puk)$.

Using the above described parametrization of lines, $\nu_d$ can 
be represented as follows. For any measurable $A \subset \R^d,$ 
\begin{equation}\label{eqn:linmeasrep}
\nu_d(\Line_A)=\Upsilon_d 
\int_{\{(\ba,\bp)\,:\,L(\ba,\bp)\in A\}} 
\frac{1}{ \|a\|^{d+1} } \id \adk \id \auk \id \pdk \id \puk
\end{equation}
where 
\begin{equation} \label{eqn:defupsilond}
\Upsilon_d=\frac{4 \pi}{\psi_d \psi_{d-1}}.
\end{equation}
The representation~\eqref{eqn:linmeasrep} is found on p.211 
in~\cite{S_1976} for the case $d=3$. Since we could not locate a 
reference for the case of general $d\ge 3$, we provide a proof in 
Section~\ref{sec:intensitymeasure}.

We note that \eqref{eqn:nudformula} is useful when performing 
calculations in the full space $\BR^d.$ For example, it is
used when proving results such as Lemmas \ref{lem:basicmeasure} 
and \ref{lem:measurebounds_org}. However, in the latter part of this paper we will 
focus on restrictions of the cylinder process to subspaces $H_k,$
and for this, \eqref{eqn:linmeasrep} will be more 
suitable. Indeed, it is essential when proving Theorem \ref{thm:cylellips} 
and therefore also for the proof of Theorem
\ref{thm:connectivitytranstition}.

\subsection{Ellipsoids}\label{sec:ellipsoids}
The set of all open ellipsoids in $\R^k$ will be denoted by 
$\mathfrak{E}^k$ while the subset of 
$\mathfrak{E}^k$ consisting of ellipsoids centered at the origin
will be denoted by $\mathfrak{E}^k_o.$
The letters $E$ and ${\mathbf E}$ will typically 
refer to an ellipsoid and a collection of ellipsoids respectively. 
For $E\in \mathfrak{E}^k$, let $\cent(E)$ be the center of $E$. Moreover, 
let $E_o=E-\cent(E)\in \mathfrak{E}_o^k$. Since an ellipsoid 
$E\in \mathfrak{E}^k$ is uniquely determined by the pair 
$(\cent(E),E_o)\in \R^k \times \mathfrak{E}^k_o$, we will in what 
follows often identify $\mathfrak{E}^k$ with $\R^k\times \mathfrak{E}^k_o$.
If $L\in \Line_{H_k^r}$, then the intersection $\mathfrak{c}(L,r)\cap H_k$
induces an ellipsoid in $\R^k$ as follows. Define the map $E_k$ by 

\begin{equation}\label{eqn:ellipsoiddef}
\begin{array}{cccc}
 E_k(\cdot,\cdot) : & {\mathrm A}(d,1)\times (0,1] &\to & \mathfrak{E}^k \\
& (L,r) & \mapsto & \mathfrak{c}(L,r) \cap H_k.
\end{array}
\end{equation}
The fact that $E_k(L,r)$ is indeed an ellipsoid in $\R^k$ is proved 
in Lemma~\ref{lem:cylinderellipsoid}. For convenience, we will let 
$E_k(L,r)=\emptyset$ whenever $L \not \in \CL_{H_k^r}.$

\subsection{The Poisson cylinder model}\label{sec:pcmodel}

As mentioned in the introduction, we will discuss our results in connection with the standard Poisson cylinder model in Section~\ref{sec:altpcproof} so we give its definition before giving the definition of the fractal Poisson cylinder modell.  We define the following space of point measures on ${\mathrm A}(d,1)$:
\begin{equation}
\widehat{\Omega }= \left\{  \hat{\omega} 
= \sum_{ i \geq 1 } \delta_{L_i} :\,L_i\in {\mathrm A}(d,1),\, 
\hat{\omega}(\Line_K) < \infty, \textrm{ for every compact } 
K \subset \R^d \right\}.
\end{equation}
where $\delta_L$ denotes point measure at $L$.  By a minor abuse of notation, 
we shall not distinguish the random measure 
$\hat{\omega} \in \widehat{\Omega}$ and its support 
${\rm supp}(\hat{\omega}) \subset {\mathrm A}(d,1)$. Similar comments apply below.

Let $\lambda\in [0,\infty)$ and let $\widehat{{\mathbb P}}_\lambda$ 
denote the law of a Poisson point process on $\widehat{\Omega}$ with 
intensity measure $\lambda \nu_d$. In the Poisson cylinder model with 
intensity $\lambda$ and radius $r\ge 0$, we first choose 
$\hat{\omega}$ from $\widehat{\Omega}$ according to 
$\widehat{{\mathbb P}}_\lambda$. Then, around each line in 
$\hat{\omega}$, we center a cylinder of base-radius $r$ and 
consider the random set $\widehat{\mathcal C}$ of 
${\mathbb R}^d$ consisting of the union of all these cylinders:

$$\widehat{\mathcal{C}}=\widehat{\mathcal{C}}(\hat{\omega})=\bigcup_{L\in \hat{\omega}}\mathfrak{c}(L,r),$$
and the corresponding vacant set $\widehat{\CV}=\R^d\setminus \widehat{\CC}$. In Section \ref{sec:altpcproof} we describe how to use results obtained in this paper to give an alternative proof of a result from \cite{TW_2012} concerning percolation in $\widehat{\CV}$.

\subsection{The fractal Poisson cylinder model}\label{sec:fpcmodel}
We now introduce the object of main interest in this paper: the fractal Poisson cylinder model. Consider the space of point measures on ${\mathrm A}(d,1)\times (0,1]$:

\begin{eqnarray*}
\lefteqn{\Omega=\{\omega=\sum_{i\ge 1}\delta_{(L_i,r_i)} \mbox{ where }(L_i,r_i)\in {\mathrm A}(d,1)\times (0,1]}\\ && \mbox{ and }\omega(\Line_K\times K') < \infty \mbox{ for all compact }K \subset \R^d,\,K'\subset (0,1]\}.
\end{eqnarray*}
For $\lambda>0$, we let ${\mathbb P}_{\lambda}$ denote the law of a Poisson process on $\Omega$ with intensity measure $\lambda \nu_d \times \varrho_s$
where $\varrho_s$ is a measure on $(0,1]$ defined by
\begin{equation} \label{eqn:scaleinvariantmeasure}
\id \varrho_s(  r) = \indicator (0<r\le 1)r^{-d} \id r.
\end{equation}
Let $K\subset \R^d,$ and consider the set $\Line_{K}\times (b,c]$
with $0<b<c\leq1.$ For $\epsilon>0,$ we observe that it follows from 
\eqref{eqn:nudformula} that $\nu_d(\Line_{\epsilon K})
=\epsilon^{d-1}\nu_d(\Line_{K})$. Therefore, for $0<\epsilon \leq 1/c$ 
we have that 
\begin{eqnarray*}
\lefteqn{\nu_d\times \varrho_s(\Line_{\epsilon K}
\times (\epsilon b,\epsilon c])
=\nu_d(\Line_{\epsilon K}) \int_{\epsilon b}^{\epsilon c} r^{-d} \id r}\\
& & =\epsilon^{d-1}\nu_d(\Line_{K})\frac{\epsilon^{1-d}(c^{1-d}-b^{1-d})}{1-d}
=\nu_d(\Line_{K})\int_{b}^{c}r^{-d}\id r
=\nu_d\times \varrho_s(\Line_{K}
\times (b,c]),
\end{eqnarray*}
where we used Lemma \ref{lem:basicmeasure} part $a)$ in the second equality.
It follows that $\nu_d \times \varrho_s$ is a semi-scale 
invariant measure (it is semi-scale invariant rather than fully scale 
invariant since there is an upper cut off on the radius of the cylinders).
As above, we will frequently abuse notation and write 
$(L,r)\in \omega$ rather than $(L,r) \in \supp(\omega).$

Given $\omega$ picked according to ${\mathbb P}_{\lambda}$, we then 
define the covered region as  
 $${\mathcal C}(\omega)=\bigcup_{(L_i,r_i)\in \omega} \mathfrak{c}(L_i,r_i),$$
 and the vacant region as 
 $${\mathcal V}(\omega)
 ={\mathbb R}^d\setminus {\mathcal C}(\omega).$$
We shall often write simply $\CC$ in place of $\CC(\omega)$ in place of and similarly for $\CV.$
Note that it follows from the semi-scale invariance that 
${\mathbb P}(o \in {\mathcal V})=0$, and so the set ${\mathcal V}$
is a (semi-)scale invariant fractal set.

\subsection{A 0-1 law}
Let $T_t$ denote a shift in the direction of $t\in \BR^d$. Then, 
let $S_t: \Omega \to \Omega$ be the induced transformation defined 
by the equation
\[
S_t \omega=\bigcup_{(L,r)\in \omega} \delta_{(L+t,r)}.
\]
Then, for any measurable event $F \subset {\mathrm A}(d,1) \times (0,1]$ we
define $S_t(F):=\bigcup_{\omega \in F} S_t(\omega).$ Furthermore, an 
event is called shift-invariant if for any $t\in \BR^d,$
we have that $S_t(F)=F.$
We have the following result.
\begin{lemma} \label{lem:zeroone}
For any shift invariant event $F$ we have that
\[
\BP(F)\in \{0,1\}.
\]
\end{lemma}
\noindent
Since the proof of Lemma \ref{lem:zeroone} follows by standard methods
(combined with elementary properties of the cylinder process) we 
refer the reader to Appendix \ref{app:ergo}.

\section{The existence phase transition} \label{sec:existence}
The purpose of this section is to prove Theorem \ref{thm:lambdae}.
The general strategy of the proof is similar to the one used in 
\cite{BJT_2017}, and we will prove Theorem \ref{thm:lambdae}
by considering a lower and upper bound on $\lambda_e$ separately.
The lower bound is given in Proposition 
\ref{prop:emptyupperbound} and and its proof uses a second moment method. 
The upper bound is dealt with in Proposition \ref{prop:emptylowerbound}
and is basically a first moment approach. However, some extra steps are 
needed  in order to establish that the random fractal dies 
out at the critical point.

We will now introduce some notation that is used in this section.
First, let
\[
\chi_n^k 
:=\left\{x+[0,2^{-n}]^d:x\in\left(2^{-n}\BZ^k\times\{0\}^{d-k}\right) 
\cap [0,1-2^{-n}]^d\right\},
\]
so that $\chi_n^k$ contains $2^{kn}$ $d$-dimensional sub-boxes of 
$[0,1]^k \times [0,2^{-n}]^{d-k}$ with side length $2^{-n}$
and with non-overlapping interior. 
For any $k,$ we will call an element $X \in \chi_n^k$ a level 
$n$ box. Furthermore, let
\[
\omega_n := \{(L,r) \in \omega: 2^{-n} \leq r \leq 1\}
\]
and define
\[
\CV_n := \BR^d \setminus \bigcup_{(L,r) \in \omega_n} \Fc(L,r).
\]
We immediately see that $\CV_n \downarrow \CV$.
Furthermore, for $0\leq m<n$ we define
\[
\CV_{m,n} 
=\BR^d \setminus \bigcup_{(L,r) \in \omega_n \setminus \omega_m} \Fc(L,r).
\]
Next, let
\[
m_n^k := \{X \in \chi^k_n : \nexists (L,r) \in \omega_n \mbox{ such that } \Fc(L,r) \cap X \neq \emptyset\},
\]
be the set of level $n$ boxes in $\chi^k_n$ that are untouched by 
the Poisson process $\omega_n.$ Moreover, let
\[
M_n^k := \{X \in \chi_n^k : \nexists (L,r) \in \omega_n \text{ such that } 
X \subset \Fc(L,r)\}
\]
be the set of level $n$ boxes in $\chi_n^k$ not covered by a single 
cylinder. 

Our first result is a proposition which establishes a number of 
preliminary estimates. These will be useful when we prove
Theorem \ref{thm:lambdae}, and also for proving Theorem \ref{thm:Hausdorff}.
\begin{proposition}\label{prop:list}  
For all $1 \leq k\leq d$,  $X \in \chi_n^k$ and 
$x,y \in [0,1]^d$ we have that
\begin{itemize}
\item[a)]  $\BP(X \in m_n^k) \geq  e^{-\lambda C_1} 2^{-\lambda n};$ 
\item[b)]  $\BP(x \in \CV_n) = 2^{-\lambda n};$
\item[c)]  $\BP(x,y \in \CV_n) 
\leq e^{\lambda C_2} 2^{-2\lambda n} \|x-y\|^{-\lambda};$
\item[d)]  $\BP(X \in M_n^k) \leq  e^{\lambda C_3} 2^{-\lambda n}$ whenever $2^{-n}\sqrt{d}\le 1.$
\end{itemize}
Here, $C_1,C_2$ and $C_3$ are constants depending on $d$ only.
\end{proposition}

\begin{proof}
We will prove the statements in order.  

\medskip

\noindent
Part $a)$: Begin by noting that for any $X \in \chi_n^k$ we have that 
$\BP(X \in m_n^k) = \BP([0,2^{-n}]^d \in m_n^k) $ by translation 
invariance. Moreover, $[0,2^{-n}]^d \subset B(o,2^{-n}\sqrt{d})$ 
since $B(o,2^{-n}\sqrt{d})$ is a closed ball. Therefore, if 
the box $[0,2^{-n}]^d$ is hit by a cylinder, then the ball 
$B(o,2^{-n}\sqrt{d})$ must also be hit by the same cylinder.
Hence,
\[
\BP([0,2^{-n}]^d \in m_n^k) 
\geq \BP\left(\left\{(L,r) \in \omega_n : B(o,2^{-n}\sqrt{d}) \cap \Fc(L,r)
\neq \emptyset \right\} = \emptyset \right).
\]
Now, by Lemma \ref{lem:basicmeasure} part $b)$,
\begin{eqnarray}\label{eq:mnprob}
\lefteqn{\BP\left(\{(L,r) \in \omega_n :
B(o,2^{-n}\sqrt{d}) \cap \Fc(L,r) \neq  \emptyset\} = \emptyset \right)} \\
&&= {\rm exp} \left(- \lambda \nu_d \times \varrho_s \left((L,r) \in \omega_n : 
B(o,2^{-n}\sqrt{d}) \cap \Fc(L,r) \neq \emptyset\right)\right) \nonumber\\
&& = {\rm exp} \left(-\lambda\int_{2^{-n}}^1 \nu_d\left(\CL_{B(o,r + 2^{-n}\sqrt{d})}\right)r^{-d} \id r \right)\nonumber \\
&& = {\rm exp} \left(-\lambda \int_{2^{-n}}^1 \left(r +2^{-n}\sqrt{d}\right)^{d-1}r^{-d} \id r \right).\nonumber
\end{eqnarray}


We have that
\begin{eqnarray}\label{eq:intbyparts}
\lefteqn{\int_{2^{-n}}^1 \left(r +2^{-n}\sqrt{d}\right)^{d-1}r^{-d} \id r}\\ && =\int_{2^{-n}}^1\sum_{i=0}^{d-1}\binom{d-1}{i}r^{d-i-1}(2^{-n}\sqrt{d})^i r^{-d}\id r\nonumber\\&&=\sum_{i=0}^{d-1}\binom{d-1}{i} (2^{-n}\sqrt{d})^i\int_{2^{-n}}^1 r^{-1-i}\id r \nonumber\\ && = n \log 2 + \sum_{i=1}^{d-1}d^{i/2}\binom{d-1}{i}\left(\frac{1}{i}-\frac{2^{-ni}}{i}\right)\\ &&\le n\log 2+ \sum_{i=1}^{d-1}\frac{d^{i/2}}{i}\binom{d-1}{i} =n \log 2 +C_1.\nonumber
\end{eqnarray}
Plugging this into \eqref{eq:mnprob} gives the desired result.

\medskip

\noindent
Part $b)$: Using Lemma \ref{lem:basicmeasure} part $b)$,
\begin{eqnarray*}
\lefteqn{\BP(x \in \CV_n) = \exp\left(-\lambda \int_{2^{-n}}^1 \int_{{\mathrm A}(d,1)} \indicator(\{L \in {\mathrm A}(d,1) : x \in \Fc(L,r)\})
\id \nu_d( L)\id \varrho_s(r)\right)}\\
&& = \exp\left(-\lambda \int_{2^{-n}}^1 \nu_d(\CL_{B(x,r)})r^{-d} \id r\right)  
= \exp\left(-\lambda \int_{2^{-n}}^1r^{-1} \id r\right)  = 2^{-\lambda n}.
\end{eqnarray*}

\medskip

\noindent
Part $c)$: Let $\CP_x^n = \{(L,r)\in {\mathrm A}(d,1) \times (2^{-n},1]: x \in \Fc(L,r)\}$ 
denote the set of cylinders with radii in $(2^{-n},1]$ that covers 
the point $x$. We have that
\begin{eqnarray}\label{eq:probx1x2inmn}
\lefteqn{\BP(x, y \in \CV_n)}\\
&&= \BP(\{(L,r) \in \omega_n : x \in \Fc(L,r)\} = \emptyset,\{(L,r) \in \omega_n : y \in \Fc(L,r)\} = \emptyset)\nonumber \\
&&= {\rm exp}\left(-\lambda \nu_d \times \varrho_s
\left(\CP^n_x \cup \CP^n_y\right)\right) 
 = {\rm exp}(-\lambda (2\nu_d \times \varrho_s(\CP^n_x)
 -\nu_d \times \varrho_s(\CP^n_x \cap \CP^n_y)))\nonumber \\
& & =2^{-2 \lambda n} {\rm exp}
(\lambda \nu_d \times \varrho_s(\CP^n_x \cap \CP^n_y)),
\nonumber 
\end{eqnarray}
where we used translation invariance and part $b)$ of the current lemma. 
We now turn our attention to $\nu_d \times \varrho_s(\CP^n_x \cap \CP^n_y)$. 
As before, note that
\begin{eqnarray}\label{eq:mub1b2}
\lefteqn{\nu_d \times \varrho_s(\CP^n_x \cap \CP^n_y) 
= \int_{2^{-n}}^1 \int_{{\mathrm A}(d,1)} \indicator 
\left(\CP^n_x \cap \CP^n_y\right)\id \nu_d(L)\id \varrho_s( r)} \\
&& = \int_{2^{-n}}^1 \nu_d\left(\CL_{B\left(x, r\right), B\left(y,r \right)}\right)\id \varrho_s(r) \nonumber 
= \int_{2^{-n}}^1 r^{d-1}\nu_d(\CL_{B(x',1), B(y',1)})r^{-d}\id r,\nonumber
\end{eqnarray}
where we used Lemma \ref{lem:basicmeasure} part $c)$ and where 
$x', y'$ are such that $\dist{x'}{y'}=\dist{x}{y}/r.$ Because of this, 
we have that for fixed $x, y$ the ``effective'' distance $\dist{x'}{y'}$ 
will be large whenever $r$ is small. Furthermore, 
if $r \leq \dist{x}{y}/4$, then $\dist{x'}{y'} = \dist{x}{y}/r \geq 4$ 
and in this case we can use Lemma 
\ref{lem:measurebounds_org} 
to obtain that
\[
\nu_d\left(\CL_{B(x',1),B(y',1)}\right) 
\leq c_2 \dist{x'}{y'}^{1-d}.
\] 
It is therefore natural to split the integral on the right hand side of 
\eqref{eq:mub1b2} into two parts. The first integral will be over 
small $r$ so that $x', y'$ are well separated, while the second integral 
part is over larger $r$. 
Now, if $\dist{x}{y} \geq 2^{2-n}$ we can write
\begin{eqnarray}\label{eq:integralsplit}
\lefteqn{\int_{2^{-n}}^1 r^{-1}
\nu_d\left(\CL_{B(x',1),B(y',1)}\right)\id r}\\
&&=\int_{2^{-n}}^{\dist{x}{y}/4}r^{-1}
\nu_d\left(\CL_{B(x',1),B(y',1)}\right)\id r 
+\int_{\dist{x}{y}/4}^1 r^{-1}\nu_d\left(\CL_{B(x',1),B(y',1)}\right)\id r. 
\nonumber
\end{eqnarray}
By Lemma \ref{lem:measurebounds_org} we have that 
$\nu_d\left(\CL_{B(x',1),B(y',1)}\right) \leq c_2 \dist{x'}{y'}^{1-d} 
= c_2 \dist{x}{y}^{1-d} r^{d-1}$ and then
\begin{eqnarray}\label{eq:integralsmallr}
\lefteqn{\int^{\dist{x}{y}/4}_{2^{-n}} r^{-1}
\nu_d\left(\CL_{B(x',1),B(y',1)}\right)\id r
\leq c_2 \dist{x}{y}^{1-d} \int^{\dist{x}{y}/4}_{2^{-n}}r^{d-2}\id r}\\
&& = \frac{c_2}{d-1} \dist{x}{y}^{1-d}
\left(\left(\frac{\dist{x}{y}}{4}\right)^{d-1}-2^{-n(d-1)} \right)
\leq \frac{c_2}{(d-1)4^{d-1}}.\nonumber
\end{eqnarray}
We will now provide a bound to the second term in the right hand 
side of Equation 
\eqref{eq:integralsplit}.
Start by noting that $\CL_{B(x',1),B(y',1)} \subset \CL_{B(x',1)}$ and  
$\nu_d(\CL_{B(x',1)})=1$ so that  
$\nu_d(\CL_{B(x',1),B(y',1)})\leq 1$. Therefore
\begin{equation}\label{eq:trivialboundD}
\int_{\dist{x}{y}/4}^1 r^{-1}\nu_d\left(\CL_{B(x',1),B(y',1)}\right)\id r 
\leq \int_{\dist{x}{y}/4}^1 r^{-1} \id r
= - \log(\dist{x}{y}/4).
\end{equation}
Using Equations \eqref{eq:integralsplit}, \eqref{eq:integralsmallr} and 
\eqref{eq:trivialboundD} we conclude that
\begin{equation}\label{eq:conclusionboundD}
\int_{2^{-n}}^1 r^{-1}
\nu_d\left(\CL_{B(x',1),B(y',1)}\right)\id r
\leq \frac{c_2}{(d-1)4^{d-1}} - \log(\dist{x}{y}/4).
\end{equation}
On the other hand, if $\dist{x}{y}< 2^{2-n}$ then $\dist{x}{y}/4 < 2^{-n}$, 
and so we can estimate \eqref{eq:integralsplit} by
\begin{eqnarray*}
\lefteqn{\int_{2^{-n}}^1 r^{-1}
\nu_d\left(\CL_{B(x',1),B(y',1)}\right)\id r}\\
&& \leq \int_{\dist{x}{y}/4}^{1}r^{-1}
\nu_d\left(\CL_{B(x',1),B(y',1)}\right)\id r 
\leq -\log(\dist{x}{y}/4), \nonumber
\end{eqnarray*}
where the last inequality comes from \eqref{eq:trivialboundD}.
Thus, we see that \eqref{eq:conclusionboundD} holds for all $\dist{x}{y}>0.$ 
Combining \eqref{eq:probx1x2inmn}, \eqref{eq:mub1b2} and 
\eqref{eq:conclusionboundD} we obtain
\begin{eqnarray}\label{eq:twopointprob}
\lefteqn{\BP(x, y \in \CV_n)}\\
&& \leq 2^{-2\lambda n}{\rm exp} 
\left(\lambda\left(\frac{c_2}{(d-1)4^{d-1}}-\log(\dist{x}{y}/4)\right)\right) 
= e^{\lambda C_2} 2^{-2\lambda n}\dist{x}{y}^{-\lambda},\nonumber
\end{eqnarray}
as desired.

\medskip

\noindent
Part $d)$: We now assume that $2^{-n}\sqrt{d}\le 1$. Let $p$ be the center point of the box $X$. Clearly, 
if $X$ is not covered by a single cylinder, then the ball 
$B(p,2^{-n}\sqrt{d})$ circumscribing it cannot be covered by 
a single cylinder. Furthermore, note that 
$B(p,2^{-n}\sqrt{d}) \subset \Fc(L,r)$ is equivalent to the condition 
that $L$ intersects the ball $B(p,r - 2^{-n}\sqrt{d})$. 
Therefore, 
\begin{eqnarray*}
\lefteqn{\BP(X \in M_n^k) 
\leq \BP\left(\{(L,r) \in \omega_n : B(p,2^{-n}\sqrt{d}) \subset \Fc(L,r)\} = \emptyset \right)} \\
&&=\exp\left(-\lambda \int_{2^{-n}}^1 
\nu_d\left(\CL_{B(p,r - 2^{-n}\sqrt{d})}\right)r^{-d}\id r\right)\\
&& = \exp\left(-\lambda \int_{2^{-n}\sqrt{d}}^1
\left(r - 2^{-n}\sqrt{d}\right)^{d-1}r^{-d}\id r\right).
\end{eqnarray*}
Now, note that
\begin{eqnarray*}
\lefteqn{\int_{2^{-n}\sqrt{d}}^1\left(r - 2^{-n}\sqrt{d}\right)^{d-1}r^{-d}\id r}\\ &&
= \int_{2^{-n}\sqrt{d}}^1\left(\sum_{i=0}^{d-1}\binom{d-1}{i}r^{d-1-i}
\left(-2^{-n}\sqrt{d}\right)^i\right)r^{-d}\id r
\\
&& = n \log 2-\frac{\log d}{2} +\sum_{i=1}^{d-1}\binom{d-1}{i}\frac{(-1)^i}{i}(1 - 2^{-in}d^{i/2}) \\ &&
\geq n\log 2 -\frac{\log d}{2}-\sum_{\substack{i=1 \\ i \text{ odd}}}^{d-1}\binom{d-1}{i}\frac{1}{i} \\
&& = n \log 2 - C_3.
\end{eqnarray*}
Thus we have that,
\[
\BP(X \in M_n^k) 
\leq \exp\left( -\lambda n \log 2 + \lambda C_3 \right) 
= e^{\lambda C_3} 2^{-\lambda n},
\]
as desired.
\end{proof}

\noindent
\begin{remark}
Although we do not need to know the values of 
$C_1, C_2$ and $C_3$ for the results of this paper, 
 we observe for possible future reference that
\begin{equation} \label{eqn:constants}
C_1 = \sum_{i=1}^{d-1}\frac{d^{i/2}}{i}\binom{d-1}{i}, \ \ 
C_2 = \frac{c_2}{(d-1)4^{d-1}}+\log 4 \ \text{ and } \ 
C_3 = \frac{\log{d}}{2}+\sum_{\substack{i=1 \\ i \text{ odd}}}^{d-1}\binom{d-1}{i}
\frac{1}{i}.
\end{equation}
\end{remark}

Throughout the rest of this section and Section \ref{sec:dimensions}, 
it will be convenient to use the notation 
$\CV^k:=\CV\cap ([0,1]^k \times \{0\}^{d-k})$
and $\CV^k_n:=\CV_n\cap ([0,1]^k \times \{0\}^{d-k})$. 
We will also let 
$\CV^k_{m,n}:=\CV_{m,n}\cap ([0,1]^k \times \{0\}^{d-k})$ when $m<n.$

We are ready to prove the following proposition which provides us 
with one direction of Theorem \ref{thm:lambdae}.
\begin{proposition}\label{prop:emptyupperbound}
For $d \geq 2$ and any $1\leq k \leq d$, we have that $\lambda_e(d,k) \geq k $.
\end{proposition}
\begin{proof}
Start by noting that since the cylinders in the process are open sets, 
the sets $\CV_n^k$ are compact for all $n$.
Moreover, we have that $\CV_n^k
\supset \CV_{n+1}^k$ for every $n$. 
Therefore we conclude that if $m_n^k \neq \emptyset$ for infinitely 
many $n\geq 1$, then $\CV_n^k \neq \emptyset$ 
for every $n\geq 1$. It then follows by compactness that
\[
\CV^k= \bigcap_{n=1}^\infty \CV_n^k \neq \emptyset.
\]
Thus, if we prove that for any $\lambda <k$, 
\begin{equation}\label{eq:mnio}
\BP(|m_n^k| > 0 \mbox{ infinitely often}) \geq c >0,
\end{equation} 
where $c$ depends only on $\lambda$ and $d$, we would get that 
$\CV^k \neq \emptyset$ with positive probability, as desired 
(here $|\cdot|$ stands for cardinality).

Observe that by the reverse Fatou's Lemma,
\[
\BP(|m_n^k| > 0 \mbox{ infinitely often}) 
\geq \limsup \BP(|m_n^k| > 0),
\]
and so it suffices to show that there exists $c = c(\lambda) >0$ such that
\[
\BP(|m_n^k| > 0)\geq 
\frac{\BE(|m_n^k|)^2}{\BE\left(|m_n^k|^2\right)}\geq c
\]
uniformly in $n$ (where we used the second moment method for the first inequality).

It follows from part $a)$ of Proposition \ref{prop:list} that 
\begin{equation}\label{eq:secondmoment}
\BE(|m_n^k|) = 2^{nk} \BP([0,2^{-n}]^d \in m_n^k)  
\geq e^{-\lambda C_1} 2^{(k-\lambda)n}.
\end{equation}
We proceed to provide an upper bound to $\BE(|m_n^k|^2).$ 
We have that
\[
\BE(|m_n^k|^2) = \sum_{X_1,X_2 \in \chi_n^k} \BP(X_1, X_2 \in m_n^k).
\]
We split the sum into two parts:
\begin{eqnarray}\label{eq:sumsplit}
\lefteqn{\sum_{X_1,X_2 \in \chi_n^k} \BP(X_1, X_2 \in m_n^k)} \\
&&= \sum_{\substack{X_1,X_2 \in \chi_n^k \\ \dist{p_1}{p_2} \leq 2^{-n+2}}} 
\BP(X_1, X_2 \in m_n^k)
+\sum_{\substack{X_1,X_2 \in \chi_n^k \\ \dist{p_1}{p_2} > 2^{-n+2}}} 
\BP(X_1, X_2 \in m_n^k),\nonumber
\end{eqnarray}
where $p_i=p_i(X_i)$ is the center point of the box $X_i$. 
This split will be necessary because the techniques used to bound 
$\BP(X_1, X_2 \in m_n^k)$ depends on the distance between the boxes.
In order to provide an upper bound to $\BP(X_1, X_2 \in m_n^k)$ 
we only need to provide an upper bound to the probability that the 
center points $p_1, p_2$ of the boxes $X_1,X_2$ are not hit by any 
cylinders in $\omega_n$. In fact, any pair of points 
$(p_i,p_j) \in X_1\times X_2$ would do.

We will begin by providing a bound to the first sum on 
the right hand side of \eqref{eq:sumsplit}.
Note that
\[
\sum_{\substack{X_1,X_2 \in \chi_n^k \\ \dist{p_1}{p_2} 
\leq 2^{-n+2}}} \BP(X_1, X_2 \in m_n^k)
\leq \sum_{X_1 \in \chi_n^k} \sum_{\substack{X_2 \in \chi_n^k \\ 
\dist{p_1}{p_2} \leq 2^{-n+2}}} \BP(X_2 \in m_n^k).
\]
For a fixed $X_1$ (and therefore fixed $p_1$), there are at most
$9^k$ different $X_2\in \chi_n^k$ (or centres $p_2$) such that 
$\dist{p_1}{p_2}\leq 2^{-n+2}$ (there can be fewer if for instance 
$X_1=[0,2^{-n}]^k \times \{0\}^{d-k}$). 
Therefore we can use Proposition \ref{prop:list} part $b)$ to see that
\[
\sum_{\substack{X_2 \in \chi_n^k \\ 
\dist{p_1}{p_2} \leq 2^{-n+2}}}\BP(X_2 \in m_n^k) 
\leq 9^k \BP(X_2 \in m_n^k) \leq 9^k \BP(p_2 \in \CV_n) 
= 9^k 2^{-\lambda n}.
\]
Hence,
\begin{equation}\label{eq:sumnearboxes}
\sum_{\substack{X_1,X_2 \in \chi_n^k \\ \dist{p_1}{p_2} \leq 2^{-n+2}}} 
\BP(X_1, X_2 \in m_n^k)
\leq 9^k 2^{(k-\lambda)n},
\end{equation}
since we have $2^{kn}$ possible choices for $X_1$.

We now turn our attention to the second sum in the right hand side 
of \eqref{eq:sumsplit}. Observe therefore that for fixed $X_1$ 
we have that 
\[
\sum_{\substack{X_2 \in \chi_n^k \\ \dist{p_1}{p_2} > 2^{-n+2}}} 
\BP(X_1, X_2 \in m_n^k)
\leq \sum_{l=2}^{n-1} \sum_{\substack{X_2 \in \chi_n^k \\ 
2^l 2^{-n} < \dist{p_1}{p_2} \leq 2^{l+1}2^{-n}}}
\BP(p_1, p_2 \in \CV_n),
\]
and that the number of $X_2\in \chi_n^k$ satisfying 
$2^l 2^{-n} < \dist{p_1}{p_2} \leq 2^{l+1}2^{-n}$ can be bounded by 
\[
(2\cdot 2^{l+1}+1)^k-(2\cdot 2^{l}+1)^k
\leq 3^k 2^{kl+k}\leq 9^k 2^{kl}.
\]
Therefore, we can use part $c)$ of Proposition \ref{prop:list} 
to conclude that
\begin{eqnarray} \label{eqn:sumdistboxes}
\lefteqn{\sum_{\substack{X_1,X_2 \in \chi_n^k \\ 
\dist{p_1}{p_2} \geq 2^{-n+2}}} \BP(X_1, X_2 \in m_n^k)
 \leq 2^{kn} \sum_{l=2}^{n-1} \sum_{\substack{X_2 \in \chi_n^k \\ 
2^l 2^{-n} < \dist{p_1}{p_2} \leq 2^{l+1}2^{-n}}}\BP(p_1, p_2 \in \CV_n)} \\
& & \leq 2^{kn} \sum_{l=2}^{n-1} 9^k 2^{kl} 
e^{\lambda C_2} 2^{-2\lambda n }(2^l 2^{-n})^{-\lambda} 
= e^{\lambda C_2}9^k 2^{(k-\lambda)n} \sum_{l=2}^{n-1} (2^{k-\lambda})^l
\leq \frac{e^{\lambda C_2}9^k}{2^{k-\lambda}-1} 2^{2(k-\lambda)n}. \nonumber
\end{eqnarray}
Using \eqref{eq:sumsplit}, \eqref{eq:sumnearboxes} and 
\eqref{eqn:sumdistboxes} we get that 
\[
\BE(|m_n^k|^2) = \sum_{X_1,X_2 \in \chi_n^k} \BP(X_1, X_2 \in m_n^k ) 
\leq 9^k 2^{(k-\lambda)n} 
+\frac{e^{\lambda C_2}9^k}{2^{k-\lambda}-1} 2^{2(k-\lambda)n}.
\]
We combine this with \eqref{eq:secondmoment} to conclude that 
\begin{eqnarray*}
\lefteqn{\BP(|m_n^k| > 0) 
\geq \frac{\BE(|m_n^k|)^2}{\BE\left(|m_n^k|^2\right)}} \\
&& \geq \frac{ e^{-2\lambda C_1} 2^{2(k-\lambda)n}}{9^k 2^{(k-\lambda)n} 
+\frac{e^{\lambda C_2}9^k}{2^{k-\lambda}-1} 2^{2(k-\lambda)n}}
=\frac{ e^{-2\lambda C_1}(2^{k-\lambda}-1)}
{9^k 2^{-(k-\lambda)n}(2^{k-\lambda}-1)+e^{\lambda C_2}9^k}.
\end{eqnarray*}
Since $\lambda<k$ we see that $\BP(|m_n^k| > 0)$ is uniformly 
bounded away from zero.
\end{proof}

We will now prove that $\lambda_e(d,k) \leq k$ and that we are in the 
empty phase at the critical point $\lambda_e = k$. For that, we will 
need some additional results. Let $D_n \subset \chi_n^d$ be a minimal 
covering of $ \CV_n^d$, that is, $\CV_n^d \subset \cup_{X\in D_n}X$ 
and let $D_n^k = D_n \cap \chi_n^k$. The choice of $D_n$ is not necessarily 
unique since a point belonging to the boundary of a box $X \in \chi_n$
can be covered by more than one box. We assume therefore that $D_n$
is picked according to some predetermined rule. 
Note also that for any $k$ we must have that 
$\CV_n^k \subset \cup_{X\in D^k_n}X.$ Obviously, $D_n^k$ depends on the configuration $\omega_n$, and therefore we sometimes emphasise this by writing $D_n^k(\omega_n)$ or similar.

The next lemma relates the event $\CV^k \neq \emptyset$ with the 
limiting behaviour of $|D^k_n|.$ The lemma is similar to Lemma $3.2$ in \cite{BJT_2017} and Lemma $2.1$ in \cite{B_2018}, but the proof provided here is more detailed. 

\begin{lemma}\label{lem:afterMn}
For any $\lambda > 0$ and $1 \leq k \leq d$ we have that
\[
\BP(\{\CV^k \neq \emptyset\} 
\setminus \{\lim_{n\to \infty} |D_n^k| = \infty\}) = 0.
\]
\end{lemma}
\begin{proof} We will construct a sequence $(\eta_n)_{n \geq 1}$ in 
a specific way so that it has the same distribution as 
$(\omega_n)_{n \geq 1}$. The statement then follows from the 
details of this construction.

Fix $L<\infty,$ and let $\CA_n$ be the 
event that $0<|D_{n}^k| \leq L.$ Note that there exists 
$\alpha = \alpha(\lambda)$ such that 
$\BP(\CV^k_1  = \emptyset)= \alpha > 0$.

The main idea is to prove that for any such $L$ the events $\CA_n$ can only occur finitely many times. Intuitively, this is clear since every time $\CA_n$ occurs, there is a uniform positive probability to "kill" the process.

Let $\eta_{1}$ be a Poisson process with intensity measure
$\lambda \nu_d \times 1\{2^{-1}<r\leq 1\} r^{-d}\id r$,
and observe that $\eta_1$ has the same distribution as $\omega_1.$
Then, we let $\eta_{1,2}$ be a Poisson process with intensity measure
$\lambda \nu_d \times 1\{2^{-2}<r\leq 2^{-1}\} r^{-d}\id r$,
and we pick this independent of $\eta_{1}.$ 
Observe that $\eta_2:=\eta_1+\eta_{1,2}$ has the same distribution 
as $\omega_2.$ If the event $\CA_n$ never occurs, we let 
$X_1,X_2,\ldots\in \{0,1\}$ be an i.i.d. sequence of random variables
such that $\BP(X_i=0)=\alpha^L$ for $i\in \{1,2,\ldots\}.$
Otherwise, we continue the procedure of constructing
$\eta_n=\eta_{n-1}+\eta_{n-1,n}$ until the first 
time $n_1$ such that $\CA_{n}$ occurs. 

Assume therefore that $\CA_{n_1}$ occurs.
As before, let $\eta_{n_1, n_1+1}$ be a Poisson process with intensity 
measure $\lambda \nu_d \times 1\{2^{-n_1-1}<r\leq 2^{-n_1}\} r^{-d}\id r$, and let 
\[
X_1=\indicator (\CV_{n_1+1}^k(\eta_{n_1}+\eta_{n_1,n_1+1})\neq\emptyset).
\]




We note that for any $\eta_{n_1}\in \CA_{n_1}$ we have that 
\begin{eqnarray*}
\lefteqn{\BP(\CV_{n_1+1}^k(\eta_{n_1}+\eta_{n_1,n_1+1})=\emptyset | \eta_{n_1})
\geq \BP(\CV_{n_1,n_1+1}^k(\eta_{n_1,n_1+1}) \cap D_{n_1}^k(\eta_{n_1})
=\emptyset | \eta_{n_1})}\\
& & \geq \prod_{X\in D_{n_1}^k} \BP(\CV_{n_1,n_1+1}^k(\eta_{n_1,n_1+1}) \cap X
=\emptyset)
=\BP(\CV_1^k=\emptyset)^{|D_{n_1}^k|}
=\alpha^{|D_{n_1}^k|}\geq \alpha^L,
\end{eqnarray*}
and so $\BP(X_1=0)\geq \alpha^L.$
Here, the second inequality follows by the FKG inequality for 
Poisson processes, while the second to last equality follows by 
scale invariance.

We then proceed as follows. If $\CA_n$ never occurs again, 
we let $X_2,X_3,\ldots\in\{0,1\}$ be an i.i.d. sequence 
such that $\BP(X_i=0)=\alpha^L$ for $i\in\{2,3\ldots\}.$
Otherwise, 
we proceed with 
the construction until time $n_2,$ which is the first time after $n_1$
such that $\CA_{n}$ occurs. Then, we construct $X_2$ analogously 
to how we constructed $X_1$ above. Note that the dependence between 
$X_1$ and $X_2$ is potentially complicated, but that for every 
$\eta_{n_2} \in \CA_{n_2}$ we must have (as above) that 
$\BP(\CV_{n_2+1}^k(\eta_{n_2}+\eta_{n_2,n_2+1})=\emptyset | \eta_{n_2})
\geq \alpha^L$. Therefore, 
\[
\BP(X_1=1,X_2=1)=\BE[\BP(X_2=1 | X_1=1 ,\eta_{n_2})]
\BP(X_1=1) \leq (1-\alpha^L)^2.
\]
Proceeding in this manner we see that 
\[
\BP(0<|D_n^k| \leq L \text{ i.o.})
=\lim_{l \to \infty} \BP(X_1=\cdots=X_l=1)
\leq \lim_{l \to \infty} (1-\alpha^L)^l=0.
\]
This implies that
$\BP(\lim_{n \to \infty}|D_n^k| \in \{0,\infty\} ) =1.$ 
Furthermore if $\CV^k \neq \emptyset,$ then $\CV_n^k \neq \emptyset$ 
for all $n$, which implies that $|D_n^k| \neq 0$ for all $n$ and so 
we conclude that 
$\BP(\lim_{n \to \infty}|D_n^k| = \infty| \CV^k \neq \emptyset ) =1.$
\end{proof}

\begin{proposition}\label{prop:emptylowerbound}
For $d \geq 2$ and $1 \leq k \leq d$ we have that $\lambda_e(d,k) \leq k$. 
Furthermore, 
\[
\BP_{\lambda_e(d,k)}(\CV = \emptyset) =1.
\]
\end{proposition} 

\begin{proof}
Observe that if $X \in D_n^k$, then $X$ is not covered by a single cylinder 
in the Poisson process and so $|D_n^k| \leq |M_n^k|.$ 
Therefore, if $\lambda$ is such that $\BP_{\lambda}(\CV^k \neq \emptyset) > 0$ 
we see that by Lemma \ref{lem:afterMn} we have that 
$\BP_{\lambda}(\lim_{n\to\infty} |D_n^k| = \infty)> 0$ and so
\begin{equation}\label{eq:boundMn}
\lim_{n\to \infty}\BE_{\lambda}(|M_n^k|) \geq \lim_{n\to \infty}\BE_{\lambda}(|D_n^k|) = \infty.
\end{equation}

We will now provide an upper bound to $\BE_{\lambda}(|M_n^k|)$. From 
Proposition \ref{prop:list} part $d)$ we have that
\[
\BP_{\lambda}(X \in M_n^k) 
\leq e^{\lambda C_3} 2^{-\lambda n}
\]
and so
\begin{equation}\label{eq:evMn}
\BE_{\lambda}(|M_n^k|) \leq 2^{kn}\BP_\lambda(X \in M_n^k) \leq e^{\lambda C_3} 2^{(k-\lambda)n}.
\end{equation}
The discussion above Equation \eqref{eq:boundMn} implies that
\[
\lim_{n \to \infty} e^{\lambda C_3} 2^{(k-\lambda)n} 
\geq \lim_{n \to \infty} \BE_{\lambda}(|M_n^k|) = \infty,
\]
whenever $\BP_{\lambda}(\CV^k\neq \emptyset) >0$.
Clearly, this only holds if $k > \lambda$. We conclude that 
$\lambda_e(d,k) \leq k$ and that for $\lambda=k$ we must have that 
$\BP_{\lambda}(\CV^k \neq \emptyset)=0$.
\end{proof}

\begin{proof}[Proof of Theorem \ref{thm:lambdae}]
This follows from Propositions \ref{prop:emptyupperbound} and 
\ref{prop:emptylowerbound}.
\end{proof}

\section{The Hausdorff dimension of $\CV$} \label{sec:HaussofV}

In this section we will prove Theorem \ref{thm:Hausdorff} which 
establishes the almost sure Hausdorff dimension of $\CV\cap H_k.$ 
We will prove Theorem \ref{thm:Hausdorff} by establishing the upper 
and lower bound separately. As usual, it is easier to determine the 
upper bound, so we will do this first. Recall that the definitions of Hausdorff measure and fractal dimensions are found in Section~\ref{sec:dimensions}.
\begin{theorem}\label{thm:Vkuppdim}
For $\lambda<k,$ we have that $\dim_\CH(\CV \cap H_k) \leq k-\lambda$ 
almost surely.
\end{theorem}
\begin{proof}
We will prove that $\dim_\CH(\CV^k) \leq k-\lambda$ almost surely.
The statement then follows by tiling $H_k$ with copies of 
$[0,1]^k \times \{0\}^{d-k}$ and using a countability argument.

Observe that the set $M_n^k$ defined before Proposition \ref{prop:list} 
is a $\sqrt{k}2^{-n}$-cover of $\CV^k$. Therefore 
$\CH^s(\CV^k) \leq \sum_{X \in M_n^k}{\mathrm{diam}}(X)^s 
= \left|M_n^k\right|(\sqrt{k}2^{-n})^s.$
Then, by~\eqref{eq:evMn} we have that 
$\BE[\CH^s(\CV^k)] \leq (\sqrt{k}2^{-n})^s\BE\left(\left|M_n^k\right|\right) 
\leq e^{\lambda C_3}(\sqrt{k})^s2^{(k-\lambda -s)n}$ for all $n$. 
Thus if $k-\lambda -s<0$ then $\BE[\CH^s(\CV^k)]= 0$ and since 
$\CH^s(\CV^k) \geq 0$, we must have that almost surely  $\CH^s(\CV^k)= 0.$
By the definition of Hausdorff 
dimension, we conclude that $\dim_\CH(\CV^k) \leq k-\lambda$ almost surely.
\end{proof}
\noindent
\begin{remark}\label{rem:vkuppdim}
We note that it does not follow from (an extension of) 
Theorem \ref{thm:Vkuppdim} that $\lambda_e \leq k$. The reason for 
this is that for $\lambda>k,$ the set 
$\CV \cap H_k$ could be non-empty while still having Hausdorff dimension 0.
However, it is the case that our next result, Theorem \ref{thm:Vklowdim},
implies Proposition \ref{prop:emptyupperbound}. The reason for
providing the proof of Proposition \ref{prop:emptyupperbound} is that 
it is done from first principles while the proof of Theorem 
\ref{thm:Vklowdim} is much more involved. 
\end{remark}


The next step is to find a lower bound on the Hausdorff dimension 
of $\CV\cap H_k$. We will do this by establishing that for any $\lambda<k,$
$\dim_\CH(\CV^k)\geq k-\lambda$ with positive probability, and
then use Lemma \ref{lem:zeroone} to conclude that
$\dim_\CH(\CV \cap H_k) \geq k-\lambda$ almost surely. The aim of the 
rest of this section is to prove the following theorem.

\begin{theorem} \label{thm:Vklowdim}
For any $\lambda<k$ we have that 
\[
\BP(\dim_\CH(\CV^k)\geq k-\lambda)>0.
\]
\end{theorem}

The proof of Theorem \ref{thm:Vklowdim} will proceed through a number 
of lemmas and the overall approach is inspired by \cite{Z_1984}.
As is common when proving lower bounds on Hausdorff dimensions we will
be utilizing Frostman's Lemma (see for example Theorem 4.13 of 
\cite{F_2014}). This lemma states that if there is a random
measure $\zeta$ supported on $\CV^k$ and satisfying 
$0 < \zeta(\CV^k) <\infty$ with finite $r-$energy, that is, 
\begin{equation}\label{eq:finiterenergy}
\CI_{r}(\zeta,[0,1]^k) 
= \int_{[0,1]^k \times [0,1]^k}\frac{\id\zeta(x)\id\zeta(y)}{\|x-y\|^{r}} < \infty, 
\end{equation}
then $\dim_\CH(\CV^k) \geq r$. When useful we will emphasize 
the dependence of $\zeta$ on $\omega$ by writing $\zeta(\omega).$
Observe that we allow a small abuse 
of notation by considering $\zeta$ to be a measure on $[0,1]^k$ rather 
than on $[0,1]^k\times\{0\}^{d-k}$. A similar comment applies to many 
places in this section. The objective is therefore to find a 
suitable random measure $\zeta$ as described.
This will allow us to conclude that $\BP(\dim_\CH(\CV^k) \geq r)$ 
is uniformly bounded away from 0 for $r< \lambda -k.$ 
It then follows that also $\BP(\dim_\CH(\CV^k) \geq \lambda-k)>0.$

The measure $\zeta$ will be obtained as a limit of a sequence 
of random measures $(\zeta_n)_{n\geq 1}.$ Therefore, we let 
$\zeta_n=\zeta_n(\omega)$ be a measure on $[0,1]^k$ defined by
\begin{equation}\label{eq:mundef}
\id\zeta_n(x) =  2^{\lambda n} \indicator (x \in \CV^k_n)\id x.
\end{equation}
We have the following lemma.
\begin{lemma}\label{lem:munmartingale}
Let $f: [0,1]^k \to \BR$ be a continuous and 
non-negative function with compact support, and let 
$X_n(f) = \int f d\zeta_n$. Then $(X_n(f))_{n \geq 1}$ is a non-negative 
martingale. 
\end{lemma}
\begin{proof}
Clearly $X_n(f) \geq 0$. Since $f$ is continuous and 
with compact support, it is also bounded, and so
\[
0 \leq \BE[X_n(f)] = \BE\left[\int f \id\zeta_n\right] 
\leq \sup(f) \BE\left[ \int \id\zeta_n\right]
= \sup(f) 2^{\lambda n} \BE [\ell_k(\CV^k_n)] < \infty,
\]
(recall that $\ell_k$ denotes $k$-dimensional normalized Hausdorff measure). 
Now let $0 \leq m < n$. We have that
\begin{eqnarray*}
\lefteqn{\BE[X_n(f) | \omega_m] = \BE\left[\int f \id\zeta_n | \omega_m\right] 
= \BE\left[\int f(x)2^{\lambda n}\indicator(x \in \CV^k_n)\id x|\omega_m\right]} \\
&& = \BE\left[\int f(x) 2^{\lambda n} \indicator(x \in \CV^k_m)
\indicator(x \in \CV^k_{m,n})\id x | \omega_m\right] \\
&& = \int f(x) 2^{\lambda n} \BE[\indicator(x \in \CV^k_m)| \omega_m]
\BE[\indicator(x \in \CV^k_{m,n})| \omega_m]\id x  \\
&& = \int f(x) 2^{\lambda n} \indicator(x \in \CV^k_{m})
\BP(x \in \CV^k_{m,n})\id x, 
\end{eqnarray*}
where we used the independence of $\CV^k_{m}$ and $\CV^k_{m,n}$ 
in the penultimate equality, while in the last equality  
we used that $\CV^k_{m}$ is measurable 
with respect to the $\sigma$-algebra generated by $\omega_m$
and that $\CV^k_{m,n}$ is independent of $\omega_m$.
As in the proof of Proposition \ref{prop:list} part $b)$ we have that 
$\BP(x \in \CV^k_{m,n})= 2^{\lambda(m- n)}$
and so
\[
\int f(x) 2^{\lambda n} \indicator(x \in \CV^k_{m})
\BP(x \in \CV^k_{m,n})\id x 
= \int f(x) 2^{\lambda m} \indicator(x \in \CV^k_{m})\id x 
= \int f d\zeta_m
= X_m(f)
\]
and the proof is complete.
\end{proof}

\begin{corollary}\label{coro:weakrv}
Let $f: [0,1]^k \to \BR$ be a continuous (not 
necessarily non-negative) function with compact support and define 
$X_n(f)$ as before. Then, there exists $X(f) < \infty$ such that 
$X_n(f) \to X(f)$ almost surely.
\end{corollary}

\begin{proof}
Decompose the function in its positive and negative parts and apply  
Lemma \ref{lem:munmartingale} to each component. The result 
follows by martingale convergence.
\end{proof}

Let $\CM$ denote the set of continuous functions $f: [0,1]^k \to \BR$,
and for $f, g\in \CM$ let $\|f-g\|_u:=\sup_{x\in [0,1]^k}|f(x)-g(x)|$
denote the supremum norm. 
As it is a standard consequence of the Stone-Weierstrass Theorem that 
$(\CM,\|\cdot\|_u)$ is separable, we leave the proof of the following lemma to the reader.

\begin{lemma}\label{lem:separable}
The space $\CM$ contains a countable and dense (with respect
to the norm $\|\cdot \|_u$) subset $\CM_0$. 
Furthermore, there exists $g\in \CM_0$ such that $1\leq g(x)<\infty$
for every $x\in [0,1]^k.$
\end{lemma}

\begin{lemma}\label{lem:muexists}
Assume that $\lim \int f \id\zeta_n(\omega)$ exists and is finite for all 
$f \in \CM_0$. Then there exists a Radon measure $\zeta(\omega)$ such 
that $\zeta_n(\omega) \to \zeta(\omega)$ weakly. 
\end{lemma}

\begin{proof}
Fix any $\omega$ satisfying the assumption. Since it is fixed, we will 
now again suppress it from our notation.
The proof is based on extending the linear functional 
$l(f) = \lim_n \int f \id \zeta_n$ initially defined for $f \in \CM_0$ to all 
$f\in \CM,$ and then apply Riesz representation theorem.
Let $g\in \CM_0$ be as in Lemma \ref{lem:separable} and note that
$\limsup_n \zeta_n([0,1]^d) = \limsup_n \int \id\zeta_n 
\leq \limsup_n \int g \id\zeta_n = \lim_n \int g \id\zeta_n < \infty$ since 
the limit exists and is finite by assumption. We conclude that
\begin{equation}\label{eq:boundedlimsup}
\limsup_{n \to \infty} \zeta_n([0,1]^d) < \infty.
\end{equation}
Fix $f\in \CM$ and let $(f_i)_{i\geq 1} \subset \CM_0$ be such that 
$\lim_{i \to \infty} \|f_i - f\|_u = 0.$ Then,
\[
\lim_{i,j \to \infty}|l(f_i) - l(f_j)| 
\leq \lim_{i,j \to \infty}\lim_{n \to \infty}\int |f_i - f_j| \id\zeta_n 
\leq \lim_{i,j \to \infty}\|f_i - f_j\|_u 
\limsup_{n \to \infty} \zeta_n([0,1])^d = 0,
\]
where we used that $(f_i)_{i\geq 1}$ converges (and in particular 
is a Cauchy sequence) and \eqref{eq:boundedlimsup}. We conclude 
that $(l(f_i))_{i \geq 1}$ is itself a Cauchy sequence and so converges,
and we can now define $l(f) := \lim_i l(f_i).$ Moreover, we have that 
the limit $\lim_n \int fd\zeta_n$ exists and equals $l(f)$. Indeed,
\begin{eqnarray*}
\lefteqn{\left|\lim_i l(f_i) -  \lim_n \int f \id\zeta_n\right|
\leq \lim_i \left|\lim_n \int f_i \id\zeta_n -\lim_n \int f \id\zeta_n \right| 
}\\
&&\leq \lim_i \lim_n \int |f_i -f| \id\zeta_n 
\leq \lim_i \|f_i - f\|_u\limsup_n \zeta_n([0,1]^d) = 0,
\end{eqnarray*}
again by \eqref{eq:boundedlimsup} and by the fact that 
$f_i \to f$ in the supremum norm.

Next, let $f\in \CM$ be a non-negative function and let  
$(f_i)_{i\geq 1} \subset \CM_0$ be a sequence approximating $f$. Then,
\begin{eqnarray*}
\lefteqn{l(f) = \lim_i l(f_i)
=\lim_i \lim_n \int f_i \id\zeta_n }\\
& & \geq \lim_i \lim_n \int -\|f-f_i \|_u \id\zeta_n
\geq - \lim_i \|f-f_i\|_u \limsup_n\zeta_n([0,1]^d) =0,
\end{eqnarray*}
where we used \eqref{eq:boundedlimsup} one more time.
We conclude that $l$ is a positive linear functional on $\CM$ and so  
by Riesz representation theorem (Theorem 7.2 from \cite{F_1999}) 
there exists a Radon measure $\zeta$ such that
\[
l(f) = \int f \id\zeta = \lim_n \int f \id\zeta_n,
\]
for all $f \in \CM$. We conclude that 
$\zeta_n(\omega) \to \zeta(\omega)$ weakly.
\end{proof}

We will now combine Corollary \ref{coro:weakrv} and Lemma \ref{lem:muexists}
to show the following result.
\begin{proposition} \label{prop:weakconv}
There exists a random measure $\zeta$ on $[0,1]^d$ such that 
for almost every $\omega,$ $\zeta_n(\omega) \to \zeta(\omega)$ weakly.
\end{proposition}
\begin{proof}
Let $\CM_0$ be as in Lemma \ref{lem:separable}.
Corollary \ref{coro:weakrv} implies that for every fixed $f \in \CM_0$
\[
\int f \id\zeta_n = X_n(f) \to X(f)
\]
except possibly for a null-set of $\omega$ that we will denote by $N(f)$. 
Then, Lemma \ref{lem:muexists} implies that for every 
$\omega \notin \cup_{f \in \CM_0}N(f)$ we have that 
$\lim \int f \id\zeta_n(\omega) \to \int f \id\zeta(\omega)$ for every $f \in \CM.$
Since $\cup_{f \in \CM_0}N(f)$ is a countable union of nullsets we
conclude that almost surely, $\zeta_n \to \zeta$ weakly.
\end{proof}

We can now prove Theorem \ref{thm:Vklowdim}.
\begin{proof}[Proof of Theorem \ref{thm:Vklowdim}]
As discussed before, we will use Frostman's Lemma. 
Proposition \ref{prop:weakconv} states that the sequence 
$(\zeta_n)_{n \geq 1}$ of random measures 
defined in \eqref{eq:mundef} converge weakly to a 
measure $\zeta$. In order to apply Frostman's Lemma, it suffices 
to show that with probability bounded away from $0$ uniformly in $r<k-\lambda$, the measure $\zeta$ is supported on $\CV^k$, that 
$0 < \zeta(\CV^k) < \infty$ and that $\zeta$ has finite 
$r-$energy (as explained in \eqref{eq:finiterenergy}). 

We will start by considering the support of $\zeta.$ To that end, 
let $x\in [0,1]^k \setminus \CV^k.$ Then we must have that for some 
$n,$ $x\in [0,1]^k \setminus \CV_n^k.$ Furthermore, since $\CV_n^k$
is a compact set, there exists some $\delta>0$ such that 
$B(x,\delta) \subset [0,1]^k \setminus \CV_n^k \subset 
[0,1]^k \setminus \CV^k.$ We conclude that $x$ is not in the 
support of $\zeta,$ and so we must have that 
$\supp \zeta \subset \CV^k.$

Second, we observe that by Proposition \ref{prop:weakconv} and 
\eqref{eq:boundedlimsup} we have that almost surely, 
$\zeta(\CV^k) = \lim \zeta_n ([0,1]^d) <\infty.$ 

Our next step is to prove that $\zeta(\CV^k) >0$ with positive probability. 
To that end, let
\[
\|\zeta_n\|_{TV} = \int_{[0,1]^k}2^{\lambda n} \indicator(x \in \CV^k_n)\id x 
= 2^{\lambda n} \ell_k(\CV^k_n)
\]
be the total variation of $\zeta_n,$ and note that 
$\zeta(\CV^k) = \|\zeta\|_{TV} = \lim \|\zeta_n\|_{TV}.$ 
In what follows, we will prove that $\BE(\|\zeta\|_{TV}) =1$ and 
therefore that $\|\zeta\|_{TV} >0$ with positive probability. 
For that purpose, note that part $a)$ of Proposition
\ref{prop:list} implies that
\begin{equation}\label{eq:expectedtv}
\BE[\|\zeta_n \|_{TV}] 
= \BE\left[\int_{[0,1]^k}2^{n\lambda}\indicator(x \in \CV^k_n) \id x\right] 
 = \int_{[0,1]^k}2^{n\lambda}\BP(x \in \CV^k_n) \id x =1.
\end{equation}
Then, observe that by Lemma \ref{lem:munmartingale} applied to the constant 
function $f = 1$ we have that $(\|\zeta_n\|_{TV})_{n \geq 1}$ is a martingale.
By standard theory, if we can show that this martingale is bounded in 
$L^2$, it follows that $\BE[\|\zeta \|_{TV}]=\lim_n \BE[\|\zeta_n \|_{TV}]=1.$
Consider therefore
\begin{eqnarray*}
\lefteqn{\BE\left[\|\zeta_n\|_{TV}^2\right] 
= \BE\left[\int_{[0,1]^{k}\times [0,1]^k} 
2^{2\lambda n}\indicator(x \in \CV^k_n)\indicator(y \in \CV^k_n) \id x \id y\right]} \\
&&  = \int_{[0,1]^{k}\times [0,1]^k} 2^{2\lambda n}\BP(x,y \in \CV^k_n) \id x \id y
\leq \int_{[0,1]^{k}\times [0,1]^k} 
\frac{2^{2\lambda n} e^{\lambda C_2} 2^{-2 \lambda n}}{\|x-y\|^\lambda}\id x \id y \\
&&  = e^{\lambda C_2} \int_{[0,1]^k}\int_{[0,1]^k}  
\frac{1}{\|x-y\|^\lambda} \id x\id y, 
\end{eqnarray*}
where we used part $c)$ of Proposition \ref{prop:list} in the inequality. 
Recall the notation $\psi_m$ defined in Section \ref{sec:lincyl}
\begin{eqnarray}\label{eq:spherical}
\lefteqn{\int_{[0,1]^k}\left(\int_{[0,1]^k}  \frac{1}{\|x-y\|^\lambda} 
\id x\right)\id y 
\leq \max_{y\in [0,1]^k}\int_{[0,1]^k}\frac{1}{\|x-y\|^\lambda} \id x} \\
&& \leq \max_{y\in [0,1]^k} \left\{ \int_{\{x : d(x,y) \leq 1\}} 
\frac{1}{\|x-y\|^\lambda} \id x +1 \right\} \nonumber\\
& & = \psi_{k-1}
\int_0^1 r^{-\lambda}r^{k-1}\id r +1
= \frac{\psi_{k-1}}{k-\lambda} +1,\nonumber
\end{eqnarray}
for all $\lambda<k$. We conclude that 
\begin{equation}\label{eq:secondmomenttv}
\BE[\|\zeta_n\|_{TV}^2] 
\leq e^{C_2\lambda}\left(\frac{\psi_{k-1}}{k-\lambda} +1\right).
\end{equation}
In particular, we see that $\sup_n \BE[\|\zeta_n\|_{TV}^2] < \infty$ and so 
(as explained above), $\BE[\|\zeta\|_{TV}]=\BE[\zeta(\CV^k)]=1.$ Thus, 
$\zeta(\CV^k)>0$ with positive probability.

It remains to show that $\zeta$ has finite $r-$energy for every 
$r<k-\lambda.$ For such $r,$ we use part $c)$ of Proposition 
\ref{prop:list} to obtain
\begin{eqnarray} \label{eqn:uniformenergybound}
\lefteqn{\BE[\CI_r(\zeta_n)] 
= \BE\left[\int_{[0,1]^k \times [0,1]^k}
\frac{\id\zeta_n(x)\id\zeta_n(y)}{\|x-y\|^r}\right]}\\
& & = \BE\left[\int_{[0,1]^k \times [0,1]^k} 2^{2\lambda n}
\frac{\indicator(x \in \CV^k_n)\indicator(y \in \CV^k_n)}{\|x-y\|^r}\id x\id y\right]
\nonumber \\
&& = \int_{[0,1]^k \times [0,1]^k} 2^{2\lambda n}
\frac{\BP( x,y \in \CV^k_n)}{\|x-y\|^r} \id x \id y \nonumber \\
&&  \leq e^{\lambda C_2} \int_{[0,1]^k}\int_{[0,1]^k}
\frac{1}{\|x-y\|^{r+\lambda}}\id x \id y 
\leq  e^{\lambda C_2}
\left( \frac{\psi_{k-1}}{k-r-\lambda}+1\right), \nonumber
\end{eqnarray}
where the last inequality follows in the same way as in~\eqref{eq:spherical}.
Furthermore, we observe that 
\begin{eqnarray}\label{eqn:liminfenergy}
\lefteqn{ \CI_r(\zeta)
=\int_{[0,1]^k \times [0,1]^k}\frac{\id\zeta(x)\id\zeta(y)}{\|x-y\|^r}}\\
& & =\lim_{M \to \infty}
\int_{[0,1]^k \times [0,1]^k}\left(\frac{1}{\|x-y\|^r} \wedge M\right)
\id\zeta(x)\id\zeta(y) \nonumber\\
& & =\lim_{M \to \infty} \liminf_{n \to \infty}
\int_{[0,1]^k \times [0,1]^k}\left(\frac{1}{\|x-y\|^r} \wedge M\right)
\id\zeta_n(x)\id\zeta_n(y) \nonumber \\
& & \leq \liminf_{n \to \infty}
\int_{[0,1]^k \times [0,1]^k}\frac{1}{\|x-y\|^r} \id\zeta_n(x)\id\zeta_n(y)
=\liminf_{n \to \infty} \CI_r(\zeta_n)
\nonumber
\end{eqnarray}
where we used the monotone convergence theorem in the second equality 
and the fact that $\zeta_n \times \zeta_n \to \zeta \times \zeta$ weakly 
(since $\zeta_n \to \zeta$ weakly) in the third equality. 
We combine \eqref{eqn:uniformenergybound} and \eqref{eqn:liminfenergy}
to conclude that 
\[
\BE[\CI_r(\zeta)]\leq \BE[\liminf_{n \to \infty} \CI_r(\zeta_n)]
\leq \liminf_{n \to \infty} \BE[\CI_r(\zeta_n)]<\infty,
\]
and so $\CI_r(\zeta(\omega)) <\infty $ almost surely.

Since both $\CI_r(\zeta) <\infty $ and $\zeta(\CV^k)<\infty$
almost surely, we see that
\[
\BP(\exists \zeta: 0<\zeta(\CV^k)<\infty,\ \CI_r(\zeta)<\infty)
=\BP(\exists \zeta: \zeta(\CV^k)>0).
\]
Furthermore, for any $\omega$ such that there exists $\zeta=\zeta(\omega)$
with $0<\zeta(\CV^k)<\infty$ and $\CI_r(\zeta)<\infty,$ it follows from 
Frostman's lemma that $\dim_\CH(\CV^k) \geq r$ for every 
$r < k-\lambda.$ From this it then follows that for any such $\omega$, we have 
$\dim_\CH(\CV^k) \geq k-\lambda,$ and so we conclude that 
\[
\BP(\dim_\CH(\CV^k) \geq k-\lambda) \geq \BP(\exists \zeta: \zeta(\CV^k)>0)>0.
\]
\end{proof}

\noindent
\begin{remark}
In the proof of the previous proposition we showed that the limiting 
measure has positive mass (i.e. $\zeta(\CV^k)>0$) with positive probability,
and from this we concluded that also $\BP(\dim_\CH(\V^k)\geq k-\lambda)>0$.
Using Theorem \ref{thm:Vkuppdim} we then conclude that 
$\BP(\dim_\CH(\V^k)= k-\lambda)>0$. However, it is not far fetched to 
suspect that in fact $\dim_\CH(\V^k)= k-\lambda$ as soon as 
$\CV^k\neq \emptyset$. That is, whenever the fractal survives 
within $[0,1]^k \times \{0\}^{d-k}$ it must have dimension $k-\lambda.$
Indeed, a similar result holds for the much simpler case of the 
 Mandelbrot fractal percolation model (see \cite{CCD_1988}). While we 
are not currently able to prove this stronger statement for our model, 
we can provide an easy explicit lower bound on the probability of 
$\zeta(\CV^k)>0$ (and therefore also on $\BP(\dim_\CH(\V^k)\geq k-\lambda)$).
\end{remark}

\begin{proposition} \label{prop:explicitlowerbound}
We have that 
\[
\BP(\zeta(\CV^k)>0) \geq 
\frac{(k-\lambda)e^{-C_2 \lambda}}
{\psi_{k-1} +k -\lambda},
\]
where $C_2$ is defined in \eqref{eqn:constants}.
\end{proposition}

\begin{proof}
Let $0<\alpha<1.$ By Proposition \ref{prop:weakconv} we have that
$\zeta_n \to \zeta,$ and so $\zeta_{n_i} \to \zeta$ along any subsequence. 
Therefore, if $\|\zeta_{n}(\omega)\|_{TV} \geq\alpha$ for infinitely many $n,$ 
there exists some subsequence $(n_i)_{i\geq 1}$ (depending on 
$\omega$) such that 
$\zeta(\CV^k)= \|\zeta\|_{TV} = \lim \|\zeta_{n_i}\|_{TV} \geq \alpha$.

Next, we use the Paley-Zygmund inequality to see that for every $n\geq1,$
\begin{eqnarray*}
\lefteqn{\BP(\|\zeta_n\|_{TV}\geq \alpha)  
= \BP(\|\zeta_n\|_{TV} \geq \alpha \BE[\|\zeta_n\|_{TV}])}\\
& & \geq (1 - \alpha)^2 \frac{\BE[\|\zeta_n\|_{TV}]^2}{\BE[\|\zeta_n\|_{TV}^2]} 
\geq \frac{(1-\alpha)^2(k-\lambda)e^{-C_2 \lambda}}{\psi_{k-1}+k -\lambda} >0 
\end{eqnarray*}
where we used \eqref{eq:expectedtv} in the first equality, and
\eqref{eq:expectedtv} and \eqref{eq:secondmomenttv} in the second inequality. 
By the reverse Fatou Lemma we then see that 
\begin{eqnarray*}
\lefteqn{\frac{(1-\alpha)^2(k-\lambda)e^{-C_2 \lambda}}{\psi_{k-1}+k -\lambda} 
\leq \limsup_n \BE[\indicator(\|\zeta_{n}\| \geq \alpha)]}\\
& & \leq \BE[\limsup \indicator(\|\zeta_{n}\| \geq \alpha)]
= \BP(\|\zeta_{n}\| > \alpha \text{ i.o.}),
\end{eqnarray*}
and so we conclude, using the discussion in the beginning of the proof, that 
\[
\BP(\zeta(\CV^k)\geq \alpha)
\geq \frac{(1-\alpha)^2(k-\lambda)e^{-C_2 \lambda}}{\psi_{k-1}+k -\lambda}.
\]
Then finally, we see that 
\[
\BP(\zeta(\CV^k) >0)
= \sup_{\alpha>0}\BP(\zeta(\CV^k) \geq \alpha)
\geq \sup_{\alpha>0}
\frac{(1-\alpha)^2(k-\lambda)e^{-C_2 \lambda}}{\psi_{k-1} +k -\lambda}
=\frac{(k-\lambda)e^{-C_2 \lambda}}{\psi_{k-1} +k -\lambda}.
\]
\end{proof}

We can now prove Theorem \ref{thm:Hausdorff}. 
\begin{proof}[Proof of Theorem \ref{thm:Hausdorff}]

Theorem \ref{thm:Vkuppdim} states that $\dim_\CH(\V \cap H_k)\leq k-\lambda$
almost surely. 

Furthermore, the event $F:=\{\dim_\CH(\V \cap H_k)\geq k-\lambda\}$
is clearly shift-invariant so that by Lemma \ref{lem:zeroone},
$\BP(F) \in \{0,1\},$ and since by Theorem 
\ref{thm:Vklowdim} we have that $\BP(F)>0$ we conclude that 
$\dim_\CH(\V \cap H_k)\geq k-\lambda$ almost surely. 
\end{proof}

\section{The invariant measure on the space of lines}\label{sec:intensitymeasure}
The aim of this section is to derive the representation of the invariant 
measure $\nu_d$ given in \eqref{eqn:linmeasrep}. We will do this starting 
from a third representation given in \cite{K_2017}.

We first recall the parametrization we use for lines in this paper. 
We write a line $L\in {\mathrm A}(d,1)$ as $L=L(a,p)=\{a t +p\,:\,t\in \R\}$ 
where $a=(a_1,\ldots,a_{d-1},1)\in \R^{d-1}\times \{1\}$ and 
$p=(p_1,\ldots,p_{d-1},0)\in \R^{d-1}\times \{0\}$.  We have the following 
result.
\begin{theorem}\label{thm:altrepresentation}
For $d\ge 2$, the Haar measure on $\mathrm{A}(d,1) $ is given by
\begin{equation}\label{eqn:altrepresentation}
\id \nu_d(L)=\id \nu_d( a_1, a_2,...,a_{d-1}, p_1,p_2,...,p_{d-1})
= \frac{\Upsilon_d}{ \|a\|^{d+1} } 
\id a_1 \id a_2 ... \id a_{d-1} \id p_1 ...\id p_{d-1}.
\end{equation}
\end{theorem}

Before the proof of Theorem~\ref{thm:altrepresentation} we describe a 
slightly different parametrization of a line $L\in {\mathrm A}(d,1)$. 
Let (as in Section \ref{sec:lincyl} )
$\uhs=\partial B(o,1) \cap \{x\in \R^d\,:\,x_d>0\}$ denote the upper hemisphere.
We will write $\alpha\in \uhs$ as $\alpha=(\alpha_1,\ldots,\alpha_d)$ 
where $\alpha_d=(1-\sum_{i=1}^{d-1}\alpha_i^2)^{1/2}$. Any line 
$L\in {\mathrm A}(d,1)\setminus \tilde{{\mathrm A}}(d,1)$ can be 
uniquely written as $L=L(\alpha,p)=\{\alpha t + p\,:\,t\in \R\}$ where 
$\alpha \in  \uhs$ and $p$ is as above.  In this parametrization, $p$ 
is again the intersection between $L(\alpha,p)$ and $H_{d-1}$, while 
$\alpha$ describes the direction of the line.

According to Equation $(1.7)$ in \cite{K_2017},  the invariant measure 
$\id \nu_d(L)$ using the parametrization $(\alpha,p)$ is given by
\begin{equation}\label{eqn:kendallrepr}
\id \nu_d(\alpha,p)
=\Upsilon_d \sin(\theta) \id p_{1}\cdots \id p_{d-1} 
\indicator(\alpha \in \uhs)\ell_{d-1}(\id \alpha),
\end{equation}
where 
$\theta=\theta(\alpha)$ is defined as the angle between $L(\alpha,p)$ 
and $H_{d-1}$. We note that the normalization used in \cite{K_2017}
is different (see in particular $(1.3)$ in that paper) from the one 
we use here, and this must be taken into account when arriving 
at \eqref{eqn:kendallrepr}.

Note that 
\begin{equation}\label{eqn:sinalpha}
\sin(\theta)=\left(1-\sum_{i=1}^{d-1}\alpha_i^2\right)^{1/2}.
\end{equation}
Moreover, according to Equation A.3 in \cite{ABR_1992}, we have that
\begin{equation}\label{eqn:volhemisphere}
\indicator(\alpha \in \uhs)\id \ell_{d-1}( \alpha) 
= \frac{\id \alpha_1\cdots \id \alpha_{d-1}}
{(1-\sum_{i=1}^{d-1} \alpha_i^2)^{1/2}}.
\end{equation}
By~\eqref{eqn:kendallrepr},~\eqref{eqn:sinalpha} and~\eqref{eqn:volhemisphere} 
we get
\begin{equation}\label{eqn:kendallrepr2}
\id \nu_d(\alpha,p)
=\Upsilon_d \id \alpha_1\cdots \id \alpha_{d-1} \id p_1\cdots \id p_{d-1}.
\end{equation}
We can now prove Theorem \ref{thm:altrepresentation}.

\begin{proof}[Proof of Theorem \ref{thm:altrepresentation}]
Note that $a\mapsto a /\|a\|$ is a bijection between 
$\R^{d-1}\times \{1\}$ and the hemisphere $\uhs$. To go between 
the parametrization $(a,p)$ and the parametrization $(\alpha,p)$ we simply 
let $\alpha= a /\|a\|$. Then $L(a,p)$ is the same line as $L(\alpha,p)$. 
From~\eqref{eqn:kendallrepr2} we see that to 
obtain~\eqref{eqn:altrepresentation} from~\eqref{eqn:kendallrepr}, it suffices 
to show that the determinant of the Jacobian corresponding to this change of 
coordinates is given by $\|a\|^{-(d+1)}$.

We have that
$$
\alpha_i=\frac{a_i}{\|a\|}=\frac{a_i}{\sqrt{a_1^2+\ldots+a_{d-1}^2+1}},
$$
for $i=1,\ldots ,d$,
where we recall that $a_d=1$. 
Let $M$ be the Jacobian corresponding to the change of variables from $a$
to $\alpha$. That is, $M$ is the $(d-1)\times (d-1)$ square matrix with entries 
$(m_{i,j})_{1\le i,j\le d-1}$ where 
$m_{i,j}=\frac{\partial \alpha_i}{\partial a_j}$. A straightforward 
calculation shows that 
\[
m_{i,i}=\frac{\partial \alpha_i}{\partial a_i}=\frac{\|a\|^2-a_i^2}{\|a\|^3},
\textrm{ and that } \
m_{i,j}=\frac{\partial \alpha_i}{\partial a_j} =\frac{-a_i a_j}{\|a\|^3}
\textrm{ whenever } i \neq j.
\]
Let $u=-(a_1,\ldots,a_{d-1})$ 
so that 
\[
M=\frac{1}{\|a\|^3}(\|a\|^2 I_{d-1}-u^T u).
\]
We then see that by the matrix determinant lemma,
\begin{eqnarray*}
\lefteqn{{\rm det}(M)
=\frac{1}{\|a\|^{3(d-1)}} {\rm det}(\|a\|^2 I_{d-1}-u^T u)}\\
& & =\frac{1}{\|a\|^{3(d-1)}}
(1-u (\|a\|^2 I_{d-1})^{-1} u^T ){\rm det}(\|a\|^2 I_{d-1}) \\
& & =\frac{1}{\|a\|^{d-1}}\left(1-\frac{\sum_{k=1}^{d-1}a_k^2}{\|a\|^2}\right)
=\frac{1}{\|a\|^{d-1}}\frac{1}{\|a\|^2}=\frac{1}{\|a\|^{d+1}},
\end{eqnarray*}
and the statement follows.
\end{proof}

\section{Cylinders intersecting subspaces} \label{sec:cylintersect} 
The aim of this section is to study the intersection 
${\mathcal C}(\omega)\cap H_k$ for $k\le d-1$.
Our first result of this section describes the shape of the intersection
$\mathfrak{c}(L,r)\cap H_k$ for $L\in \Line_ {H_k^r}$. Recall the notation introduced in Section ~\ref{sec:lincyl}.

\begin{lemma}\label{lem:cylinderellipsoid}
Fix $k\in \{1,\ldots,d-1\}$. For $L=L(\ba,\bp) \in \Line_{H_k^r}$ 
the set
\[
\mathfrak{c}(L,r)\cap H_k
\]
is an ellipsoid defined by the equation
\[
(\cent_k-x) \left( I_{k\times k} - \|\ba \|^{-2} A \right) (\cent_k-x)^T 
< r^2 -\|\puk\|^2 +\frac{\langle \auk,\puk \rangle^2}{1+\|\auk \|^2},
\]
where 
\begin{equation}\label{eqn:centerequation}
\cent_k = \pdk- \frac{\langle \auk,\puk \rangle^2}{1+\|\auk \|^2} \adk
\end{equation}
and
\[
A = \adk^T \adk = (a_{(k),i} a_{(k),j})_{1 \leq i,j \leq k},
\]
is the outer product of the vector $\adk$ with itself. More concretely, 
the ellipsoid has one major axis given by
\[
\frac{\| a\|}{\sqrt{1 + \|\auk\|^2}} \left( r^2 -\|\puk \|^2 
+ \frac{\langle \auk,\puk \rangle^2}{1+ \|\auk\|^2} \right)^{1/2}
\]
which extends in the direction $\adk/\|\adk\|$ and the lengths of
the remaining $k-1$ axes are given by
\[
\left( r^2 -\|\puk \|^2 
+ \frac{\langle \auk,\puk \rangle^2}{1+ \|\auk\|^2} \right)^{1/2}
\]
and these extends in the directions orthogonal to $\adk/\|\adk\|$ in $H_k$.
\end{lemma}

\begin{proof}
Recall that we write $L(a,p)=\{at + p\,:\,t\in \R\}$, see Section~\ref{sec:lincyl}. Let $x_k\in \R^k$, and for convenience, write $x=(x_k,0,\ldots,0)\in \R^d.$ 
The squared distance from the point $x$ to the point $L_t:=at +p$ is 
given by 
\[
f(t)=\|L_t-x\|^2=\|at+p-x\|^2
=\|a\|^2t^2+\|p-x\|^2+2t\langle a,p-x\rangle,
\]
and since $f'(t)=2t\|a\|^2+2\langle a,p-x\rangle$, we see that $f(t)$ is 
minimised at $t^*=-\frac{\langle a,p-x\rangle}{\|a\|^2}.$ Furthermore, 
\begin{eqnarray*}
\lefteqn{f(t^*)=\frac{\langle a,p-x\rangle^2}{\|a\|^2}+\|p-x\|^2
-2\langle a,p-x\rangle\frac{\langle a,p-x\rangle}{\|a\|^2}}\\
& & =\|p-x\|^2-\frac{\langle a,p-x\rangle^2}{\|a\|^2}
=\|\pdk-x_k\|^2+\|\puk\|^2-\frac{\left(\langle \adk,\pdk-x_k\rangle
+\langle \auk,\puk\rangle\right)^2}{\|a\|^2} \\
& & =\|\pdk-x_k\|^2+\|\puk\|^2-\frac{\langle \adk,\pdk-x_k\rangle^2
+\langle \auk,\puk\rangle^2
+2\langle \adk,\pdk-x_k\rangle \langle \auk,\puk\rangle}{\|a\|^2}.
\end{eqnarray*}
Let $A$ be as above, and let 
\[
q_k=(q_{k,1},\ldots,q_{k,k})=-\frac{\langle \auk,\puk\rangle}{\|\auk\|^2+1} \adk.
\]
Since $I-A/\|a\|^2$ is symmetric, we have that
\begin{eqnarray} \label{eqn:threeterms}
\lefteqn{(\pdk+q_k-x_k)\left(I-\frac{1}{\|a\|^2}A\right)(\pdk+q_k-x_k)^T}\\
& & =(\pdk-x_k)\left(I-\frac{1}{\|a\|^2}A\right)(\pdk-x_k)^T \nonumber \\
& & \ \ \ \ +2q_k\left(I-\frac{1}{\|a\|^2}A\right)(\pdk-x_k)^T
+q_k\left(I-\frac{1}{\|a\|^2}A\right)q_k^T, \nonumber
\end{eqnarray}
and furthermore,
\begin{eqnarray*}
\lefteqn{(\pdk-x_k)A(\pdk-x_k)^T}\\
& & =(\pdk-x_k)\adk^T \adk(\pdk-x_k)^T
=\langle \adk,\pdk-x_k \rangle^2.
\end{eqnarray*}
Therefore, the first term on the right hand side of \eqref{eqn:threeterms} 
equals
\[
(\pdk-x_k)\left(I-\frac{1}{\|a\|^2}A\right)(\pdk-x_k)^T
=\|\pdk-x_k\|^2-\frac{\langle \adk,\pdk-x_k \rangle^2}{\|a\|^2}.
\]
Continuing, we note that since $\adk A=\adk \adk^T \adk=\|\adk\|^2 \adk$ 
we have that
\begin{eqnarray*} 
\lefteqn{q_k\left(I-\frac{1}{\|a\|^2}A\right)(\pdk-x_k)^T}\\
& & =-\frac{\langle \auk,\puk\rangle}{\|\auk\|^2+1}
\left(\adk-\frac{1}{\|a\|^2}\|\adk\|^2\adk\right)(\pdk-x_k)^T \\
& & =-\frac{\langle \auk,\puk\rangle}{\|\auk\|^2+1}
\frac{\|a\|^2-\|\adk\|^2}{\|a\|^2}\langle \adk,\pdk-x_k \rangle
=-\frac{\langle \auk,\puk\rangle}{\|a\|^2}\langle \adk,\pdk-x_k \rangle.
\end{eqnarray*}
For the third term note that in the same way,
\begin{eqnarray*} 
\lefteqn{q_k\left(I-\frac{1}{\|a\|^2}A\right)q_k^T
=-\frac{\langle \auk,\puk\rangle}{\|a\|^2}\langle \adk,q_k\rangle}\\
& & =-\frac{\langle \auk,\puk\rangle}{\|a\|^2}\langle \adk,
-\frac{\langle \auk,\puk\rangle}{\|\auk\|^2+1} \adk\rangle \\
& & =\frac{\langle \auk,\puk\rangle^2}{\|a\|^2(\|\auk\|^2+1)}
\langle \adk,\adk\rangle
=\frac{\langle \auk,\puk\rangle^2 \|\adk\|^2}{\|a\|^2(\|\auk\|^2+1)}.
\end{eqnarray*}
Inserting all of the above into \eqref{eqn:threeterms} we arrive at 
\begin{eqnarray}
\lefteqn{(\pdk+q_k-x_k)\left(I-\frac{1}{\|a\|^2}A\right)(\pdk+q_k-x_k)^T}\\
& & =\|\pdk-x_k\|^2-\frac{\langle \adk,\pdk-x_k \rangle^2}{\|a\|^2}
-2\frac{\langle \auk,\puk\rangle}{\|a\|^2}\langle \adk,\pdk-x_k \rangle
+\frac{\langle \auk,\puk\rangle^2 \|\adk\|^2}{\|a\|^2(\|\auk\|^2+1)} \nonumber\\
& & =\|\pdk-x_k\|^2+\|\puk\|^2-\frac{\langle \adk,\pdk-x_k\rangle^2
+\langle \auk,\puk\rangle^2
+2\langle \adk,\pdk-x_k\rangle \langle \auk,\puk\rangle}{\|a\|^2} \nonumber \\
\ \ \ \ & & -\|\puk\|^2+\frac{\langle \auk,\puk\rangle^2 }{\|a\|^2}
\left(\frac{\|\adk\|^2}{\|\auk\|^2+1}+1\right) \nonumber \\
& & =f(t^*)-\|\puk\|^2+\frac{\langle \auk,\puk\rangle^2 }{\|a\|^2}
\left(\frac{\|a\|^2}{\|\auk\|^2+1}\right)
=f(t^*)-\|\puk\|^2+\frac{\langle \auk,\puk\rangle^2 }{\|\auk\|^2+1}.
\nonumber
\end{eqnarray}
The intersection of $\mathfrak{c}(L,r)$ with $H_k$ is given by 
those $x$ such that $t^*=t^*(x)$ satisfies the inequality
$f(t^*)< r^2$. The boundary is therefore given by $f(t^*)=r^2$ 
or equivalently
\begin{equation} \label{eqn:ellipseq}
(\pdk+q_k-x_k)\left(I-\frac{1}{\|a\|^2}A\right)(\pdk+q_k-x_k)^T
=r^2-\|\puk\|^2+\frac{\langle \auk,\puk\rangle^2 }{\|\auk\|^2+1}.
\end{equation}

It is well known that  
$(x-v)B(x-v)^T=r^2$ defines an ellipsoid centred at $v$ and with 
axis along the eigenvectors of $B$ whenever $B$ is a positive definite matrix.
Furthermore, the length of these axes are given by $r$ times one over 
the square root of the corresponding eigenvalues.

Clearly we have that for any $x\in \BR^k,$
\[
x\left(I-\frac{1}{\|a\|^2}A\right)x^T
=\|x\|^2-\frac{x \adk^T \adk x^T}{\|a\|^2}
=\|x\|^2-\frac{\langle x, \adk \rangle^2 }{\|a\|^2}
\geq \|x\|^2\left(1-\frac{\|\adk\|^2}{\|a\|^2}\right),
\]
and so $\left(I-\frac{1}{\|a\|^2}A\right)$ is a positive definite matrix.
Furthermore, the center of the ellipsoid is given by 
\[
\pdk+q_k=\pdk-\frac{\langle \auk,\puk\rangle}{\|\auk\|^2+1} \adk.
\]

It remains to determine the eigenvectors and the corresponding 
eigenvalues. To that end, observe that 
\[
\left(I-\frac{1}{\|a\|^2}A\right)\adk^T
=\adk^T-\frac{\adk^T \adk}{\|a\|^2}\adk^T
=\frac{\|a\|^2-\|\adk\|^2}{\|a\|^2}\adk^T
=\frac{\|\auk\|^2+1}{\|a\|^2}\adk^T
\]
and so $u_1=\adk$ is an eigenvector corresponding to the eigenvalue 
$\frac{\|\auk\|^2+1}{\|a\|^2}$.
Furthermore, let 
\[
u_l=(-a_{(k),l},0,\ldots,0,a_{(k),1},0,\ldots,0)
=(-a_{(k),l}e^k_1+a_{(k),1}e^k_l),
\]
where $e^k_l$ is the vector of length $k$ consisting of all zeros except 
entry number $l$ which equals one. We get that 
\[
Au_l^T=(-a_{(k),1} a_{(k),1}a_{(k),l}+a_{(k),1} a_{(k),l} a_{(k),1},
-a_{(k),2} a_{(k),1}a_{(k),l}+a_{(k),2} a_{(k),l} a_{(k),1},\ldots)^T
=0
\]
and so for every $u_l,$ $l=2,\ldots,k$ we have that $Au_l^T=0$ whence
\[
\left(I-\frac{1}{\|a\|^2}A\right)u_l^T=u_l^T,
\]
and so $u_l$ is an eigenvector corresponding to the eigenvalue 1.

Comparing this to \eqref{eqn:ellipseq} we see that if 
\[
\tilde{r}^2:=r^2-\|\puk\|^2+\frac{\langle \auk,\puk\rangle^2 }{\|\auk\|^2+1},
\]
then the lengths of the axes are given by 
\[
\left(\frac{\|a\|}{\sqrt{\|\auk\|^2+1}}
\tilde{r},\tilde{r},\ldots,\tilde{r}\right).
\]
\end{proof}

As an immediate corollary we obtain the volume and the diameter of the 
ellipsoid.
\begin{corollary}\label{corr:ellipsoidvolume}
Let $L \in \Line_{H_K^r}$. The volume of the ellipsoid $E_k(L,r)$ is given by
\[
\vol(E_k(L,r)) = 
\frac{\psi_{k+1}}{2 \pi}
\frac{\|a\|}{\sqrt{1+\|\auk\|^2}} 
\left( r^2-\|\puk\|^2 
+\frac{\langle \auk,\puk \rangle^2}{ 1+\|\auk\|^2} \right)^{k/2}.
\]
Moreover, the diameter is given by
\[
\diam(E_k(L,r))=2 \frac{ \|a\|}{\sqrt{1+\|\auk\|^2}} 
\left(r^2-\|\puk\|^2
+\frac{\langle\auk,\puk\rangle^2}{1+\|\auk\|^2} \right)^{1/2}.
\]
\end{corollary}
\begin{proof}
The proof is immediate since the volume of the ellipse
is given by $\ell_k(B^k(o,1)) \prod_{i=1}^k l_i,$ where $l_i$ is the 
length of axis number $i.$ Furthermore from Section \ref{sec:gennotation}
we have that $\ell_k(B^k(o,1))=\frac{\psi_{k+1}}{2 \pi}.$
Finally, the diameter is immediate from 
Lemma \ref{lem:cylinderellipsoid}.
\end{proof}

\subsection{The induced ellipsoid models} \label{sec:inducedmodel}
The aim of this subsection is to prove Theorem \ref{thm:cylellips},
i.e.\ to show that the random set 
${\mathcal C}\cap H_k$ can be generated (in law) using a Poisson 
process of ellipsoids in $\R^k$. Recall therefore the notation of 
Section \ref{sec:ellipsoids}.

Note that it follows from Lemma~\ref{lem:cylinderellipsoid} that for any 
$L=L(\adk,\auk,\pdk,\puk)\in \Line_{H_k^r}$ and $\pdk,$
\begin{equation}\label{eqn:pkindep}
E_k(L(\adk,\auk,\pdk,\puk),r)_o=E_k(L(\adk,\auk,{\bf 0}_k,\puk),r)_o,
\end{equation}
where ${\bf 0}_k$ is the zero-vector of length $k$.
For $k\in \{1,\ldots d-1\}$ and $r\in (0,1]$, define the measure 
$\mu_{k,r}$ on  $\mathfrak{E}^k$ by letting
\begin{equation}\label{eqn:ellipsmeasdef1}
\mu_{k,r}(\mathbf{E}) = \nu_d(L:E_k(L,r) \in \mathbf{E}),
\end{equation}
for any measurable ${\mathbf E} \subset \mathfrak{E}^k.$ 
According to Lemma~\ref{lem:cylinderellipsoid}, if 
$L=L(\adk,\auk,\pdk,\puk)\in {\mathrm A}(d,1)$, the cylinder $\mathfrak{c}(L,r)$ 
intersects $H_k$ if and only if 
\begin{equation}\label{eqn:hitcond}
\| \puk \|^2 -\frac{ \langle \auk,\puk \rangle^2}{1+\|\auk\|^2} < r^2.
\end{equation}
Note that when $k=d-1$ we have that $\puk$ and $\auk$ are empty,
and so the above expression is not defined. However, in this case
every cylinder intersects $H_{d-1}$ and therefore we simply interpret the
the condition to always be satisfied. (Of course, there is a
technical issue with lines parallel to the subspace $H_{d-1}$, but as explained
in Section \ref{sec:lincyl} these can be ignored).

By using \eqref{eqn:linmeasrep}, we see that we can write 
\eqref{eqn:ellipsmeasdef1} in the following way
\begin{equation}\label{eqn:ellipsintegral}
\mu_{k,r}(\mathbf{E})=\Upsilon_d \int_{L\,:\,E_k(L,r)\in {\mathbf E}} 
\indicator \left(\|\puk\|^2 
- \frac{\langle \auk, \puk \rangle^2 }{1 + \|\auk\|^2} < r^2\right) 
\frac{1}{\|a\|^{d+1}} \id \adk \id \auk \id \pdk \id \puk. 
\end{equation}
The measure $\mu_{k,r}$ is a measure on ellipsoids induced by the 
cylinder process with fixed radius $r.$ For the scale invariant 
process we define
\begin{equation}\label{eqn:fractalellipsmeasdef}
\mu_{k}(\mathbf{E}) =\int_{0}^1 \mu_{k,r}(\mathbf{E}) \id \varrho_s( r).
\end{equation}
As noted in Section \ref{sec:ellipsoids}, an ellipsoid $E$ is uniquely 
determined by the pair $(E_o,\cent(E))$. We will show that $\mu_k$ and 
$\mu_{k,r}$ can be written as product measures, where one factor is a 
measure on $\mathfrak{E}^k_o$ (corresponding to shapes of ellipsoids) and 
the other factor is a measure on $\R^k$ (corresponding to centres 
of ellipsoids).

To this end, we define the measures $\xi_{k,r}$ and $\xi_k$ on 
$\mathfrak{E}^k_o$ by 
\begin{equation}\label{eqn:xikrdef}
\xi_{k,r}(\cdot) = \Upsilon_d \int_{(E_k(L,r))_o\in \cdot } 
\indicator\left(\|\puk\|^2 - \frac{\langle \auk, \puk \rangle^2 }
{1 + \|\auk\|^2} 
< r^2\right)  \frac{1}{\|a\|^{d+1}} \id \adk \id \auk \id \puk
\end{equation}
and
\begin{equation}\label{eqn:xikdef}
\xi_{k}(\cdot)=\int_0^1 \xi_{k,r} (\cdot)\id \varrho_s ( r).
\end{equation}
Observe that $\mu_{k,r}$ depends on $d$ through the constant $\Upsilon_d$
and through  $\|a\|^{-(d+1)}$ in the integrand.
Furthermore, $\mu_k$ depends on $d$ further through the use of the 
measure $\varrho_s$ (see \eqref{eqn:scaleinvariantmeasure}). 

\noindent

We will now turn to the proof of Theorem \ref{thm:cylellips}. 
Let $\omega$ be as before (i.e. chosen according to 
${\mathbb P}_{\lambda}$) and define
\[
\omega_e:=\sum_{(L,r)\in \omega} \delta_{(\cent(E_k(L,r)),E_k(L,r)_o)}.
\]

\begin{proof}[Proof of Theorem \ref{thm:cylellips}]
We will only prove \eqref{eqn:measdecomp2}, as \eqref{eqn:measdecomp1}
follows in the same way by considering a Poisson cylinder process with
fixed radius $r$ (in place of the full fractal model).
We therefore need to prove that $\omega_e$ is a Poisson processes 
with intensity measure given by \eqref{eqn:measdecomp1}, 
and we start by showing that $\omega_e$ is a Poisson process. Suppose 
therefore that ${\mathbf E}_1, {\mathbf E}_2 \subset \FE^k=\BR^k \times \FE^k_o$ 
are disjoint collections of ellipsoids. Then, $E_k^{-1}({\mathbf E}_1)$ 
and $E_k^{-1}({\mathbf E}_2)$ are disjoint subsets of ${\mathrm A}(d,1)\times (0,1]$, 
and since $\omega_e({\mathbf E}_i)=\omega(E_k^{-1}({\mathbf E}_i))$
we see that $\omega_e$ is indeed a Poisson process on $\FE^k$.
Furthermore, the mean of $\omega_e({\mathbf E}_i)$ equals 
$\lambda \mu_k({\mathbf E}_i)$.

Our next step is to prove that $\mu_k$ equals
$\ell_k \times \xi_{k,r}.$
Consider any measurable sets $A\subset \BR^k$ and 
${\mathbf E}_o \subset \FE_o^k,$ and let 
\[
{\mathbf E}(A,{\mathbf E}_o)
=\{E\in \FE^k: E=(x,E_o) \textrm{ for some } x\in A
\textrm{ and } E_o \in {\mathbf E}_o\}\subset \FE^k.
\]
The statement follows if we can show that for any $A$ and 
${\mathbf E}_o$ as above we have that
\[
\mu_k({\mathbf E}(A,{\mathbf E}_o))
= \ell_k(A)\xi_k({\mathbf E}_o).
\]
To this end, observe that by \eqref{eqn:ellipsintegral} and 
\eqref{eqn:fractalellipsmeasdef},
\begin{eqnarray*}
\lefteqn{\mu_k({\mathbf E}(A,{\mathbf E}_o))
=\int_{0}^1\nu_d(L:\,\cent(E_k(L,r))\in A,\,E_k(L,r)_o\in {\mathbf E}_o)
\id \varrho_s(r)}\\
& & =\Upsilon_d \int_{ \substack{ E_k(L,r)_o\in {\mathbf E}_o \\ 
\mathrm{cent}(E_k(L,r) )\in A \\0<r<1} } \indicator
\left(\|\puk\|^2 - \frac{\langle \auk, \puk \rangle^2 }{1 + \|\auk\|^2} 
< r^2\right) \\
& & \hspace{50mm} \times
\frac{1}{\|a\|^{d+1}}\id \adk \id \auk \id \pdk\id \puk \id \varrho_s(r) \\
& & 
=\Upsilon_d \int_{\substack{E_k(L,r)_o\in {\mathbf E}_o \\0<r<1}} \indicator
\left(\|\puk\|^2-\frac{\langle \auk, \puk \rangle^2 }{1 + \|\auk\|^2} 
< r^2\right)  \int_{ \mathrm{cent}(E_k(L,r) )\in A} \id \pdk \\ & & 
\hspace{50mm} \times \frac{1}{\|a\|^{d+1}}\id \adk \id \auk \id \puk \id \varrho_s( r) \\
&& =\ell_k(A)\xi_k({\mathbf E}_o),
\end{eqnarray*}
where we used~\eqref{eqn:pkindep} in the third equality and~\eqref{eqn:centerequation} in the fourth equality.
\end{proof}

Our next step is to couple the fractal ellipsoid model with a fractal ball model. This fractal ball model will be of use in the upcoming sections, and in particular for the proof of Theorem \ref{thm:connectivitytranstition}. Informally, for any 
fixed ellipsoid in the ellipsoid process we shall replace it with a 
ball centred in the same point and with the same diameter as the 
replaced ellipsoid. More precisely, we will consider 
\[
\omega_b:=\sum_{(L,r)\in \omega} \delta_{(\cent(E_k(L,r)),\diam(E_k(L,r))/2)},
\]
and then let 
\[
\CV_b:=\BR^k \setminus \bigcup_{(x,R)\in \omega_b} B^k(x,R),
\]
where we recall the notation $B^k(x,R)$ from Section~\ref{sec:gennotation}. The fact that the ball $B^k(x,R)$ is closed (see Section \ref{sec:gennotation}) 
rather than open will not be important.
Obviously it follows that 
\[
\CV_b \subset \CV \cap H_k,
\]
since every ellipsoid is replaced by a larger ball.
Let 
\begin{equation}\label{eqn:gdefin}
g(\cdot)=\xi_k(E\,:\diam(E)/2 \in \cdot).
\end{equation}
We have the following theorem which is an analogue of Theorem 
\ref{thm:cylellips}. Since the proof is almost identical, we only 
give the statement. The reason for using $\diam(E)/2$ (rather than just
$\diam(E)$) in the definition of $g,$ stems from the fact that it 
is customary to specify a ball by using the radius rather than the diameter.

\begin{theorem}\label{thm:ballfrac}
For $k\in \{1,\ldots,d-1\}$, $\omega_b$ is a Poisson process on 
$\BR^k \times \BR_+$ with intensity measure 
\[
\lambda \ell_k \times g.
\]
\end{theorem}


\section{Analysis of the ellipsoid model} \label{sec:anaofellips}
The purpose of this section is mainly to prepare the grounds for
the proof (see Section \ref{sec:connectivitydomination}) of 
Theorem \ref{thm:connectivitytranstition}. This theorem 
will be proved by comparing the ellipsoid model on $H_k$ with the fractal 
ball model described at the end of Section \ref{sec:inducedmodel}. 
Then, we will relate this induced ball model to classical results about fractal 
ball models (see for instance \cite{MR_1996} Section 8, and \cite{BC_2010}). However, in 
order to obtain useful results when performing this comparison, we first 
need a better understanding of the induced processes. 

In Subsection \ref{sec:prelcalc} we perform some preliminary calculations
that will be useful to appeal to in further subsections. Then, Subsection
\ref{sec:ellipsestimates} establishes some basic properties of 
the measure $\xi_{k,r}$ when $r>0$. Finally, in Subsection \ref{sec:diam}
we focus on the distribution of the diameters of the induced ellipsoid 
process, since that is what will be useful when studying the induced 
ball model.

We point out that in the end we will not use the full results that we obtain 
below. For instance, it would be sufficient to restrict our attention to the 
case of $H_2$ when studying the connectivity phase transition (i.e. when
proving Theorem \ref{thm:connectivitytranstition}). However, we still chose to 
study the properties of the induced ellipsoid process in its generality for 
several reasons. Firstly, we do not believe that the general case 
makes the arguments more complicated or much longer. Secondly, for any 
eventual future projects it will be useful to have the full results 
in hand. Finally, in Section~\ref{sec:altpcproof} we demonstrate how to use our results in order to obtain an alternative proof of a result from \cite{TW_2012} concerning the standard Poisson cylinder model.

\subsection{Preliminary calculations} \label{sec:prelcalc}
We will frequently deal with integrals over spheres and will
repeatedly use the following coordinate change (see p.7 in \cite{AH_2012}). 
Recall that $e_1,...,e_d$ is the standard basis on $\R^d$. If $d\ge 3$, let $\phi \in \partial B^{d}(o,1)$ be 
written as 
\begin{equation}
\phi = s e_ d + \sqrt{1-s^2} u,\ s \in[-1,1], u\in \partial B^{d-1}(o,1).
\end{equation}
The corresponding surface element $\id \sigma_{d-1} (\phi)$  
is then 
\begin{equation}\label{eqn:coordinatechange}
\id\sigma_{d-1} (\phi) 
= (1-s^2)^{\frac{d-3}{2}} \id s \id \ell_{d-2}(u).
\end{equation}

We will use this change of coordinates when we need to perform integrals
over spheres. One integral will in particular surface in many places, and
so we study it separately in the following lemma. Recall the notation 
$\psi_m=\ell_m(\partial B^{m+1}(o,1))$ from Section \ref{sec:lincyl}.
It will be convenient to write $\CS^{d-1}=\partial B^{d}(o,1)$ 
so that $\psi_m=\ell_m(\CS^{m})$ here.

\begin{lemma}\label{lem:integralcomp}
Suppose that $k \in \{2,..., d-3\}$ and let 
\begin{equation}\label{eqn:i2}
I_{d,k}
:=\int_{(\CS^{d-k-2})^2}\int_{0}^{\infty} \frac{\kappa^{d-k-2}}{1+\kappa^2}
\frac{\id \kappa \id \ell_{d-k-2}(  \phi') \id \ell_{d-k-2}( \phi )}
{(1+(1-\langle \phi',\phi \rangle^2)\kappa^2)^{(d-k-1)/2}}.
\end{equation}
We have that 
\begin{equation}\label{eqn:i3}
I_{d,k} =\frac{ \psi_{d-k-1} \psi_{d-k-2}}{2}.
\end{equation}
\end{lemma}

\begin{proof}
First, we assume that $2\le k\le d-4$. By the change of variable $\kappa=t/\sqrt{1-\langle \phi',\phi \rangle^2}$ and 
straightforward calculations it follows that 
\begin{equation}\label{eqn:i1}
I_{d,k}= \int_{(\CS^{d-k-2})^2 } \int_0^\infty 
\frac{ t^{d-k-2}}{(1+t^2)^{\frac{d-k-1}{2} } } 
\frac{1}{1 -\langle  \phi',\phi \rangle^2 +t^2 } 
\frac{\id t \id \ell_{d-k-2}(  \phi') \id \ell_{d-k-2}( \phi )}
{\left( 1- \langle  \phi', \phi \rangle^2 \right) ^{\frac{d-k-3}{2}} }, 
\end{equation}
and we now calculate~\eqref{eqn:i1}. Since this integral only depends on 
$\phi,\phi'$ through the inner-product, we can simply let 
$\phi = e_ {d-k-1}$ so that
\[
I_{d,k}=\int_{\CS^{d-k-2} } \int_0^\infty \frac{t^{d-k-2}}{(1+t^2)^{(d-k-1)/2} }
\frac{\ell_{d-k-2}(\CS^{d-k-2})}{1-\langle \phi',  e_ {d-k-1} \rangle^2 +t^2 } 
\frac{\id t  \id \ell_{d-k-2}( \phi')}
{\left( 1-\langle \phi',  e_ {d-k-1} \rangle^2  \right)^{(d-k-3)/2} }  .
\]
Since we want to integrate over $\CS^{d-k-2}=\partial B^{d-k-1}(o,1)$ and $d-k-1\ge 3$, we can now we apply the coordinate change of Equation \eqref{eqn:coordinatechange} 
where we write 
\[
\phi' = s e_{d-k-1} +\sqrt{1-s^2} u,\ s \in [-1,1], u \in \CS^{d-k-3},
\]
so that the surface measure now equals
\[
\id \ell_{d-k-2}(  \phi') 
= \left( 1-s^2 \right)^{ (d-k-4)/2} \id s \id \ell_{d-k-3}( u)
\]
where $\id \ell_{d-k-3}$ is the surface measure on $\CS^{d-k-3}$. 
We have that $\langle \phi',  e_ {d-k-1} \rangle^2=s^2$ so that
\begin{align*}
I_{d,k}& = \int_{\CS^{d-k-3} } \int_0^\infty \int_{s=-1}^1 \frac{t^{d-k-2}}{(1+t^2)^{(d-k-1)/2} }
\frac{\psi_{d-k-2}}{1-s^2 +t^2 } \frac{\left( 1-s^2 \right)^{ (d-k-4)/2} }{\left( 1-s^2  \right)^{(d-k-3)/2} }\id s \id t  \id \ell_{d-k-3}(u) \\
&= \psi_{d-k-2}\psi_{d-k-3} 
\int_0^\infty \int_{s=-1}^1 \frac{t^{d-k-2}}{(1+t^2)^{(d-k-1)/2}} \frac{1}{1-s^2 +t^2} \frac{1}{\sqrt{1-s^2}} \id s \id t. 
\end{align*}
Now let $\alpha=\sin^{-1}(s)$ for $s\in [-1,1]$. Then 
$\id s=\sqrt{1-s^2} \id \alpha$ and $1-s^2=\cos^2(\alpha)$ so that
\begin{align*}
I_{d,k}&=  \psi_{d-k-2}\psi_{d-k-3} \int_0^\infty \int_{ - \pi/2}^{\pi/2} 
\frac{t^{d-k-2} }{(1+t^2)^{(d-k-1)/2} } \frac{1}{\cos^2( \alpha) +t^2 } \id \alpha \id t \\ & 
= \psi_{d-k-2}\psi_{d-k-3}  \int_0^\infty \frac{t^{d-k-2} }{(1+t^2)^{(d-k-1)/2} }\left[\frac{1}{t \sqrt{1+t^2}}\arctan{\left(\frac{t}{\sqrt{1+t^2}}\tan{\alpha}\right)} \right]_{\alpha=-\pi/2}^{\pi/2}  \id t \\
&=  \psi_{d-k-2}\psi_{d-k-3} \pi \int_0^\infty \frac{t^{d-k-3} }{(1+t^2)^{(d-k)/2} } \id t \\ &=\psi_{d-k-2}\psi_{d-k-3} \pi
\left[\frac{t^{d-k-2}}{(d-k-2)(1+t^2)^{(d-k-2)/2}}\right]_0^\infty \\
& = \frac{\psi_{d-k-2}\psi_{d-k-3} \pi}{d-k-2}
=\frac{ \psi_{d-k-1} \psi_{d-k-2}}{2}.
 \end{align*}
Here, the second and fourth equality follows by standard methods and 
the last equality follows from \eqref{eqn:psirecform}.

Now assume that $k=d-3$. In this case, to calculate~\eqref{eqn:i2}, we need to integrate over $\CS^{1}=\partial B^2(o,1)$ so we cannot use the coordinate change in~\eqref{eqn:coordinatechange}.  Instead, we see that~\eqref{eqn:i2} becomes

\begin{eqnarray*}
\lefteqn{I_{d,d-3}  =\int_{(\CS^{1})^2}\int_{0}^{\infty} \frac{\kappa}{1+\kappa^2}
\frac{\id \kappa \id \ell_{1}(  \phi') \id \ell_{1}( \phi )}
{(1+(1-\langle \phi',\phi \rangle^2)\kappa^2)}} \\ && = \ell_{1}(\CS^{1})\int_{\CS^{1}}\int_{0}^{\infty} \frac{\kappa}{1+\kappa^2}
\frac{\id \kappa \id \ell_{1}(  \phi') } {(1+(1-\langle \phi',e_2 \rangle^2)\kappa^2)}. 
\end{eqnarray*}
Writing $\phi'=(\cos(\theta),\sin(\theta))$ we have $\id \ell_{1}(  \phi')=\id \theta$ and $\langle \phi',e_2 \rangle^2=\sin^2(\theta)$. Hence,

\begin{eqnarray*}
\lefteqn{I_{d,d-3}=\ell_{1}(\CS^{1})\int_{0}^{\infty}\frac{\kappa}{1+\kappa^2} \int_0^{2\pi} 
\frac{1}{1+\cos^2(\theta)\kappa^2}\id \theta \id \kappa} \\ && =2 \pi \ell_1(\CS^1) \int_{0}^{\infty} \frac{\kappa}{(1+\kappa^2)^{3/2}}\id \kappa =2\pi \ell_1(\CS^1) =\frac{\psi_2 \psi_1}{2},
\end{eqnarray*}
where the second and third equality follows from standard methods, and the fourth equality follows since $\psi_2/2=\pi^{3/2}/\Gamma(3/2)=\pi^{3/2}/(\pi^{1/2}/2)=2\pi$. Here, we used the formula~\eqref{eqn:psidef}. 
\end{proof}

In the next lemmas, we present two useful integrals which will be used frequently in what follows. Since these integrals can be found in standard tables, we omit the proofs.
\begin{lemma}\label{l:gammaintegral}
Suppose that $b>0$, $-2 s +t<-1$ and $t>-1$. Then
\begin{equation}\label{eqn:gammaintegral}
\int_{0}^{\infty} (x^2+b)^{-s} x^t \id x 
=b^{(-2 s+t+1)/2} \frac{\psi_{2s-1}}{\psi_t \psi_{2s-t-2}}.
\end{equation}
\end{lemma}

\begin{lemma}\label{l:gammaintegral2}
Suppose that $s,t>-1$. Then
\begin{equation}\label{eqn:gammaintegral2}
\int_{0}^1 (1-x^2)^s x^t \id t 
=\frac{\psi_{2s+t+2}}{\psi_{2s+1} \psi_{t}}.
\end{equation}
\end{lemma}

\subsection{Total variation of the ellipsoid measures}
\label{sec:ellipsestimates}

The purpose of this subsection is to calculate the total variation of the measure $\xi_{k,r}$ for $r>0$.   First, we 
will describe some useful changes of coordinates which 
will simplify our calculations. These will also be used in later sections.

In order to perform our calculations we will need to differentiate between
three different cases, namely when $2\le k \leq d-3$, $k=d-2$ and $k=d-1.$
The first of these will be the most involved and it will be useful to 
introduce the following change of variables. First, let
\begin{equation}\label{eqn:transformation1}
T_1 : \ \ \left\{ 
\begin{array}{l}
( \rho,\theta):=(\|\adk \|,\adk/\|\adk \|)\in [0,\infty) \times \CS^{k-1} \\
( \kappa,\phi):=(\|\auk \|,\auk/\|\auk \|)\in [0,\infty) \times \CS^{d-k-2} \\
( \gamma,\varphi):=  (\|\puk \|,\puk/\|\puk \|)\in [0,\infty) \times \CS^{d-k-2},
\end{array}
\right.
\end{equation}
which transforms $\adk,\auk$ and $\puk$ into
a suitable set of spherical coordinates. Observe that if 
$k=d-2$ or $k=d-1$ then some of these changes does not really make sense.
This is the reason why we need to divide into several cases as mentioned
above. Our second transformation rescales the length $\|\puk\|=\gamma$ in 
terms of the length of $\| \auk\| = \kappa$ and the normalized angle 
between $\auk$ and $\puk$. This transformation depends on $r$ and so we 
emphasize this in the notation. Our second transformation is therefore
\begin{equation}\label{eqn:transformation2}
T_2^r: \ \ t:=\frac{\gamma}{r}
\left(\frac{ 1+ \kappa^2 }{1+ (1-\langle \varphi, \phi \rangle^2 ) \kappa^2 }
\right)^{-1/2}.
\end{equation}
We will of course make use of the composition $T^r :=  T_2^r \circ T_1.$
For clarity, we will write $\xi_{k,r}\circ T^r$ (or $\xi_k \circ T_1$ e.t.c.) 
when we work with the measure $\xi_{k,r}$ using these new coordinates.

Recalling the definition (i.e. \eqref{eqn:xikrdef}) of $\xi_{k,r}$, 
we see that 
\begin{eqnarray}\label{eqn:meastrans1}
\lefteqn{ \frac{\id (\xi_{k,r} \circ T_1) }{ \id \rho \id \kappa \id \gamma \id \ell_{k-1}(\theta) \id \ell_{d-k-2}(\phi) \id \ell_{d-k-2}(\varphi) }}\\
& & = \Upsilon_d \indicator \left( \gamma^2- \frac{ \gamma^2 \kappa^2 \langle \varphi,\phi  \rangle^2}{1+\kappa^2} < r^2  \right)\frac{ \rho^{k-1} ( \kappa \gamma)^{d-k-2} }{( \rho^2 + \kappa^2+ 1)^{(d+1)/2 } } \nonumber
\end{eqnarray}
and that
\begin{align}\label{eqn:meastrans2}
&\frac{\id (\xi_{k,r} \circ T^r)}{ \id \rho \id \kappa \id t \id \ell_{k-1}(\theta) \id \ell_{d-k-2}(\phi) \id \ell_{d-k-2}(\varphi)}\\ \nonumber & 
= \Upsilon_d \indicator \left(0 \leq t<1 \right) r^{d-k-1} \left( \frac{1+ \kappa^2  }{1+(1-\langle  \varphi,\phi \rangle^2) \kappa^2 }\right)^{(d-k-1)/2} \frac{ \rho^{k-1} ( \kappa t)^{d-k-2} }{( \rho^2 + \kappa^2 +1)^{(d+1)/2 } }.
\end{align}
It should be noted that the right hand side of \eqref{eqn:meastrans2}
depends on $r$ only through the factor $r^{d-k-1}$ and so
\begin{equation}\label{eqn:meashomogeneous}
\id (\xi_{k,r} \circ T^r) = r^{d-k-1} \id (\xi_{k,1} \circ T^1)
\end{equation}
(this can also been seen by using the scale invariance of the Poisson 
cylinder model).
Note also that both of the above Jacobians depends on $\varphi$ and $\phi$ 
only through the scalar product $\langle \varphi, \phi \rangle$.

Given $L\in \Line_{H_K^r}$ one can also check that under $T^r,$ the 
formulas for the volume and diameter of $E_k(L,r)$ from 
Corollary~\ref{corr:ellipsoidvolume} are given by (recall from Section 
\ref{sec:gennotation} that $\ell_k(B^k(o,1))=\psi_{k-1}/k=\psi_{k+1}/(2\pi)$)
\begin{equation}\label{eqn:vdtransformation}
\arraycolsep=1.4pt\def\arraystretch{2.2}
\begin{array}{rl}
\vol(E_k(L,r))\circ T^r ~&=\frac{\psi_{k-1}}{k} r^k \left( \frac{ \rho^2 + \kappa^2 +1}{\kappa^2 +1} \right)^{1/2} (1-t^2)^{k/2} ,\\
\diam(E_k(L,r))\circ T^r& = 2 r \left( \frac{ \rho^2 + \kappa^2 +1}{\kappa^2 +1} \right)^{1/2} (1-t^2)^{1/2}.
\end{array}
\end{equation}


\begin{remark}\label{rem:coords}
It should be remarked that one can make sense of this for $k=d-2,$ but 
as mentioned, the coordinate maps $T_1$ and $T_2^r$ have to change accordingly. 
More precisely, $T_1$ is now just the identity on $\auk,\puk$ since 
$\auk,\puk\in \R$ and moreover since 
$\langle \auk/|\auk|, \puk/|\puk| \rangle^2 = 1$ the denominator of 
$T_2^r$ vanishes. Additionally the expressions for the volume and 
diameter do not change but the Jacobian is changed slightly. The 
case $k=d-1$ is also different. In this case almost every line 
intersects $H_k$, so for example the indicator functions appearing 
in~\eqref{eqn:meastrans1} and~\eqref{eqn:meastrans2} disappear. 
\end{remark}

In the next lemma, we calculate the total mass (i.e. total variation) 
of $\xi_{k,r}$. 
\begin{lemma}\label{l:finite}
For $r\ge 0$, $d\ge 3$ and $2\le k\le d-1$, we have that the total
mass of $\xi_{k,r}$ is given by
\[
\|\xi_{k,r}\|_{TV} 
=\frac{r^{d-k-1}\psi_{d-k-1}}{\psi_{d-1}}.
\] 
\end{lemma}
\begin{proof}
We only prove the case $k \leq d-3$. The reason for this is that (as in Lemma~\ref{lem:integralcomp}) the cases $k=d-2$ and $k=d-1$ while somewhat easier, follow slightly different paths. The adjustments needed are outlined in Remark~\ref{rem:coords}.

We start by calculating
\begin{equation}\label{eqn:lemmafinite1}
\|\xi_{k,r}\|_{TV}=\Upsilon_d \int  \indicator
\left(r^2 -\|\puk\|^2 +\frac{ \langle \auk,\puk\rangle^2}{1+\|\auk\|^2 } > 0 \right)\frac{1}{\| a \|^{d+1}} \id \adk \id \auk \id \puk,
\end{equation}
where the integral is over $\adk\in \R^k$, $\auk \in \R^{d-k-1}$ and $\puk \in \R^{d-k-1}$.
Using~\eqref{eqn:meastrans2} we get that  \eqref{eqn:lemmafinite1} equals
\begin{align}\label{eqn:lemmafinite2}
 r^{d-k-1}\Upsilon_d& \int_{\substack{\rho,\kappa \geq 0\\ \varphi,\phi \in \CS^{d-k-2},\,\theta\in \CS^{k-1} } } 
\frac{ \rho^{k-1} \kappa^{d-k-2} }{(\rho^2+ \kappa^2 +1)^{(d+1)/2}} \\ \nonumber 
&  \hspace{5mm} \times \left(  \frac{1+ \kappa^2 }{1+(1-\langle \varphi,\phi \rangle^2) \kappa^2 } \right)^{(d-k-1)/2} \int_0^1 t^{d-k-2} \id t \id \rho \id \kappa \id \ell_{d-k-2}(\varphi) \id \ell_{d-k-2}(\phi) \id \ell_{k-1}(\theta) \\ \nonumber
&=\Upsilon_d\frac{r^{d-k-1} \psi_{k-1}}{d-k-1} \int_{ \substack{ \varphi, \phi \in \CS^{d-k-2},\\ \kappa,\rho \geq 0} }\frac{ \rho^{k-1} \kappa^{d-k-2} }{(\rho^2+ \kappa^2 +1)^{(d+1)/2}} \\ \nonumber &  \hspace{5mm} \times \left(  \frac{1+ \kappa^2 }{1+(1-\langle \varphi,\phi \rangle^2) \kappa^2 } \right)^{(d-k-1)/2} \id \rho \id \kappa \id \ell_{d-k-2}(\varphi) \id \ell_{d-k-2}(\phi),\end{align}
where the equality follows from performing the integration over $\theta$ and $t$.

Integrating with respect to $\rho$ we get by using \eqref{eqn:gammaintegral}
\[
\int_0^\infty \frac{ \rho^{k-1} }{( \rho^2 +\kappa^2 +1)^{(d+1)/2}} \id \rho 
=\frac{1}{(1+\kappa^2)^{(d-k+1)/2}}\frac{\psi_{d}}{\psi_{k-1}\psi_{d-k}}
\]
and so \eqref{eqn:lemmafinite2} equals
\begin{equation} \label{eqn:lemmafinite3}
\frac{r^{d-k-1}\psi_{k-1}}{d-k-1}  \frac{\psi_{d}}{\psi_{k-1}\psi_{d-k}} 
\Upsilon_d\int_{\substack{\varphi,\phi \in \CS^{d-k-2}\\ \kappa \ge 0} } 
\frac{ \kappa^{d-k-2} }{1+\kappa^2} \frac{\id \kappa \id \ell_{d-k-2}(\varphi) \id \ell_{d-k-2}(\phi)}
{ \left(  1+ \left( 1-\langle \varphi, \phi \rangle^2 \right)\kappa^2 \right)^{(d-k-1)/2} }.
\end{equation}
The integral on the right hand side is simply $I_{d,k}$ of Lemma
\ref{lem:integralcomp}. Therefore, by combining \eqref{eqn:lemmafinite1},
\eqref{eqn:lemmafinite2} and \eqref{eqn:lemmafinite3} with \eqref{eqn:i3}
we see that 
\begin{eqnarray*}
\lefteqn{\|\xi_{k,r}\|_{TV}
=\Upsilon_d\frac{r^{d-k-1}}{d-k-1}\frac{\psi_{d}}{\psi_{d-k}}
\frac{ \psi_{d-k-1} \psi_{d-k-2}}{2}}\\
& & =\frac{4 \pi}{\psi_d \psi_{d-1}}\frac{r^{d-k-1}\psi_{d}}{\psi_{d-k}}
\frac{\psi_{d-k}\psi_{d-k-1}}{4\pi}
=\frac{r^{d-k-1}\psi_{d-k-1}}{\psi_{d-1}}
\end{eqnarray*}
where we used \eqref{eqn:psirecform} in the second equality.

\end{proof}

\subsection{Properties of the diameter of random ellipsoids} \label{sec:diam}
By using Lemma \ref{l:finite} we see that
\[
\tilde{\xi}_{k,r}(\cdot):=\frac{1}{\|\xi_{k,r}\|_{TV}}\xi_{k,r}(\cdot),
\]
defines a probability measure.

The first goal of this section is to calculate the moments of $\diam(E)$ when $E$ is chosen according to $\tilde{\xi}_{k,r}$, see Lemma~\ref{lem:diametermoments2}. Then, we study the behavior of $\xi_{k}(E\,:\,\diam(E)\ge \tau)$ as a function of $\tau>0$.  The first such result is Lemma~\ref{lem:diameterlemma1}, which considers large values of $\tau$.  The second result is Lemma~\ref{lem:diam2} which holds for all  $\tau> 0$, but will only be used for small values of $\tau$.

\begin{lemma}\label{lem:diametermoments2}
Suppose that $d\ge 3$, $2\le k \le d-1$ and $n\in {\mathbb N}_+$. 
Then for any $r> 0$,
\begin{equation}
\E_{\tilde{\xi}_{k,r}}[\diam(E)^n]
=\left\{  
\begin{matrix}
2^n r^n \frac{2 \pi \psi_{d+n-k} \psi_{d-n}}{\psi_d \psi_{n+1} \psi_{d-n-k}} 
& \textrm{ if } n< d-k+1 \\
\infty & \textrm{ if } n\ge d-k+1.
\end{matrix}
\right.
\end{equation}
\end{lemma}

\begin{proof}
For the same reasons as in Lemma~\ref{l:finite} we only give the proof for the case $k \le d-3$. Using~\eqref{eqn:meastrans2}
and~\eqref{eqn:vdtransformation}, we see that  
$\E_{\tilde{\xi}_{k,r}}[\diam(E)^n]$ equals 
\begin{align*}
\frac{\Upsilon_d}{\|\xi_{k,r}\|_{TV}}&\int_{\substack{0\le t< 1,\, \rho,\kappa\ge 0 \\ \theta\in \CS^{k-1},\,\phi,\varphi\in \CS^{d-k-2}} } (2  r)^n 
\left( \frac{ \rho^2 + \kappa^2 +1}{\kappa^2 +1} \right)^{n/2} \\
& \times (1-t^2)^{n/2}   r^{d-k-1} 
\left( \frac{1+ \kappa^2  }{1+(1-\langle  \varphi,\phi \rangle^2) \kappa^2 }\right)^{(d-k-1)/2} \\ 
&  \times
\frac{ \rho^{k-1} ( \kappa t)^{d-k-2} }{( \rho^2 + \kappa^2 +1)^{(d+1)/2 } } 
\id \rho \id \kappa \id t \id \ell_{k-1}(\theta) \id \ell_{d-k-2} (\phi) \id \ell_{d-k-2}(\varphi).
\end{align*}
After simplifying the integrand and integrating over $\theta$, we obtain
\begin{align}\label{eqn:diameterm1}
\E_{\tilde{\xi}_{k,r}}&[\diam(E)^n] \nonumber 
= \Upsilon_d \frac{2^n r^{d+n-k-1} \psi_{k-1}}{\|\xi_{k,r}\|_{TV}}\\&
\times 
\int_{\substack{\rho,\kappa\ge 0 \\ \phi,\varphi\in \CS^{d-k-2}} }   \frac{ (\rho^2 + \kappa^2 +1)^{(-d+n-1)/2} \rho^{k-1} \kappa ^{d-k-2} }{ (\kappa^2 +1)^{(-d+n+k+1)/2}(1+(1-\langle  \varphi,\phi \rangle^2) \kappa^2 )^{(d-k-1)/2}} \id \rho \id \kappa \id \ell_{d-k-2}(\phi) \id \ell_{d-k-2}(\varphi) \\ & 
\hspace{40mm}\times \int_{t=0}^1 (1-t^2)^{n/2} t^{d-k-2}\id t. \nonumber
\end{align}
The integral over $t$ is, using~\eqref{eqn:gammaintegral2},
\begin{equation}\label{eqn:tintegral1}
\int_{t=0}^1 (1-t^2)^{n/2} t^{d-k-2}\id t 
=\frac{\psi_{d+n-k}}{\psi_{n+1} \psi_{d-k-2}}.
\end{equation}
Moreover, for $\kappa \in [0,\infty)$, we have by 
\eqref{eqn:gammaintegral} that
\begin{equation}\label{eqn:twocases}
\int_{0}^{\infty}(\rho^2 + \kappa^2 +1)^{(-d+n-1)/2} 
\rho^{k-1}\id \rho =
\left\{
\begin{matrix}
\frac{\psi_{d-n}}{\psi_{k-1} \psi_{d-n-k}}(1+\kappa^2)^{(-d+n+k-1)/2}, & n<d-k+1 \\
\infty, & n\ge d-k+1.
\end{matrix}
\right.
\end{equation}
Hence, we get that in the case $n<d-k+1$,
\begin{align}\label{eqn:idkgamma}
\int_{\substack{\rho,\kappa\ge 0 \\ \phi,\varphi\in \CS^{d-k-2}} } &  \frac{ (\rho^2 + \kappa^2 +1)^{(-d+n-1)/2} \rho^{k-1} \kappa ^{d-k-2} }{ (\kappa^2 +1)^{(-d+n+k+1)/2}(1+(1-\langle  \varphi,\phi \rangle^2) \kappa^2 )^{(d-k-1)/2}} \id \rho \id \kappa \id \ell_{d-k-2}(\phi) \id \ell_{d-k-2}(\varphi)\nonumber \\ & 
= \frac{\psi_{d-n}}{\psi_{k-1} \psi_{d-n-k}} \int_{\substack{\kappa\ge 0 \\ \phi,\varphi\in \CS^{d-k-2}}}\frac{\kappa ^{d-k-2}}{(\kappa^2 +1)(1+(1-\langle  \varphi,\phi \rangle^2) \kappa^2 )^{(d-k-1)/2}} \id \kappa \id \ell_{d-k-2}(\phi) \id \ell_{d-k-2}(\varphi). \nonumber \\ & 
\end{align}
The integral on the right hand side is $I_{d,k}$ of 
Lemma \ref{lem:integralcomp}. Using \eqref{eqn:i3} we then obtain 
by plugging~\eqref{eqn:tintegral1} and~\eqref{eqn:idkgamma} 
into~\eqref{eqn:diameterm1} that
\begin{eqnarray*}
\lefteqn{\E_{\tilde{\xi}_{k,r}}[\diam(E)^n] 
=  \Upsilon_d \frac{2^n r^{d+n-k-1} \psi_{k-1}}{\|\xi_{k,r}\|_{TV}}  
 \frac{\psi_{d+n-k}}{\psi_{n+1} \psi_{d-k-2}}
 \frac{\psi_{d-n}}{\psi_{k-1} \psi_{d-n-k}}I_{d,k}}\\
& & =  \frac{4 \pi}{\psi_d \psi_{d-1}}
\frac{2^n r^{d+n-k-1} \psi_{k-1}\psi_{d-1}}{r^{d-k-1}\psi_{d-k-1}}  
 \frac{\psi_{d+n-k}}{\psi_{n+1} \psi_{d-k-2}}
 \frac{\psi_{d-n}}{\psi_{k-1} \psi_{d-n-k}}\frac{\psi_{d-k-1}\psi_{d-k-2}}{2}\\
& & =2^n r^n\frac{2 \pi \psi_{d+n-k} \psi_{d-n}}{\psi_d \psi_{n+1} \psi_{d-n-k}}.
\end{eqnarray*}

\end{proof}

For us, the most important case of Lemma~\ref{lem:diametermoments2} is 
when $n=k$. Then the lemma implies that
\begin{equation}
\E_{\tilde{\xi}_{k,r}}[\diam(E)^k]
=\left\{  
\begin{matrix}
 2^k r^k \frac{2 \pi \psi_{d-k}}{\psi_{k+1} \psi_{d-2k}}, & k< (d+1)/2 \\
\infty, & k\ge (d+1)/2.
\end{matrix}
\right.
\end{equation}

\begin{lemma}\label{lem:diameterlemma1}
Suppose that $d\ge 3$ and $2\le k \le d-1$.  There are constants $c_3,\ldots,c_6$ such that for all $\tau\ge 4$ and $r\ge 0$, 
\begin{equation}\label{eqn:derivativediamest}
-c_3  \tau^{-d+k-2} r^{2(d-k)} \le \frac{\id}{\id \tau} \xi_{k,r}(E\,:\,\diam(E)\ge \tau)\le -c_4 \tau^{-d+k-2} r^{2(d-k)},
\end{equation}
\begin{equation}\label{eqn:taildiamest}
c_5  \tau^{-d+k-1} r^{2(d-k)} \le \xi_{k,r}(E\,:\,\diam(E)\ge \tau)\le c_6 \tau^{-d+k-1} r^{2(d-k)}.
\end{equation}
\end{lemma}

\begin{proof}

Again, as in Lemma~\ref{l:finite} we only give the proof for the case $k \le d-3$.

We first prove~\eqref{eqn:derivativediamest}, from 
which~\eqref{eqn:taildiamest} easily follows. Fix $r\in (0,1)$ and 
$\tau\ge 4$. Observe that by \eqref{eqn:xikrdef}, \eqref{eqn:meastrans2} and 
\eqref{eqn:vdtransformation} we have that 
\begin{align}\label{eqn:dde1}
\xi_{k,r}&(E\,:\,\diam(E)\ge \tau)\nonumber=\Upsilon_d r^{d-k-1}\int \left( \frac{1+ \kappa^2  }{1+(1-\langle  \varphi,\phi \rangle^2) \kappa^2 }\right)^{(d-k-1)/2} \\ & \times \int_{A_{\tau}} \frac{ \rho^{k-1} ( \kappa t)^{d-k-2} }{( \rho^2 + \kappa^2 +1)^{(d+1)/2 } } \id \rho \id \kappa \id t \id \ell_{k-1}(\theta) \id \ell_{d-k-2}(\phi) \id \ell_{d-k-2}(\varphi),
\end{align}
where $A_{\tau}=\left\{\rho\,:\,2r\left(\frac{\rho^2+\kappa^2+1}{\kappa^2+1}\right)^{1/2}(1-t^2)^{1/2}\ge\tau\right\}$ and the first integral is over $\kappa\ge 0,\,t\in (0,1),\,\theta\in \CS^{k-1}$ and $\phi,\varphi\in \CS^{d-k-2}$.  

Hence,
\begin{eqnarray}\label{eqn:dde2}
\lefteqn{\frac{\id}{\id \tau}\xi_{k,r}(E\,:\,\diam(E)\ge \tau) =\Upsilon_d r^{d-k-1}\int \left( \frac{1+ \kappa^2  }{1+(1-\langle  \varphi,\phi \rangle^2) \kappa^2 }\right)^{(d-k-1)/2}}\\  \nonumber && \times \frac{\id}{\id \tau}\int_{A_{\tau}} \frac{ \rho^{k-1} ( \kappa t)^{d-k-2} }{( \rho^2 + \kappa^2 +1)^{(d+1)/2 } } \id \rho \id \kappa \id t \id \ell_{k-1}(\theta) \id \ell_{d-k-2}(\phi) \id \ell_{d-k-2}(\varphi).
\end{eqnarray}
Now let $g(\tau,t,\kappa,r)=\sqrt{(1+\kappa^2)\left(\frac{\tau^2}{4 r^2 (1-t^2)}-1\right)}$. Observe that since $\tau \geq 4$ we have that
$\frac{\tau^2}{4 r^2 (1-t^2)}-1\geq 3$ and so $g(\tau,t,\kappa,r)$ is real.
In addition, let $f(\rho,\kappa)=\frac{ \rho^{k-1} }{( \rho^2 + \kappa^2 +1)^{(d+1)/2 }}$.  Then,
\[
\frac{\id}{\id \tau}\int_{A_{\tau}} \frac{ \rho^{k-1} }{( \rho^2 + \kappa^2 +1)^{(d+1)/2 } } \id \rho 
=  \frac{\id}{\id \tau} \int_{\rho=g(\tau,t,\kappa,r)}^{\infty} f(\rho,\kappa)\id \rho = -f(g(\tau,t,\kappa,r),\kappa) \frac{\id }{\id \tau} g(\tau,t,\kappa,r)
\]
We have that
\begin{eqnarray*}
\lefteqn{\frac{\id g(\tau,t,\kappa,r)}{\id \tau} =\frac{(1+\kappa^2)\tau}{2 r^2(1-t^2)}\frac{1}{2} \left((1+\kappa^2)\left(\frac{\tau^2}{4 r^2(1-t^2)}-1\right)\right)^{-1/2}} \\ && =\frac{\tau(1+\kappa^2)^{1/2}}{(4 r^2 (1-t^2))^{1/2}(\tau^2-4 r^2 (1-t^2))^{1/2}},
\end{eqnarray*}
and that
\begin{eqnarray*}
\lefteqn{f(g(\tau,t,\kappa,r),\kappa)=\frac{(1+\kappa^2)^{(k-1)/2}\left(\frac{\tau^2-4 r^2 (1-t^2)}{4 r^2 (1-t^2)}\right)^{(k-1)/2}}{\left(\frac{(1+\kappa^2)\tau^2}{4 r^2(1-t^2)}\right)^{(d+1)/2}}} \\ && 
= \tau^{-d-1} (1+\kappa^2)^{(-d+k-2)/2}(4 r^2 (1-t^2))^{(d-k+2)/2}(\tau^2-4 r^2 (1-t^2))^{(k-1)/2}.
\end{eqnarray*}
Hence,
\begin{eqnarray*}\label{eqn:fgder1}
\lefteqn{-f(g(\tau,t,\kappa,r),\kappa)\frac{\id }{\id \tau} g(\tau,t,\kappa,r) }
\\ && = -2^{d-k+1}\tau^{-d} (1+\kappa^2)^{(-d+k-1)/2} (r^2 (1-t^2))^{(d-k+1)/2}(\tau^2-4 r^2 (1-t^2))^{(k-2)/2}.
\end{eqnarray*}
For $\tau\ge 4$ we have (keeping in mind that $r,t\in (0,1)$),
\begin{equation}\label{eqn:fgder2}
\tau^2/2\le \tau^2-4 r^2 \left(1-t^2\right)\le \tau^2.
\end{equation}
We will use the notation $h_1 \asymp h_2$ whenever the two functions 
$h_1,h_2$ are such that $h_1(x)\leq c h_2(x)$ and $h_1(x)\geq c' h_2(x)$
for some constants $0<c,c'<\infty.$
We then see that
\begin{equation}\label{eqn:fgder4}
 -f(g(\tau,t,\kappa,r),\kappa) \frac{\id }{\id \tau} g(\tau,t,\kappa,r)
 \asymp -\tau^{-d+k-2} (1+\kappa^2)^{(-d+k-1)/2}r^{d-k+1}(1-t^2)^{(d-k+1)/2}.
\end{equation}
Inserting~\eqref{eqn:fgder4} into~\eqref{eqn:dde2} now gives, for $\tau\ge 4$,
\begin{eqnarray*}
\lefteqn{\frac{\id}{\id \tau}\xi_{k,r}(E\,:\,\diam(E)\ge \tau)}\nonumber\\ &&\asymp -\frac{r^{2(d-k)}}{\tau^{d-k+2}}\int \frac{ (1-t^2)^{(d-k+1)/2}(\kappa t)^{d-k-2}}{(1+(1-\langle  \varphi,\phi \rangle^2) \kappa^2)^{(d-k-1)/2}(1+ \kappa^2 )}  \id \kappa \id t \id \ell_{k-1}(\theta) \id \ell_{d-k-2}(\phi) \id \ell_{d-k-2}(\varphi) \\ && \asymp -\frac{r^{2(d-k)}}{\tau^{d-k+2}}\int \frac{ \kappa^{d-k-2}}{(1+(1-\langle  \varphi,\phi \rangle^2) \kappa^2)^{(d-k-1)/2}(1+ \kappa^2 )} \id \kappa \id \ell_{d-k-2}(\phi) \id \ell_{d-k-2}(\varphi) \\ &&= -I_{d,k} \frac{r^{2(d-k)}}{\tau^{d-k+2}}.
\end{eqnarray*}
where we use \eqref{eqn:i2} in the last equality.
This proves~\eqref{eqn:derivativediamest}. We now obtain~\eqref{eqn:taildiamest} using the identity
$$\xi_{k,r}(E\,:\,\diam(E)\ge \tau)=-\int_{\tau}^{\infty} \frac{\id}{\id s}\xi_{k,r}(E\,:\,\diam(E)\ge s)\id s.$$

\end{proof}

\begin{lemma}\label{lem:diameterlemma2}
Suppose that $d\ge 3$ and $2\le k \le d-1$. If $k\le d/2$, then there are 
constants $c_7,\ldots,c_{10}$ such that for all $\tau\ge 4$ and $r\ge 0$, 
\begin{equation}\label{eqn:derivativediamest2}
-c_7  \tau^{-d+k-2} \le \frac{\id}{\id \tau} \xi_{k}(E\,:\,\diam(E)\ge \tau)\le -c_8 \tau^{-d+k-2},
\end{equation}
\begin{equation}\label{eqn:taildiamest2}
c_9  \tau^{-d+k-1}\le \xi_{k}(E\,:\,\diam(E)\ge \tau)\le c_{10} \tau^{-d+k-1}.
\end{equation}
On the other hand, if $k\ge (d+1)/2$, then for all $\tau\ge 0$, 
\begin{equation}\label{eqn:infinitediamest}
\xi_{k}(E\,:\,\diam(E)\ge \tau)=\infty.
\end{equation}
\end{lemma}
\begin{proof}
If $k\le d/2$ we get using \eqref{eqn:xikdef}, \eqref{eqn:scaleinvariantmeasure}
and \eqref{eqn:derivativediamest} that for $\tau\ge 4$,
\begin{eqnarray*}
\lefteqn{\frac{\id}{\id \tau} \xi_{k}(E\,:\,\diam(E)\ge \tau)}\\ && 
=\int_{0}^1\frac{\id}{\id \tau} \xi_{k,r}(E\,:\,\diam(E)\ge \tau) r^{-d} \id r 
\asymp -\tau^{-d+k-2} \int_{0}^1 r^{d-2k}\id r \asymp -\tau^{-d+k-2}.
\end{eqnarray*}
This proves~\eqref{eqn:derivativediamest2}, from which~\eqref{eqn:taildiamest2} 
follows easily. Next, if $k\ge (d+1)/2$, then
\begin{eqnarray*}
\lefteqn{\xi_{k}(E\,:\,\diam(E)\ge \tau)=\int_{0}^1 \xi_{k,r}(E\,:\,\diam(E)\ge \tau) r^{-d} \id r} \\ && 
\ge c \tau^{-d+k-1} \int_0^1 r^{2(d-k)} r^{-d} \id r 
=c \tau^{k-d-1} \int_0^1 r^{d-2k} \id r 
=\infty,
\end{eqnarray*}
where we used~\eqref{eqn:taildiamest} in the inequality.

\end{proof}

In the next lemma and its proof, we will use the notation 
\[
\xi_{k,r} \left( \diam(E)^k ; \diam(E) \geq \tau \right)=\|\xi_{k,r}\|_{TV}\BE_{\tilde{\xi}_{k,r}}[\diam(E)^k : \diam(E)\ge \tau].
\]

\begin{lemma}\label{lem:diam2}

Suppose that $d\ge 4$ and $2\le k \le d/2$. Then, for any $\tau>0$ we have
\begin{eqnarray}\label{eqn:fractaldiameter}
\nonumber
\lefteqn{\xi_k  (E\,:\,\diam(E)\ge \tau)} \\ && 
= \frac{1}{k \tau^k} \xi_{k,1} \left( \diam(E)^k ; \diam(E) \geq \tau \right)
-\frac{1}{k} \xi_{k,1} \left(E\,:\,  \diam( E ) \geq \tau \right).
\end{eqnarray}
Moreover,
\begin{equation}\label{eqn:fractaldiameterderivative}
\frac{\id}{\id \tau} \xi_k(E\,:\,\diam(E)\ge \tau)
=-\frac{1}{\tau^{k+1}}\xi_{k,1} \left( \diam(E)^k ; \diam(E) \geq \tau \right),
\end{equation}
and so for $0<a\le b,$
\begin{equation}\label{eqn:diameterinterval}
\xi_k(E\,:\,a\le \diam(E)\le b)
=\int_{a}^{b}\frac{1}{\tau^{k+1}}
\xi_{k,1} \left( \diam(E)^k ; \diam(E) \geq \tau \right) \id \tau.
\end{equation}
\end{lemma}

\begin{proof}
We first show~\eqref{eqn:fractaldiameter}. 
Observe that by \eqref{eqn:xikdef} and \eqref{eqn:meashomogeneous} we 
have that 
\begin{eqnarray*}
\lefteqn{\xi_{k,r}(E\,:\,\diam(E)\ge \tau)}\\
& & =\int \indicator\left(2 r \left( \frac{ \rho^2 + \kappa^2 +1}{\kappa^2 +1} \right)^{1/2} (1-t^2)^{1/2} \geq \tau \right) \id (\xi_{k,r}\circ T^r) \\
& & =\int \indicator\left(2 \left( \frac{ \rho^2 + \kappa^2 +1}{\kappa^2 +1} \right)^{1/2} (1-t^2)^{1/2} \geq \frac{\tau}{r} \right) 
r^{d-k-1} \id (\xi_{k,1}\circ T^1)\\
& & =r^{d-k-1} \xi_{k,1}(E\,:\,\diam(E)\ge \tau/r).
\end{eqnarray*}
In fact, this equality also follows from scale invariance.
We see that 
\begin{eqnarray*}
\nonumber
\lefteqn{\xi_k(E\,:\,\diam(E)\ge \tau) =\int_{0}^{1} \xi_{k,r}(E\,:\,\diam(E)\ge \tau)r^{-d} \id r }\\ \nonumber && 
=\int_{0}^{1} \xi_{k,1}(E\,:\,\diam(E)\ge \tau/r)r^{-k-1} \id r\\
& & =\int_{0}^{1}\int  \indicator(\diam(E) \geq \tau/r) \id
\xi_{k,1}r^{-k-1} \id r \\
& & =\int \indicator (\diam(E) \geq \tau) \int_{0}^{1} 
\indicator(r \geq \tau/\diam(E)) r^{-k-1}\id r \id\xi_{k,1} \\
& & =\frac{1}{k}\int \indicator (\diam(E) \geq \tau) 
\left(\frac{\diam(E)^k}{\tau^k}-1 \right) \id\xi_{k,1}\\
& & =\frac{1}{k \tau^k} \xi_{k,1} 
\left( \diam(E)^k ; \diam(E) \geq \tau \right)
-\frac{1}{k} \xi_{k,1} \left(E\,:\,  \diam( E ) \geq \tau \right),
\end{eqnarray*}
proving~\eqref{eqn:fractaldiameter}. We move on to 
prove~\eqref{eqn:fractaldiameterderivative}. It is readily shown that 
the second term on the right hand side of~\eqref{eqn:fractaldiameter} 
is differentiable with respect to $\tau$. We now consider the derivative 
of $\xi_{k,1} \left( \diam(E)^k ; \diam(E) \geq \tau \right)$.
Observe that for $h> 0$,
\begin{eqnarray}\label{eqn:derdim1}
\nonumber \lefteqn{\frac{1}{h}\left(\xi_{k,1} \left( \diam(E)^k ; \diam(E) \geq \tau + h \right)-\xi_{k,1} \left( \diam(E)^k ; \diam(E) \geq \tau \right)\right)} \\ \nonumber 
&& = -\frac{1}{h}\xi_{k,1} 
\left( \diam(E)^k ; \tau \le \diam(E) \le \tau + h \right) \\ 
&& \geq -(\tau+h)^k \frac{1}{h}\xi_{k,1} 
(\tau \le \diam(E) \le \tau + h ).
\end{eqnarray}
This inequality can obviously be reversed if we replace $(\tau+h)^k$
with $\tau^k.$
Letting $h\to 0$, we then see that
\begin{eqnarray}\label{eqn:derdim3}
\lefteqn{ \frac{\id}{\id \tau}\xi_{k,1} \left( \diam(E)^k ; \diam(E) \geq \tau \right) }\\ \nonumber 
&&=-\tau^k \frac{\id}{\id \tau} \xi_{k,1} \left(E\,:\,  \diam( E ) \leq \tau \right) 
=\tau^k \frac{\id}{\id \tau} \xi_{k,1} \left(E\,:\,  \diam( E ) \geq \tau \right).
\nonumber
\end{eqnarray}
Hence, by~\eqref{eqn:fractaldiameter}
\begin{eqnarray*}
\lefteqn{\frac{\id}{\id \tau} \xi_k (E\,:\,\diam(E)\ge \tau)}\\ 
&& =\frac{\id}{\id \tau}\left( \frac{1}{k \tau^k} 
\xi_{k,1} \left( \diam(E)^k ; \diam(E) \geq \tau \right)
-\frac{1}{k} \xi_{k,1} \left(E\,:\,  \diam( E ) \geq \tau \right) \right)\\ 
&&=-\frac{1}{\tau^{k+1}}\xi_{k,1}\left( \diam(E)^k ; \diam(E) \geq \tau \right)\\ 
&& \hspace{10mm} + \frac{1}{k \tau^{k}} \tau^k \frac{\id}{\id \tau} \xi_{k,1} \left(E\,:\,  \diam( E ) \geq \tau \right) -\frac{1}{k} \frac{\id}{\id \tau} \xi_{k,1} \left(E\,:\,  \diam( E ) \geq \tau \right) \\ 
&& = -\frac{1}{\tau^{k+1}}\xi_{k,1}\left(\diam(E)^k ;\diam(E)\geq \tau\right),
\end{eqnarray*}
where we used~\eqref{eqn:derdim3} in the penultimate equality. This finishes the proof~\eqref{eqn:fractaldiameterderivative}.
Finally, we get~\eqref{eqn:diameterinterval} from~\eqref{eqn:fractaldiameterderivative} since
\[
\xi_k(E\,:\,a\le \diam(E)\le b)=-\int_{a}^{b}\frac{\id}{\id \tau} \xi_k(E\,:\,\diam(E)\ge \tau)\id \tau.
\]

\end{proof}

\subsection{A remark on the standard Poisson cylinder model}\label{sec:altpcproof}

Recall the notation $\widehat{\Omega}$, $\hat{\omega}$ and $\widehat{\CV}$ concerning the standard (that is, non-fractal) Poisson cylinder model  from Section~\ref{sec:pcmodel}. We will now describe how to obtain an alternative proof of one of the main results in \cite{TW_2012} using our results above. We point out that our alternative proof is not shorter than the original proof, since we will use for example Lemma \ref{lem:diametermoments2}. Let $\mbox{Perc}_2$ denote the event that $\widehat{\CV}\cap H_2$ contains unbounded connected components (that is, it percolates). The following theorem is Theorem 5.1 in \cite{TW_2012}.
\begin{theorem}[\cite{TW_2012}]\label{thm:twthm}
Let $d\ge 4$. There is a constant $c=c(d,r)>0$ such that if $\lambda\in [0,c]$, then $\widehat{\BP}_{\lambda}(\mathrm{Perc}_2)=1$.
\end{theorem}
We will proceed by comparing $\widehat{\CV}\cap H_2$ with the vacant set of a Poisson Boolean ball model and appealing to a theorem appearing independently both in \cite{ATT_2018} and \cite{P_2018}. Therefore we first recall the part which we will use from that theorem. Let $f$ be a positive measure with finite mass on $\R_+$ and let $\hat{\omega}_{\mathrm {pb}}$ be a Poisson point process on $\R^2\times \R_+$ with intensity measure $\lambda \ell_2\times f$. Then let
\[
\widehat{\CV}_{\mathrm{pb}}=\R^2 \setminus \bigcup_{(x,R)\in \hat{\omega}_{\mathrm{pb}}} B^2(x,R),
\]
be the vacant set in this Poisson Boolean ball model. Theorem $2$ in \cite{P_2018} or alternatively Theorem $2$ in \cite{ATT_2018} (applied to $\R^2$) implies that if 
\begin{equation}\label{eqn:2momentcond}
\int_{0}^{\infty} R^2 \id f(R)<\infty,
\end{equation}
then there is a.s.\  percolation in $\widehat{\CV}_{\mathrm{pb}}$ when $\lambda>0$ is sufficiently small.

\begin{proof}[Alternative proof of Theorem \ref{thm:twthm}.]

Let 
\[
\hat{\omega}_b=\sum_{L\in \hat{\omega}} \delta_{(\cent(E_2(L,r)),\diam(E_2(L,r))/2)}.
\]
In the same way as Theorem \ref{thm:ballfrac} was argued, one shows that $\omega_b$ is a Poisson point process on $\R^2\times \R_{+}$ with intensity measure $\lambda \ell_2\times \hat{g}$, where
\[
\hat{g}(\cdot)=\xi_{2,r}(E\,:\,\diam(E)/2\in \cdot).
\]
Moreover, if we let 
\[
\widehat{\CV}_b=\R^2\setminus \bigcup_{(x,R)\in \hat{\omega}_b} B^2(x,R),
\]
then clearly
\[
\widehat{\CV}_b\subset \widehat{\CV}\cap H_2.
\]
Similar constructions were used in Section \ref{sec:inducedmodel}. We will now be done if we can show that condition \eqref{eqn:2momentcond} holds with $\hat{g}$ in place of $f$. We have that
\begin{eqnarray*}
\lefteqn{\int_{0}^{\infty} R^2\id \hat{g}(R)=\int_{0}^{\infty} R^2 \xi_{2,r}(E\,:\,\diam(E)/2\in \id R)} \\ && =\|\xi_{2,r}\|_{TV}\int_{0}^{\infty} R^2  \tilde{\xi}_{2,r}(E\,:\,\diam(E)/2\in \id R)=\frac{\|\xi_{2,r}\|_{TV}}{4}\E_{\tilde{\xi}_{2,r}}[\diam(E)^2].
\end{eqnarray*}
Lemma \ref{lem:diametermoments2} applied to $k=2$ and $n=2$ shows that the last expectation is finite whenever $d\ge 4$, finishing the proof.
\end{proof}

\section{The fractal process: Connectivity phase and domination by 
the fractal ball model }\label{sec:connectivitydomination}

We will split the proof of Theorem \ref{thm:connectivitytranstition} into 
two separate statements as the proofs uses different approaches.

\subsection{Connectivity when $d\geq 4.$}

The objective of this subsection is to show that whenever $\lambda>0$ is
small enough, the fractal model contains connected components with 
probability one. This result is known to hold for the regular fractal ball 
model, and our strategy is to couple the fractal cylinder model with 
this ball model and infer the result from this coupling. Since the 
cylinders are unbounded we will need to consider the intersection 
of the cylinder process with a lower dimensional subspace $H_k,$ i.e.
we will use the induced ellipsoid process of Sections \ref{sec:cylintersect}
and \ref{sec:anaofellips}.

Therefore, define the regular fractal ball model as follows.
Consider a Poisson process $\omega_\reg$ on 
$\BR^k \times \BR_+$ with intensity measure 
\begin{equation}\label{eqn:regintens}
\lambda \ell_k \times R^{-k-1} \indicator(0<R\leq 2) \id R,
\end{equation}
and then let 
\[
\CV_\reg(\omega_\reg):=\BR^k \setminus \bigcup_{(x,R)\in \omega_\reg} B^k(x,R).
\]
The intensity measure in \eqref{eqn:regintens} corresponds to a 
scale invariant model with an upper cutoff (of 2) on the radius.
It is well known that for $\lambda>0$ small enough, 
$\CV^k_\reg(\omega_\reg):=\CV_\reg(\omega_\reg)\cap [0,1]^k \times\{0\}^{d-k}$ 
contains connected components with 
positive probability (see for example Theorem $2.4$ in \cite{BC_2010}). We note that 
it is customary to use $\indicator(0<R\leq 1)$ in place of 
$\indicator(0<R\leq 2)$ in \eqref{eqn:regintens}. However, because 
of scaling, this does not change the conclusion that $\CV^k_\reg(\omega_\reg)$
contains connected components with positive probability as long as $\lambda>0$
is small enough.

We now informally explain how the coupling between $\omega_\reg$
and the fractal cylinder model will be performed. There are 
three steps to the procedure.
First of all, the cylinder process induces a Poisson ball model as described
in Section \ref{sec:cylintersect} (see in particular Theorem 
\ref{thm:ballfrac}). Secondly, we will argue that ``large'' balls of this 
induced process can be disregarded (this uses the results of Section 
\ref{sec:anaofellips}), and so essentially the induced fractal 
ball model will have a cutoff similar to the regular fractal ball model. 
The third and final step will be to prove that the induced ball model
with a cutoff can be suitably dominated by a regular fractal ball model, 
and thereby we obtain a comparison between $\CV^k$ and $\CV^k_\reg$
(recall that $\CV^k=\CV\cap [0,1]^k \times \{0\}^{d-k}$).

We can now state and prove the following result.
\begin{theorem} \label{thm:conndgeq4}
For $d\geq 4,$ $\lambda_c \in(0,\infty)$.
\end{theorem}
\begin{proof}
Fix $d\geq 4$ and $k\leq d/2.$ We note that it would suffice to 
let $k=2$ throughout, but keeping $k$ in place does not change the proof.
The proof will rely on the discussion above that for $\lambda>0$ small
enough, $\CV^k_\reg$ contains connected components 
with positive probability. 

The first step is short as most of the work is already done. Recall the 
notation $\omega_b,$ Theorem \ref{thm:ballfrac}, and the fact that 
\[
\CV_b \subset \CV \cap H_k,
\]
where $\CV_b$ is the induced fractal ball model as defined in Section 
\ref{sec:cylintersect}.

Our second step will be to consider ``large'' balls. To that end, let 
\[
\tilde{\omega}_b:=\bigcup_{(x,R)\in \omega_b:R\leq 2} \delta_{(x,R)},
\]
so that $\tilde{\omega}_b$ is obtained by taking $\omega_b$ and removing 
any ball with radius larger than 2.
Then, let 
\[
\tilde{\CV}_b:=\BR^k \setminus \bigcup_{(x,R)\in\tilde{\omega}_b}B^k(x,R)
\]
and observe that $\tilde{\CV}_b \supset \CV_b$. 
As before, we will use 
the notation $\tilde{\CV}^k_b=\tilde{\CV}_b\cap[0,1]^k \times \{0\}^{d-k}$ 
and similar for $\CV_b^k.$ It is clearly the case that 
$\BP(\omega_b=\tilde{\omega}_b)=0$ since these are processes on the entire
space $\BR^k$. However, by restricting our attention to 
$[0,1]^k \times \{0\}^{d-k}$ this will not pose a problem. 
Therefore, let
\[
\CH_\emptyset:=\{\not \exists (x,R)\in \omega_b: R>2, 
B(x,R)\cap [0,1]^k\times\{0\}^{d-k} \neq \emptyset\}
\]
be the event that no ``large'' balls 
from $\omega_b$ hits $[0,1]^k\times\{0\}^{d-k}.$

Observe that by the nature of Poisson processes, conditioned on 
$\CH_\emptyset,$ we have that $\CV^k_b$ has the same distribution 
as $\tilde{\CV}^k_b$. Therefore,
\begin{eqnarray} \label{eqn:conncomp}
\lefteqn{\BP(\CV^k_b \textrm{ contains connected components})}\\
& & \geq \BP(\CV^k_b \textrm{ contains connected components}
| \CH_\emptyset) \BP(\CH_\emptyset) \nonumber \\
& & =\BP(\tilde{\CV}^k_b \textrm{ contains connected components}
) \BP(\CH_\emptyset).\nonumber 
\end{eqnarray}
Consider now $\BP(\CH_\emptyset)$ and 
note that by Theorem \ref{thm:ballfrac}, 
and Lemma \ref{lem:diameterlemma2} (which holds when the diameter is at 
least 4) we have that
\begin{eqnarray} \label{eqn:Hempty}
\lefteqn{\BP(\CH_\emptyset)=\exp\left(-\lambda \ell_k \times g 
((x,R):R >2, B^k(x,R)\cap [0,1]^k\times\{0\}^{d-k}\neq \emptyset)\right)}\\
& & \geq \exp\left(-\lambda \ell_k \times g 
((x,R):R >2, B^k(x,R)\cap (B^k(o,\sqrt{k})\times\{0\}^{d-k})
\neq \emptyset)\right) \nonumber \\
& & =\exp\left(-\lambda \int_{\BR^k} g(R \geq \max(2, \|x\|-\sqrt{k})) 
\id \ell_k( x)\right) \nonumber \\
& & =\exp\left(-\lambda \int_{\BR^k} \xi_k(E:\diam(E)
\geq \max(4, 2\|x\|-2\sqrt{k}))\id \ell_k( x)\right) \nonumber \\
& & \geq \exp\left(-\lambda \int_{\BR^k} c_{10} \max(4, 2\|x\|-2\sqrt{k})^{-d+k-1}
\id \ell_k( x)\right) \nonumber \\
& & =\exp\left(-\lambda c \int_0^\infty \max(4, 2R-2\sqrt{k})^{-d+k-1}
R^{k-1} \id R \right) \nonumber \\
& & \geq \exp\left(-\lambda c \int_{2\sqrt{k}}^\infty R^{-d+2k-2}\id R\right)>0,
\nonumber 
\end{eqnarray}
whenever $d\geq 4$ and $k\leq d/2.$
Observe that \eqref{eqn:conncomp} and \eqref{eqn:Hempty} allow us to 
transfer the problem from $\CV^k_b$ to one about $\tilde{\CV}_b^k.$ 
This concludes the second step of the proof.

The third step will be to show that $\tilde{\CV}^k_b$ contains 
connected components with positive probability, and this will be done 
by comparison with $\CV^k_\reg$. To that end we note that the 
intensity measure corresponding to $\tilde{\CV}^k_b$ is simply 
\[
\lambda \id \ell_k(x) \times \id g( R) \indicator (0<R\leq 2),
\]
since there are no balls of radius larger than 2 in $\tilde{\omega}_b$.

By Lemma \ref{lem:diam2} (in particular Equation~\eqref{eqn:diameterinterval}), 
\begin{eqnarray}\label{eqn:gdens}
\lefteqn{\id g( R)=2^{-k}\xi_{k,1}(\diam(E)^k ; \diam(E)\ge 2 R)R^{-k-1}\id R}\nonumber\\ && = 2^{-k} \|\xi_{k,1}\|_{TV}\E_{\tilde{\xi}_{k,1}} \left[ \diam(E)^k ; 
\diam(E) \geq 2 R \right]R^{-k-1}\id R \nonumber \\ && =2^{-k}\beta_{ R} \|\xi_{k,1}\|_{TV}R^{-k-1}\id R,
\end{eqnarray}
where 
\[
\beta_R:=\E_{\tilde{\xi}_{k,1}} \left[ \diam(E)^k ; 
\diam(E) \geq 2 R \right]
\leq \E_{\tilde{\xi}_{k,1}} \left[ \diam(E)^k \right]
(=\beta_0)<\infty.
\]
In~\eqref{eqn:gdens}, the factor $2^{-k}$ appears since Lemma \ref{lem:diam2} 
is formulated for $\diam(E)$ rather than $\diam(E)/2$.
The fact that $\beta_0<\infty$ comes from  
Lemma \ref{lem:diametermoments2} where we used $n=k$ and that $k \leq d/2.$ 
Letting 
\[
\id g_0( R):=2^{-k}\beta_0 \|\xi_{k,1}\|_{TV}
R^{-k-1}\indicator(0<R\leq 2)\id R,
\]
we let $\omega_{b,0}$ be a Poisson process on $\BR^k \times \BR_+$
with the intensity measure 
\[
\lambda \id \ell_k( x) \times \id g_0( R).
\]
Clearly, this is the same intensity measure as in \eqref{eqn:regintens}
but with a different constant. Letting

\[
\CV(\omega_{b,0})=\BR^k \setminus \bigcup_{(x,R)\in \omega_{b,0}}B^k(x,R),
\]
it therefore follows that 
$\CV(\omega_{b,0})$ contains connected components whenever $\lambda>0$ 
is chosen small enough.

It is straightforward to couple $\tilde{\omega}_b$ and 
$\omega_{b,0}$ so that 
\[
\CV(\omega_{b,0}) \subset \tilde{\CV}_b. 
\]
Indeed, one can do this by considering a point process on the space
$\BR^d\times\BR_+ \times \BR_+$ with intensity measure
\[
\lambda \id x \times R^{-k-1}\indicator(0<R\leq 2)\id R 
\times \indicator (q>0) \id q,
\]
and then for a given triple $(x,R,q)$ we let $(x,R)\in\omega_{b,0}$ iff
$q\leq 2^{-k} \beta_0 \|\xi_{k,1}\|_{TV}$, and $(x,R)\in \tilde{\omega}_{b}$ iff
$q\leq 2^{-k} \beta_{R}  \|\xi_{k,1}\|_{TV}.$
Letting $\BP_{b,0}$ denote the law of $\omega_{b,0}$, we conclude that 
\begin{eqnarray} \label{eqn:conncomp2}
\lefteqn{\BP(\tilde{\CV}^k_b \textrm{ contains connected components}
)}\\
& & \geq \BP_{b,0}(\CV^k(\omega_{b,0})\textrm{ contains connected components})>0
\nonumber
\end{eqnarray}
By combining \eqref{eqn:conncomp}, \eqref{eqn:Hempty} and 
\eqref{eqn:conncomp2} the statement follows.
\end{proof}

\subsection{Connectivity when $d=2,3.$}

We now turn to the final case of connectivity when $d=2,3.$
We will prove the following theorem.

\begin{theorem}\label{thm:disconn23}
Let $\lambda>0$. 
\begin{enumerate}
\item[a)] For $d=2$, there are almost surely no connected components in $\CV$.
\item[b)] For $d=3$, there are almost surely no connected components in $\CV\cap H_2$.
\end{enumerate}
\end{theorem}

Before we proceed with the proof of Theorem \ref{thm:disconn23},
we need to establish some notation and auxiliary results.
For $s,t>0$, we define the 
rectangle $K(s,t) := [ -s/2  , s/2 ] \times [-t/2, t/2]\subset H_2$ and we 
write $K(s)=K(s,s)$ for the square of side-length $s$ centred at the origin. 
For $\epsilon>0$, let $LR=LR(\epsilon)$ denote the set of ellipsoids centred 
in the square $K(\epsilon/2)\subset H_2$ intersecting both the left and 
right-hand sides of the rectangle $K(3\epsilon,\epsilon)$. That is,
\begin{eqnarray}\label{eqn:crossingevent}
\lefteqn{LR= \{  E\in \mathfrak{E}^2\, :  \ \cent(E) \in K(\epsilon/2), } \\ 
&& \nonumber
 E \cap \{-3 \epsilon/2 \} \times[-\epsilon/2,\epsilon/2] \neq \emptyset, 
 E \cap \{3 \epsilon/2 \} \times[-\epsilon/2,\epsilon/2] \neq \emptyset     \}.
\end{eqnarray}

For any set $\mathcal{R}\subset \R^2$ and $A ,B \subset	\R^2$ we 
define 
\begin{equation*}
A \overset	{ \mathcal{R}}{\longleftrightarrow} B
\end{equation*}
to be the event that there exists a connected component in 
$\mathcal{R}$  intersecting both $A$ and $B$. 
Let 
\begin{equation*}
\mathrm{Arm}( \epsilon,\mathcal{R}) = \left\{ K(\epsilon) \overset{\mathcal{R}}{\longleftrightarrow} \partial K(3 \epsilon) \right\},
\end{equation*}
be the event that there is a crossing in $\mathcal{R}$ of the 
annulus $K(3\epsilon) \setminus K(\epsilon)$.

\begin{lemma}\label{lem:noarmdisconn}
Assume that $\CR$ is a random closed subset of $\R^2$ and that the law  ${\mathbf P}$ of 
$\CR$ is invariant under translations and rotations of $\R^2$. Assume further
that for every $\epsilon\in (0,1/5)$ we have 
\begin{equation}\label{eqn:assumption1}
{\mathbf P}\left(\{-3\epsilon/2\}\times [-\epsilon/2,\epsilon/2] 
\overset{ \mathcal{R}^c \cap K(3\epsilon,\epsilon)}{\longleftrightarrow}
\{3\epsilon/2\}\times [-\epsilon/2,\epsilon/2] \right)=1.
\end{equation}
Then ${\mathbf P}(\CR \mbox{ is totally disconnected})=1$.
\end{lemma}
Observe that the event inside~\eqref{eqn:assumption1} is the existence 
of a connected component contained in 
$\CR^c \cap K(3\epsilon,\epsilon)$ connecting the right and left sides 
of the rectangle $K(3\epsilon,\epsilon)$.
\begin{proof}
Let $\epsilon\in (0,1/5)$. First observe that if the four events 
\[
\{-3\epsilon/2\}\times [\epsilon/2,3\epsilon/2] \overset	{ \mathcal{R}^c \cap (K(3\epsilon,\epsilon)+(0,\epsilon))}{\longleftrightarrow}\{3\epsilon/2\}\times [\epsilon/2,3\epsilon/2],
\]
\[
\{-3\epsilon/2\}\times [-3\epsilon/2,-\epsilon/2] \overset	{ \mathcal{R}^c \cap (K(3\epsilon,\epsilon)-(0,\epsilon))}{\longleftrightarrow}\{3\epsilon/2\}\times [-3\epsilon/2,-\epsilon/2]
\]
\[
 [\epsilon/2,3\epsilon/2] \times \{-3\epsilon/2\}\overset	{ \mathcal{R}^c \cap (K(\epsilon,3\epsilon)+(\epsilon,0))}{\longleftrightarrow}[\epsilon/2,3\epsilon/2]\times \{3\epsilon/2\},
\]
\[
[-3\epsilon/2,-\epsilon/2] \times \{-3\epsilon/2\}\overset	{ \mathcal{R}^c \cap (K(\epsilon,3\epsilon)-(\epsilon,0))}{\longleftrightarrow}[-3\epsilon/2,-\epsilon/2]\times \{3\epsilon/2\},
\]
all occur, then $\mathrm{Arm}(\epsilon,\CR)$ does not occur. Furthermore,
since the aforementioned four events all have probability $1$ (by 
assumption~\eqref{eqn:assumption1} and the rotational and translation 
invariance of the law of $\CR$), we get that 
\begin{equation}\label{eqn:zeroarmprob}
{\mathbf P}(\mathrm{Arm}(\epsilon,\CR))=0.
\end{equation}
Now let $q \in \Q^2$, $\epsilon \in (0,1/5) \cap \Q$ and 
$\mathrm{Arm}_q(\epsilon,\CR)$ denote the event that there is a crossing in 
$\CR$ in the annulus 
\[
q + K(3\epsilon) \setminus K(\epsilon).
\]
Then we have that
\begin{equation*}
\left\{ \CR \text{ is totally disconnected} \right\}^c 
\subset \bigcup_{q \in \Q^2} \bigcup_{\epsilon \in(0,1/5) \cap \Q}  
\mathrm{Arm}_q(\epsilon,\CR),
\end{equation*}
since if $\CR$ has a connected component there must exist some $q\in \Q^2$ and some $\epsilon \in (0,1/5) \cap \Q$ such that there is a crossing of the annulus 
\[
q  + K(3\epsilon) \setminus K(\epsilon).
\]
 Hence we have
\begin{equation*}
{\mathbf P}\left( \left\{ \CR \text{ is totally disconnected} \right\}^c \right) 
\leq \sum_{ q \in \Q^2 } \sum_{ \epsilon \in (0,1/5) \cap \Q} 
{\mathbf P} \left( \mathrm{Arm}_q(\epsilon,\CR) \right)=0,
\end{equation*}
using~\eqref{eqn:zeroarmprob} and translational invariance in the last 
equality. Thus 
\begin{equation*}
{\mathbf P} \left(\CR\text{ is totally disconnected } \right)=1,
\end{equation*}
as required.
\end{proof}

To deal with $\CV\cap H_2$ when $d=3$ we will need one additional result. 
We first recall some formulas used if $d=3$ and $k=2$. For 
$d=3$, the expression for the shape measure $\xi_{2,r}$ is given by
\begin{eqnarray*}
\xi_{2,r}(\cdot) = \int_{ E_2(L,r)_o \in \cdot} 
\frac{1}{\left( 1+a_1^2 +a_2^2 \right)^2} \id a_1 \id a_2,
\end{eqnarray*}
see Theorem \eqref{thm:cylellips} and Equation~\eqref{eqn:xikrdef}. 
Moreover, according to Corollary \ref{corr:ellipsoidvolume},
the expression for the diameter of an ellipse $E_2(L,r)$ is then given by 
\[
\diam(E_2(L,r)) = 2 r \sqrt{1+a_1^2 +a_2^2}.
\]
Recall also the notation
\begin{equation*}
\omega_{e} = \sum_{ (L,r) \in \omega } \delta_{ (\cent(E_2(L,r)),E_2(L,r)_o) }.
\end{equation*}
We have the following lemma.
\begin{lemma}\label{lem:ellipsecrossing}
Let $d=3$ and $\epsilon \in(0,1/5)$. Then for every $\lambda>0$,
\begin{equation}\label{eqn:ellipsecrossing}
\Pm \left( LR(\epsilon,\omega_{e})\right) =1.
\end{equation}

\end{lemma}
\begin{proof}
The proof essentially follows that of the lower bound of Proposition 5.1 in \cite{TU_2017}.
For an ellipse $E\in {\mathfrak E}^2$, let $ \arg(E) \in [-\pi/2,  \pi/2)$ 
denote the angle between the $e_1$ axis and the line containing the major 
axis of $E$ . It is easy to check that if $E$ satisfies
\begin{align*}
& \cent(E) \in K(\epsilon/2) \\
& \diam(E) \geq 10 \epsilon,\ |\arg (E)| < 1/10,
\end{align*}
then $E\in LR$. Therefore, if we let
\begin{equation*}
LR_1=LR_1(\epsilon,\omega_{e})
=\left\{ E\in {\mathfrak E}^2 : \cent(E) \in K(\epsilon/2), 
\diam(E) \geq 10 \epsilon,\ |\arg (E)| < 1/10 \right\},
\end{equation*}
then we have $LR_1\subset LR$ so that
\begin{equation*}
LR_1(\omega_{e}) \subset LR(\omega_{e}).
\end{equation*}

Using Theorem \ref{thm:cylellips} we have that the intensity measure 
of $\omega_{e}$ is given by $\lambda \ell_2\times \xi_2$ where 
\begin{equation*}
\xi_2( \cdot)= \int_{0}^1 \xi_{2,r}( \cdot) r^{-3} \id r.
\end{equation*}
Hence 
\begin{align*}
\Pm \left(  LR(\omega_{e}) \right) \geq \Pm \left( LR_1(\omega_{e}) \right)
=1-\e^{ -\lambda \ell_2\times \xi_2(LR_1) }.
\end{align*}
It remains to show that $ \ell_2\times \xi_2(LR_1)=\infty$. We have that
\begin{eqnarray}\label{eqn:easycrossing}
\nonumber
\lefteqn{ \ell_2\times \xi_2(LR_1) = \ell_2(K(\epsilon/2)) \int_{0}^1 \int_{ \substack{ \diam(E_2(L,r)) \geq 10 \epsilon \\ |\arg (E_2(L,r))|<1/10 }} \frac{1}{(1 +a_1^2 +a_2^2)^2} r^{-3} \id  a_1 \id a_2 \id r }\\
&& = \frac{\epsilon^2}{4}  \int_{0}^1\int_{ \substack{ \diam(E_2(L,r)) \geq 10 \epsilon \\ |\arg (E_2(L,r))|<1/10 }} \frac{1}{(1 +a_1^2 +a_2^2)^2} r^{-3} \id  a_1 \id a_2 \id r .
\end{eqnarray}
We now change coordinates from $(a_1,a_2)$ to polar coordinates $( \rho ,\theta) \in \R_+ \times [- \pi , \pi).$ By the facts that $\theta = \arg(E(r)) (= \arctan(a_2/a_1))$ and $\diam(E(r)) = 2 r \sqrt{1+a_1^2 +a_2^2}$, Equation \eqref{eqn:easycrossing} equals 
\begin{eqnarray}\label{eqn:easycrosscomp}
\nonumber
\lefteqn{\frac{\epsilon^2}{4} \int_{0}^1 \int_{ \substack{ 2 r \sqrt{1+\rho^2 } \geq 10 \epsilon \\|\theta|<1/10 }} \frac{\rho}{(1 +\rho^2)^2} r^{-3} \id  \rho \id \theta \id r= \frac{\epsilon^2}{20}
\int_{0}^1 \int_{ r \sqrt{1+\rho^2} \geq 5 \epsilon} \frac{\rho}{(1 +\rho^2 )^2} r^{-3} \id \rho \id r }\\ \nonumber
&&=\frac{\epsilon^2}{20} \int_{0}^1\left[\frac{-1}{2(1+\rho^2)}\right]_{\sqrt{\max(\frac{25\epsilon^2}{r^2}-1,0)}}^{\infty} r^{-3} \id r=\frac{\epsilon^2}{40} \int_{0}^1 \frac{1}{ \max( 5 \epsilon/r,1)^2} r^{-3} \id r \\ && \ge \frac{\epsilon^2}{40} \int_{0}^{5 \epsilon} \frac{1}{ 25\epsilon^2} r^{-1} \id r=\infty,
\end{eqnarray}
as required.

\end{proof}

\begin{proof}[Proof of Theorem~\ref{thm:disconn23}.]
We start with part $a)$. Recall that in this case, we work with 
$d=2$. The aim is to verify the assumption~\eqref{eqn:assumption1} of 
Lemma~\ref{lem:noarmdisconn}. For $\epsilon>0$, define 
$$
\mbox{Cross}(\epsilon)=\Line_{\{-3\epsilon/2\}\times [-\epsilon/2,\epsilon/2],\{3\epsilon/2\}\times [-\epsilon/2,\epsilon/2]}\subset A(2,1).
$$
We observe that for any $\lambda>0$ and  $\epsilon >0$ we have that 
\begin{eqnarray*}
(\lambda \nu_2 \times \rho_s) ((L,r)\,:\,L\in \mbox{Cross}(\epsilon)) 
=\int_{0}^1 \nu_2\left(\mbox{Cross}(\epsilon)\right)r^{-2}\id r=\infty,
\end{eqnarray*}
since $\nu_2(\mbox{Cross}(\epsilon))>0$ (this claim is easy to check and 
left to the reader). This implies that 
\begin{equation}\label{eqn:linecross}
\BP\left(\exists (L,r)\in \omega: L\in \mbox{Cross}(\epsilon)\right)=1.
\end{equation}
Since
\begin{eqnarray*}
\left\{ \exists (L,r)\in \omega: L\in \mbox{Cross}(\epsilon)\right\}
\subset \left\{\{-3\epsilon/2\}\times [-\epsilon/2,\epsilon/2] 
\overset{ \CV^c \cap K(3\epsilon,\epsilon)}{\longleftrightarrow}\{3\epsilon/2\}\times [-\epsilon/2,\epsilon/2] \right\},
\end{eqnarray*}
 we get the result from~\eqref{eqn:linecross} and Lemma~\ref{lem:noarmdisconn}.

We now move on to prove part $b),$ for which most of the work has already 
been done. Let $d=3$ and $\lambda>0$. We have that 
\[
\{LR(\epsilon,\omega_e)\}
\subset \left\{\{-3\epsilon/2\}\times [-\epsilon/2,\epsilon/2] 
\overset { (\CV\cap H_2)^c \cap K(3\epsilon,\epsilon)}{\longleftrightarrow}\{3\epsilon/2\}\times [-\epsilon/2,\epsilon/2] \right\}
\]
and so the statement follows from Lemmas~\ref{lem:ellipsecrossing} 
and~\ref{lem:noarmdisconn}.
\end{proof}

We can now prove Theorem \ref{thm:connectivitytranstition}.

\begin{proof}[Proof of Theorem~\ref{thm:connectivitytranstition}]
The statement for $d\geq 4$ is Theorem \ref{thm:conndgeq4}
while the statements for $d=2,3$ are covered by Theorem \ref{thm:disconn23}.
\end{proof}

\begin{appendices}

\section{A 0-1 law} \label{app:ergo}

The purpose of this appendix is to prove Lemma \ref{lem:zeroone}.

\begin{proof}[Proof of Lemma \ref{lem:zeroone}]
We will start by considering events depending only on some finite region of 
$A(d,1) \times (0,1].$ To that end, for $R\ge 1$, let 
\[
\Gamma_{x,R} = \{(L,r) \in A(d,1)\times (0,1] : 
L \in \CL_{B(x,R)} \text{ and } r\in (R^{-1},1] \},
\]
and
\[
\left.\omega\right|_{\Gamma_{x,R}} 
= \{(L,r) \in \omega : (L,r) \in \Gamma_{x,R}\}
\]
so that $\left.\omega\right|_{\Gamma_{x,R}}$ is the restriction of 
$\omega \in \Omega$
to the finite region $\Gamma_{x,R}.$ When $x=o$ we will 
simply write $\Gamma_R.$
The event $F$ is determined by $\Gamma_{x,R}$ if and only if for every $\omega \in F$ and any 
$\tilde{\omega}$ such that $\left.\omega\right|_{\Gamma_{x,R}} = \left.
\tilde{\omega}\right|_{\Gamma_{x,R}}$ we have that 
$\tilde{\omega} \in F.$ Note that if $F$ is determined by $\Gamma_{x,R}$ then 
it is also determined by $\Gamma_{x,R'}$ for any $R' > R.$

We will now prove that, for any events $F_i$ determined by $\Gamma_{x_i,R}$ for $i= 1,2$ where 
$\dist{x_1}{x_2}\geq 4 R$
we have that
\begin{equation} \label{eqn:E1E2bound}
| \BP(F_1 \cap F_2) - \BP(F_1) \BP(F_2) | 
\leq 4\left(1 - \exp\left(-\frac{\lambda c_2R^{2(d-1)}
\left(R^{d-1} - 1\right)}{(d-1)\dist{x_1}{x_2}^{d-1}} \right)\right).
\end{equation}
For simplicity, write 
$\Gamma_{x_1,x_2,R} = \CL_{B(x_1,R), B(x_2,R)}\times (R^{-1},1]$ 
and observe that
\begin{eqnarray}
 \lefteqn{\BP(F_1 \cap F_2) } \label{eq:almostindependency1}\\
 & & =\BP(F_1|\omega(\Gamma_{x_1,x_2,R}) =0)\BP(F_2|\omega(\Gamma_{x_1,x_2,R}) =0)\BP(\omega(\Gamma_{x_1,x_2,R}) =0) \nonumber\\
 & & \ \ \ + \BP(F_1 \cap F_2|\omega(\Gamma_{x_1,x_2,R})\neq 0) 
 \BP(\omega(\Gamma_{x_1,x_2,R}) \neq 0),\nonumber
\end{eqnarray}
since the events $F_1$ and $F_2$ are conditionally independent on the event that
$ \omega(\Gamma_{x_1,x_2,R})=0$.
Furthermore, writing
\begin{eqnarray*}
\lefteqn{ \BP(F_i) 
=\BP(F_i|\omega(\Gamma_{x_1,x_2,R}) =0)\BP(\omega(\Gamma_{x_1,x_2,R}) =0)} \\
 && \ \ \ + \BP(F_i|\omega(\Gamma_{x_1,x_2,R}) \neq 0)\BP(\omega(\Gamma_{x_1,x_2,R}) \neq 0)
\end{eqnarray*}
 for $i= 1,2$ and using \eqref{eq:almostindependency1}, a straightforward calculation gives us that 
\[
|\BP(F_1 \cap F_2) - \BP(F_1) \BP(F_2)| \leq 4 \BP(\omega(\Gamma_{x_1,x_2,R}) \neq 0).
\]
We will now bound 
$\BP\left(\omega\left(\Gamma_{x_1,x_2,R}\right) \neq 0 \right)$. 
First we use Lemma \ref{lem:basicmeasure} part $c)$ to see that
\begin{eqnarray*}
\lefteqn{\BP\left(\omega\left(\Gamma_{x_1,x_2,R}\right)\neq 0\right)} \\
&&= 1 - \exp\left(-\lambda \int_{R^{-1}}^1 \nu_d\left(\CL_{B(x_1, R),B(x_2, R)}\right)r^{-d}\id r\right) \\
&& = 1 - \exp\left(-\lambda \int_{R^{-1}}^1 R^{d-1 }\nu_d\left(\CL_{B(x'_1, 1),B(x'_2, 1)}\right)r^{-d}\id r\right), 
\end{eqnarray*}
where $x'_1$ and $x'_2$ is as in Lemma \ref{lem:basicmeasure}. Since $\dist{x'_1}{x_2'}=R^{-1}\dist{x_1}{x_2}\ge 4$ we use Lemma \ref{lem:measurebounds_org} to see that 
\begin{eqnarray*}
\lefteqn{\BP\left(\omega\left(\Gamma_{x_1,x_2,R}\right)\neq 0\right)} \\
&& \leq 1 - \exp\left(-\lambda \int_{R^{-1}}^1 
c_2 \left(\frac{R}{\dist{x'_1}{x'_2}}\right)^{d-1} r^{-d}\id r \right) \\
 && = 1 - \exp\left(-\lambda c_2 \left(\frac{R}{\dist{x'_1}{x_2'}}\right)^{d-1}
 \left(\frac{1}{1-d} - \frac{R^{d-1}}{1-d}\right)\right) \\
&& = 1 - \exp\left(-\frac{\lambda c_2 R^{d-1}
\left(R^{d-1} - 1\right)}{(d-1)\dist{x'_1}{x'_2}^{d-1}} \right) \\
&& = 1 - \exp\left(-\frac{\lambda c_2 R^{2(d-1)}
\left(R^{d-1} - 1\right)}{(d-1)\dist{x_1}{x_2}^{d-1}} \right)
\end{eqnarray*}
which proves \eqref{eqn:E1E2bound}.

Consider now some arbitrary shift-invariant event $F.$
Define 
\[
\indicator_{F,x,R}:=\indicator\left(\omega \in \left\{
\BP(F| \left.\omega\right|_{\Gamma_{x,R}})>1/2\right\}\right),
\] 
and note that informally, we have that $\indicator_{F,x,R}(\omega)=1$ if $\omega$ is such that knowledge of $\omega |_{\Gamma_{x,R}}$, 
i.e. the configuration inside $\Gamma_{x,R}$, makes the occurrence of $F$ probable. The value $1/2$ is somewhat arbitrary.  

Note also that the event that $\indicator_{F,x,R}=1$ is determined by
$\Gamma_{x,R}.$   Then, by L\'evy's 0-1 law,
\[
\lim_{R \to \infty} \indicator_{F,o,R}=\indicator_F \textrm{ a.s. }
\]
Using that $F$ is invariant under isometries, it is straightforward to 
prove that the laws of $(\indicator_F,\indicator_{F,o,R})$
and $(\indicator_F,\indicator_{F,x,R})$ are the same for every $x$, and so 
$\indicator_{F,R^4 e_1,R}$ converges in probability to $\indicator_F.$
Thus, $\lim_{R \to \infty} \BP(I_{F,o,R}=I_{F,R^4 e_1,R}=I_F)=1 $ and so 
\begin{equation} \label{eqn:Fprob}
\lim_{R \to \infty} \BP(I_{F,o,R}=I_{F,R^4 e_1,R}=1)=\BP(F).
\end{equation}

By using \eqref{eqn:E1E2bound} we then see that 
\begin{eqnarray*}
\lefteqn{
\lim_{R \to \infty}\left|\BP(I_{F,o,R}=I_{F,R^4 e_1,R}=1)-\BP(I_{F,o,R}=1)\BP(I_{F,R^4 e_1,R}=1)\right|}\\
& & \leq \lim_{R \to \infty} 4\left(1 - \exp\left(-c\frac{R^{2(d-1)}
\left(R^{d-1} - 1\right)}{R^{4(d-1)}} \right)\right)=0.
\end{eqnarray*}
We therefore conclude that 
\[
\lim_{R \to \infty}\BP(I_{F,o,R}=I_{F,R^4 e_1,R}=1)
=\lim_{R \to \infty}\BP(I_{F,o,R}=1)\BP(I_{F,R^4 e_1,R}=1)=\BP(F)^2,
\]
and by comparing this to \eqref{eqn:Fprob} we conclude that 
$\BP(F)\in\{0,1\}$.
\end{proof}
\end{appendices}
\bibliography{references} 
\end{document}